\documentclass[11pt]{amsart}
\usepackage{epsfig}
\usepackage{graphicx, subfigure}
\usepackage{amsmath, amssymb,mathabx,mathrsfs}
\usepackage{color}
\usepackage[toc,page]{appendix}
\usepackage{comment}
\usepackage{hyperref}
\usepackage{tikz}
\usepackage{orcidlink}

\newtheorem{thm}{Theorem}[section]
\newtheorem{cor}[thm]{Corollary}
\newtheorem{lem}[thm]{Lemma}
\newtheorem{prop}[thm]{Proposition}
\newtheorem{ex}[thm]{Example}
\theoremstyle{definition}
\newtheorem{defn}[thm]{Definition}
\theoremstyle{remark}
\newtheorem{rem}[thm]{Remark}
\newtheorem{conjecture}[thm]{Conjecture}
\numberwithin{equation}{section}

\newcommand{\R}{\mathbb R}
\newcommand{\T}{\mathbb T}
\newcommand{\Z}{\mathbb Z}
\def\integer{{\mathbb Z}}
\newcommand{\N}{\mathbb N}
\newcommand{\eps}{\varepsilon}

\newcommand{\A}{\mathcal{A}}
\newcommand{\B}{\mathcal{B}}
\newcommand{\C}{\mathcal{C}}

\newcommand{\F}{\mathcal{F}}
\newcommand{\G}{\mathcal{G}}
\newcommand{\OO}{\mathcal{O}}

\newcommand{\U}{\mathcal{U}}

\def\torus{{\mathbb T}}
\def\real{{\mathbb R}}

\def\L{{\mathcal L}}

\def\Id{ {\rm Id}}

\def\cR{\mathcal{R}}

\begin{document}
\title[ ]
      {Geometric, topological and dynamical properties of conformally symplectic systems,
       normally hyperbolic invariant manifolds, and   scattering maps}
\author[M. Gidea]{Marian  Gidea  \orcidlink{0000-0002-4248-9890}}
\address{Yeshiva University, Department of Mathematical Sciences, New York, NY 10016, USA }
\email{Marian.Gidea@yu.edu}
\thanks{Research of M.G. was partially supported by the NSF grant  DMS-2307718. }
\author[R. de la Llave]{Rafael de la Llave \orcidlink{0000-0002-0286-6233}}
\address{School of Mathematics (Emeritus), Georgia Institute of Technology, Atlanta, GA 30332, USA}
\address{Yeshiva University, Department of Mathematical Sciences (affiliate), New York, NY 10016, USA }
\email{rafael.delallave@math.gatech.edu}
\author[T. M-Seara]{Tere M-Seara \orcidlink{0000-0001-8421-8717}}
\address{Departament de Matem\`{a}tiques, Universitat Polit\`{e}cnica de Catalunya, Diagonal 647, 08028 Barcelona, Spain, Centre de Recerca Matem\`{a}tica.}
\email{Tere.M-Seara@upc.edu}
\thanks{Research of T.M-Seara is supported by  the grants PID-2021-122954NB-100 and  PID2024-158570NB-I00, both
financed by  MCIN/AEI/ 10.13039/501100011033/ and by “ERDF A way of making Europe” and by
 the Spanish State Research Agency, through the Severo Ochoa and María de Maeztu Program for Centers and Units of Excellence in R\&D (CEX2020-
001084-M)}

\begin{abstract}

Conformally symplectic  diffeomorphisms $f:M \rightarrow M$ transform a symplectic form $\omega$ on a  manifold $M$ into a multiple of itself, $f^* \omega = \eta \omega$. They appear  naturally in applications. We assume $\omega$ is bounded, as some of the results in this paper may fail otherwise.

We show that there are deep interactions between the topological properties of the manifold, the dynamical properties of the map, and the geometry of invariant manifolds. When $\eta \ne 1$  the manifold $M$  cannot be closed, as the volume grows or decays under  iteration.

We show   that,  when the symplectic form is not exact, the possible conformal factors $\eta$ are related to topological properties of the  manifold. For some manifolds the conformal factors are restricted to be algebraic numbers.

We also find relations between dynamical properties (relations between  growth rate of vectors and $\eta$) and symplectic properties (whether $\omega$ vanishes or is non-degenerate on certain subspaces).

Normally hyperbolic invariant manifolds (NHIMs) and their (un)stable manifolds are important landmarks that organize long-term dynamical behaviour.
We prove that a NHIM is symplectic if and only if the rates satisfy certain pairing rules and if and only if the rates and the conformal
factor satisfy certain (natural) inequalities.

Homoclinic excursions to NHIMs, which are crucial for long-term dynamics -- particularly for Arnold diffusion -- are quantitatively described by scattering maps. These maps give the trajectory asymptotic in the future as a function of the trajectory asymptotic in the past. We prove that the scattering maps are symplectic even if the dynamics is dissipative. We also show that if the symplectic form is exact, then the scattering maps are exact,
even if the dynamics is not exact. We  give a variational interpretation of scattering maps in  the conformally symplectic setting.

We also show that similar properties of NHIMs and scattering maps hold in the case when $\omega$ is presymplectic.
In dynamical systems with many rates (e.g., quasi-integrable systems near multiple resonances), pre-symplectic geometries appear naturally.

\end{abstract}
\maketitle

\tableofcontents

\section{Introduction}\label{sec:introduction}

\subsection{Overview}

We study  interactions between topology, geometry, and dynamics in the context of conformally symplectic maps and their invariant manifolds.

\medskip

Given a symplectic
manifold $(M, \omega)$ of dimension  $d$, a diffeomorphism on $M$ is conformally
symplectic if there exists a conformal factor $\eta >0 $ such that
\begin{equation*}
\label{first-cs}
f^* \omega = \eta \omega .
\end{equation*}

Conformally symplectic maps and their invariant manifolds
present intrinsic mathematical  interest, e.g.,
\cite{AllaisA24,Banyaga02,Lee46,Vaisman75,Vaisman85}.
They also arise in  applications, such as mechanical systems
with friction proportional to velocity, e.g., \cite{CallejaCL22, MargheriOR14, BiascoCh09}, as Euler-Lagrange
equations of exponentially discounted
systems  -- which appear in financial models and control
theory -- e.g., \cite{Bensoussan88}, or  thermostat  systems, e.g.  \cite{wojtkowski1998conformally}

\medskip

The goal of this paper is to study the interactions between the topology of the manifold $M$, the dynamical properties of the map $f$, and the geometric structures that organize the dynamics.

The case $\eta = 1$  corresponds to symplectic maps.
The case $\eta \ne 1$ has some similarities with the symplectic case
but also important differences\footnote{
The limit $\eta \to 1$  is a singular limit, as the properties of the  system change dramatically.}.
Notably, the symplectic volume changes under iteration by a factor, $f^* \left(\omega^{ d/2}\right ) = \eta^{d/2} \left( \omega^{ d/2}\right)$.
So, the only invariant objects are of volume $0$ or $\infty$.
The volume $0$ case includes interesting objects such as
Birkhoff attractors \cite{arnaud2024higherdimensionalbirkhoffattractors},
however,  in this paper  we will concentrate on  invariant manifolds of infinite volume.
In the infinite-volume setting,   uniform boundedness properties of differentiable objects are not straightforward, and  they must be  carefully formulated and handled.

We will assume  that $\|\omega\|$ is bounded (a non-trivial
assumption in non-compact manifolds)
Indeed, we will show that some  of the  results in this paper  fail
for unbounded $\omega$ (see example \ref{ex:unbounded_symplectic_form}).

On the other hand, we do not need that the non-degeneracy properties of
$\omega$ are uniform (that is, we can allow that $|\omega^{ d/2}|$
is not equivalent to the Riemannian volume). Many of the results of the main Theorems \ref{thm:main1} and \ref{thm:main2}
(not the pairing rules obtained in Theorem \ref{thm:main1})  work even when $\omega$ is a presymplectic form (see Theorem \ref{mainpre}).

Many of the remarkable properties of symplectic dynamics extend, often with suitable modifications,
 to the conformally symplectic setting.  Examples include KAM
 \cite{CallejaCL13, CallejaCL22}, and, under convexity assumptions,
 Aubry-Mather theory
 \cite{maro2017aubry}   and Hamilton-Jacobi theory \cite{Gomes08}.
 On the other hand, new phenomena such as attractors appear for conformally
 symplectic systems. In conformally symplectic (but not symplectic)
 systems it is impossible
 to  have invariant manifolds finite non-zero volume.

 \bigskip
 A large part of this paper is devoted to the study of Normally Hyperbolic Invariant Manifolds (NHIMs) and their stable/unstable manifolds,
 which -- together with KAM theory -- are among the principal sources of invariant objects in symplectic dynamics.
 We study the properties of NHIMs, their stable and unstable manifolds and their homoclinic intersections,  in the conformally symplectic setting.

\medskip

 In this overview, we will present informally the main ideas and results in
 this paper. Precise formulations will require preliminary definitions
 undertaken in Section~\ref{sec:preliminaries}.

\subsection{Summary of the results}
We provide a road map for the
results in this paper.

 \subsubsection{Topology and conformally symplectic dynamics}

 \begin{itemize}
 \item We examine the role of exactness in conformally symplectic dynamics.

   In symplectic geometry, it makes a significant
   difference  when the symplectic form is exact, $\omega = d \alpha$,
   and the map preserves the action form $\alpha$  up to
   an exact differential. In conformally symplectic systems the analogue  is:
   \[
   f^* \alpha = \eta \alpha + dP^f,
   \]
   for some \emph{primitive function} $P^f$.

   \smallskip

    Notice that, even if we assume that $\omega$ is bounded, it is not natural -- and
   often impossible -- to assume $\alpha$ is bounded.  See Section~\ref{sec:unbounded_action}.
   In the conformally symplectic case, it happens often that
   a map is exact for some action forms but not for others (see
   Section~\ref{ex:standard_map}).

   This has relations with the de Rham cohomology of the manifold.
   For two
   action forms $\alpha,\tilde\alpha$ for $\omega$, we have $d(\alpha - \tilde \alpha) = 0$ but
   not necessarily $\tilde \alpha -  \alpha  = dG$.
   For any action form $\alpha$, $\alpha + dG$ is an action form too,
   and if $f$ is exact for $\alpha$, then $f$ is also  exact for $\tilde \alpha$.
   We call the addition of $dG$ to an action form a
   \emph{gauge transformation}.
   We  study systematically the
   effect of gauge transformations  on primitive functions.

 \item
   In some manifolds, there are relations between the algebraic topology
   and the conformal factors $\eta$. Consequently, on these manifolds,
   the possible values of $\eta$  for non-exact $\omega$
   are algebraic numbers.
   See Section~\ref{sec:ArnaudF}.
   Using these relations, in Section~\ref{sec:ArnaudF},
   we present an answer to a question raised in \cite{ArnaudF24}.
 \end{itemize}

\subsubsection{Vanishing lemmas}
The interaction between dynamics and geometry arises when we consider rates of growth of vectors
 under iteration of the map $f$.

The following is a description of the results in  Section~\ref{sec:vanishing_lemmas}:
\begin{itemize}
 \item
 Since
 \[
 \omega (x)(u,v) = \eta^{-n} \omega (f^n(x)) ( Df^n(x) u, Df^n(x)v),
 \]
 we see that
 \[
 |\omega (x) (u,v)|  \le \eta^{-n} \|\omega\| \| Df^n(x)u\| \|Df^n(x)v\|.
 \]
From this, we can derive vanishing lemmas stating that
$\omega$ vanishes on subbundles of tangent vectors whose growth rates satisfy certain relations with respect to $\eta$.

\item
   Conversely, if the restriction of the form is non-degenerate, the rates in different subbundles must be related in a way that prevents the vanishing of $\omega$.

   These allows us to prove  global versions  for NHIM of the
   widely studied \emph{pairing rules} \cite{Dessler88,dettmann1996proof,
     wojtkowski1998conformally} for periodic orbits of Lyapunov
   exponents.
 \item
   One consequence of the vanishing lemmas is that, in systems with many rates (e.g., quasi-integrable systems near multiple resonances), presymplectic geometry naturally arises (i.e., $\omega$ is closed but may be degenerate).
 \end{itemize}

\subsubsection{ NHIMs for conformally symplectic systems}

   We recall that a NHIM is an
   invariant manifold  such that there are gaps between the rates of growth
   of vectors tangent to the manifold  and the rates of growth in
     the stable and unstable bundles that span the normal
    bundle.  See \cite{Fenichel71,Fenichel74,Fenichel77,Pesin,BatesLZ08}, and also Appendix~\ref{sec:appendix_NHIM}.

  NHIMs   enjoy regularity properties,
   are persistent under perturbations and, more importantly for us,
   they have (un)stable manifolds
   foliated by strong (un)stable manifolds.

   The fact that the  invariant manifolds for conformally symplectic maps cannot be compact creates some technical
   subtleties in the  analysis of these objects. The results in this paper
   depend only on a few regularity properties that we have identified explicitly
   {\bf (H1)-(H4)}. Some details on the theory  of NHIMs  that puts these properties in
   a broader context are in Appendix~\ref{sec:appendix_NHIM}.

   It should be noted that the unboundedness of manifolds is not merely a technical inconvenience;
   it can give rise to new geometric phenomena, such as an NHIM folding into itself.
   See \cite[Example~3.8]{Eldering12},
   reproduced here in Figure~\ref{fig:manifold_no_unif_tubular}.
   To avoid such pathological examples, we have introduced an explicit assumption {\bf (U2)}
   that the NHIM has a uniform tubular neighborhood.

     \begin{itemize}
     \item
     The main result on NHIM is the following:

     \medskip

     \emph{A NHIM is symplectic if and only if either of the two holds:}
     \begin{itemize}
     \item
     the rates of vectors in the tangent space satisfy some
     pairing rules and the rates along the stable and unstable manifold
     also satisfy other pairing rules. See \eqref{eqn:pairing_rules}.
\item
The rates
 and the conformal factor satisfy certain (natural) inequalities (see \eqref{eqn:conformal_rates})
\end{itemize}
     \medskip

     See Theorem~\ref{thm:main1} and Corollary \ref{partial_converse}.
\medskip
   \item
     A
     consequence of the  NHIM being symplectic is that
     the (un)stable manifold is co-isotropic (hence presymplectic),
     and the kernel of
     $\omega_{\mid W^s_\Lambda}$ integrates to give the strong (un)stable foliation.

     This relation between
     dynamics and  presymplectic geometry  turns out to be crucial for  other subsequent results.
\end{itemize}

\medskip
     The  proof of the results in this paragraph  show that the non-degeneracy of the symplectic form on the NHIM implies pairing rules on the rates, and
     that   pairing rules on the rates imply vanishing lemmas for the
     strong stable and unstable foliations. Interestingly, starting from geometric assumptions, we obtain relations between rates that, in turn, yield new geometric conclusions.

     We also obtain results that  start with relations between rates  and pass through geometric structures to produce new relations between rates.
     See Section~\ref{sec:isotropic}.

\medskip

\subsubsection{Scattering maps for conformally symplectic systems}

It is well known that homoclinic excursions resulting from intersections of stable and unstable manifolds of a NHIM give rise to large-scale motions, symbolic dynamics, and other phenomena.
One of our main goals is to understand the symplectic geometry associated to these homoclinic intersections.

   An important tool in the study of homoclinic excursions to a NHIM  is the \emph{scattering map},
   introduced in \cite{DLS00}.

   We consider a class of homoclinic excursions given by a transversal
   intersection of stable/unstable manifolds (the precise conditions
   are given in Definition \ref{def:channel} in Section~\ref{sec:SM}).
   Given a homoclinic excursion $\{f^n(x)\}_{n \in \mathbb{Z}}$
   to a NHIM $\Lambda$, there are unique points $x_+, x_- \in \Lambda$ such
   that
   \[
   d( f^n(x), f^n( x_\pm) ) \le C e^{-\lambda |n|} \quad \textrm{as} \quad n \to \pm \infty,
   \]
   where $\lambda$ is strictly bigger than the rates in
   the tangent space to the NHIM.

  The scattering map $S$ is defined by $S(x_-) = x_+$.
  As  the homoclinic orbit moves along a transversal intersection,
  the scattering map is  defined in an open set.
  See Figure~\ref{fig:scattering_map} for a pictorial
  representation.

     The scattering map is
     a very useful tool to understand long range motions and instability.
     Most of the applications of the scattering map in the instability problem
     have been for symplectic dynamics,  since it gives
     a global way to connect invariant objects such
     as whiskered tori of (possibly) different topology or dimension
     \cite{DelshamsLS06a,zbMATH06567884}.
     There are also results using the scattering map in
     systems with dissipation \cite{Granados17,GideaLM22,AkingbadeGS23}.

     The scattering map is not itself part of the dynamics; rather, it should be regarded as a comparison between the dynamics restricted to
     the NHIM and the dynamics along the homoclinic excursion.  Nevertheless,
     a surprising result
     \cite{GideaLS20a, GideaLS20} shows that iterations of the  scattering map
     interspersed with  iterations of the  dynamics restricted to the NHIM corresponds to
     true orbits of the map. The basic results of the previous two papers apply to general systems, including conformally symplectic systems.
     See also \cite{DelshamsGR12}.

     Geometric properties of the scattering map in the symplectic case
     were systematically studied in \cite{DLS08}. In this
     paper we  develop  similar results
     for conformally symplectic systems, but we encounter significant differences.

\begin{itemize}

   \item
     The anchor result in this paper is:

\medskip

     \emph{The scattering map  for conformally symplectic maps is symplectic.}

\medskip

      See Theorem~\ref{thm:main2} (A version of this result in the presymplectic case is Theorem~\ref{mainpre}).

     This result is surprising because the dynamics itself could be dissipative;
     nevertheless, the fact is that the scattering map is a comparison
     among different dynamics leads to the cancellation  of the effects of the dissipation,  which are of geometric nature.

     We present seven different proofs of the symplecticity of
     the scattering map.

     The variety of arguments
     allows to extend the result to other models
     that have appeared in
     the literature. See Appendix~\ref{othermodels}.

     \item
       Another result owed to cancellations is:

     If the symplectic form is exact, $\omega = d \alpha$, then
     the scattering map is exact.

     Notice that this result does not use that the map is exact. Again, this shows
     that the scattering map may enjoy properties that the original dynamics does not have.

     We  present two different proofs. They depend crucially on
     the presymplectic nature of $W^s_\Lambda$.

   \item
     If moreover, the map $f$ is exact, we derive formulas for the primitive
     function of the scattering map. See Section~\ref{sec:primitive_series}.

     The formulas consist of a finite series plus an integral term as the remainder.
     Similar formulas were known for symplectic twist maps, and  used
     for numerical computations \cite{Tabacman95}.

     The theory behind these formulas, in the case of unbounded manifolds, presents some notable differences from that of bounded manifolds, which are commonly considered in symplectic systems. As we pointed out,
     there are different action forms corresponding to gauge transformations.
    The gauges chosen can affect the convergence of the series.

    When either the orbit asymptotic in the future or the orbit
    asymptotic in the past
     escape to $\infty$ (i.e., the orbit leaves any compact set after a finite
     number of steps), we show that there are (rather explicit)
     gauges where the series converges, and  there are also
     gauges that make  the series  divergent.

   \item
     We provide a variational interpretation of the primitive function of the scattering map.
     We remark that, under certain convexity assumptions and provided the Legendre transform can be implemented, the primitive function of the scattering map coincides with the renormalized variational principle used to construct connecting orbits. \cite{Rabinowitz08,Bessi96}.
\end{itemize}

   \subsection{Organization of the paper}

   In Section~\ref{sec:preliminaries},
   we give careful presentations of several standard concepts.
   We expect that most readers will be familiar with some versions of these concepts.
   However, since we cross category boundaries and deal with technicalities such as finite differentiability on unbounded manifolds, certain subtleties must be explicitly addressed. This section could be skipped and used as a reference.

   The precise formulation of the results are in Sections~\ref{sec:mainresults} and ~\ref{sec:topological}.

   Section~\ref{sec:examples} provides several examples that illustrate the phenomena and serve as motivation for the results. It also contain examples which show that some hypotheses are necessary in the main theorems of sections ~\ref{sec:mainresults} and ~\ref{sec:topological}.
    The examples do not enter into the formulation of results or the proofs.

   The proofs of the main results
   are contained in the following sections. Notably, Section~\ref{sec:vanishing_lemmas}
 includes statements and proofs of vanishing lemmas which
 will be a basic tool for other results.
 The proofs of the main results Theorem~\ref{thm:main1} and
 Theorem~\ref{thm:main2} are in Section~\ref{sec:proofAmain1total}
 and in Section~\ref{sec:proof_main_2}.

Section~\ref{sec:primitive_series} studies the
primitive function of the scattering map and provides formulas for it when the map $f$ is exact conformally symplectic.

The short Section
\ref{sec:proofs_presymplectic},  which provides the proof of theorem \ref{mainpre}, checks that   the proofs for some of the results for conformally symplectic maps   also apply to
presymplectic  maps.

In Appendix~\ref{sec:appendix_NHIM}, we  collect without proofs
some of the results on invariant objects for  NHIMs.
In Appendix ~\ref{sec:NHIM_results}, we   formulate  and prove   several
results on hyperbolic bundles. Two are  elementary results,
Lemma~\ref{lem:rates_inverses}, \ref{rem:ratesplusminus},
about forward and backward rates. Lemma~\ref{lem:globalrates}
is a delicate perturbative
result about rates of cocycles.
In Appendix~\ref{sec:fiber_contraction}, we present a detailed
proof with explicit constants of a result
sometimes called the  Fiber Contraction Theorem or the Inclination Lemma.
Some versions of this result are known in
the  theory of NHIMs (without
explicit constants, or with compactness assumptions).
In Appendix~\ref{othermodels} we describe other relevant models that appeared in
the recent literature, and sketch how several of the results here can be adapted.
We hope that  new results can be obtained for these and other
models.

\section{Preliminaries}
\label{sec:preliminaries}
In this section, we recall some standard definitions and tools to set the
notation.
This section  can be skipped and referred to as needed.  This paper involves
connections between different categories, so some precise definitions are
needed.

\subsection{Some standard notions in differential geometry}
\label{sec:prelimgeom}

\subsubsection{Riemannian manifolds, differentiable maps and forms}

Throughout this paper, we assume that $M$ is  a $d$-dimensional, orientable,  connected,
$r$ times differentiable Riemannian manifold, with $r\ge r_0$ for some $r_0$ sufficiently large.
We also assume that the Riemannian metric is $r$ times differentiable with
the derivatives uniformly bounded (it may be more convenient
assuming  that the metric is $r+1$ times differentiable, see
{\bf (U1')}).

Since the goal is dealing with conformally symplectic systems,
this will require that the manifold $M$ is unbounded,
hence, we will need to make explicit uniformity assumptions for the size
and regularity of objects.
Later, in Section~\ref{sec:NHIM},  when we consider normally hyperbolic invariant manifolds (NHIMs)  contained in $M$, we will need to deal with finite differentiability. (Even if the system is infinitely differentiable or analytic,
 the NHIM could be  finitely differentiable even if it is compact.
See Example \ref{regularitylimited}. The regularity of the NHIM
and several other objects associated to it is limited by  formulas
involving the hyperbolicity rates).

\begin{defn}\label{def:Cr}
Given an open set $\mathcal{O} \subset M$,
we say that map  $f:\mathcal{O}\to M$ is of class $\C^r$ if it has
\emph{continuous and uniformly bounded derivatives} of order $1, \ldots, r$.

Similarly, we say that a form, or a vector field, or other
object    is $\C^r$ if
the object and its derivatives up to order $r$ are
\emph{continuous and uniformly bounded}.
\end{defn}

The Definition~\ref{def:Cr} is different from other notions
of $C^r$-differentiable maps on non-compact manifolds.
For example, it is different from differentiability in the sense of
the Whitney. In this paper, Definition~\ref{def:Cr}
is the only notion of differentiability used.

The sets of $\C^r$  vector fields,  forms, and other objects with a linear structure, have a $\C^r$ norm given by the supremum
of the size of the objects and their derivatives, which makes them into
a Banach space.
Under suitable uniformity assumptions (see \eqref{U1}), the set of $\C^r$ diffeomorphisms is a Banach manifold \cite{Banyaga97}.

\subsubsection{Distances and regularity of   manifolds}
A  Riemannian manifold $M$ is said to be bounded if it has finite diameter, i.e., $\sup_{x,y\in M} d(x,y)<\infty$, where $d$ is the Riemannian distance function.
As mentioned earlier, in this paper we will consider unbounded manifolds.

The following definitions are standard, but we record them
here since there are subtle. See also
\cite{PughSW97}.

\begin{defn}\label{distance}
  We define the $\C^0$  distance among two submanifolds  of a common
  Riemannian manifold,   $N_1, N_2 \subset M$, as the Hausdorff distance:
  \[
  d_{\C^0}( N_1, N_2) =  \sup_{x_1 \in N_1} \inf_{x_2 \in N_2} d(x_1, x_2)
  +\sup_{x_2 \in N_2} \inf_{x_1 \in N_1} d(x_1, x_2) .
  \]

  For $\C^1$ manifolds, we define:
  $d_{\C^1}( N_1, N_2) = d_{C^0} (T N_1, T N_2) $
  where $TN$ is the tangent bundle of $N$.
  We recall that, if $M$ is a Riemannian manifold,
  we can define a natural metric in $TM$.

    Analogously, we define higher differentiable distances
    by
    $d_{\C^r}(N_1, N_2) = d_{C^{r-1}}( TN_1, TN_2)$.

    We also define the $\C^r$ distance among maps
    as the  $\C^r$ distance among their graphs.  For
    diffeomorphisms, the distance among diffeomorphisms
    is the maximum of the distance between the maps
    themselves and the distance among the inverses.
\end{defn}

In this paper, we will need to study some foliations of
a manifold  by submanifolds.
  Foliations have regularity characterized by two
  numbers. One of them is the regularity of the leaves
  and another one is the dependence of the  leaves on
  the base point along a  transversal.   We use the definition of \cite{HPS,PughSW97}.

  \begin{defn} \label{Ctildemm}
    We say that a foliation is of class $\C^{\tilde m, m}$
    when the leaves are uniformly $\C^m$ manifolds
    and the dependence of the leaves in the $\C^m$ topology on the base point   is locally $\C^{\tilde m}$.
  \end{defn}

  Some formulations in terms of coordinates
  appear in \eqref{eqn:mixed partials reg 1}
  and \eqref{eqn:mixed partials reg m}.

\subsubsection{Exponentials and connectors}\label{sec:metrics}

We think of the geodesic flow as a second order equation on
 the manifold $M$.

Given $x \in M$, $v \in T_xM$, the exponential mapping
$\exp_x(v)$ is defined as the point in the manifold which
is  the solution
at time $1$  of the geodesic flow with initial point $x$ and
initial velocity $v$. So, we can think of the exponential
mapping at $x$ as a mapping from a neighborhood of $0 \in T_xM$
to a neighborhood of $x$ in $M$,

It is well known \cite[Prop 4.5.2]{BurnsG05}
that on a ball $B_{\rho_x}(0)\subseteq T_xM$  of radius $0 < \rho_x \ll 1$,
the exponential mapping is a diffeomorphism and induces a system of smooth coordinates on the neighborhood $\exp_x(B_{\rho_x}(0))$ of $x$ in $M$,
referred to as geodesic coordinates.
For any $x \in M$, the supremum of all such radii  $\rho_x$ is called the injectivity radius.

If the metric is $\C^2$, it follows
$\exp_x(B_{\rho_x/2}(0))$  contains a ball in $M$
centered in $x$ and with radius bounded from below.

If $N \subset M$ is  a $\C^2$ submanifold of $M$, we can restrict the metric
of $M$ to $N$ and use the exponential mapping in $N$ using the
geodesic flow in $N$. The
geodesic flow  for the metric on $M$ with initial velocity in $TN$,
in general, does not map to
$N$.

The exponential mapping can also be  used to identify
neighboring tangent spaces.
If $\exp_x(v) = y$, consider the linear map
\begin{equation}
\label{eqn:connectors}
S_x^y:=D \exp_x(v): T_x M \to T_y M .
\end{equation}
If $x,y$ are sufficiently close (depending only on
derivatives of the metric) then $S_x^y$  is invertible,
and we think of it as providing an identification of the two
tangent spaces.  The maps $S_x^y$ are called
\emph{connectors} in \cite{HirschPPS69}, where they are used  in applications to hyperbolic systems.

\subsubsection{Pullback operator and cells}

We  recall  the following standard definition of pullback.

If $f:M\to M$ is a $\C^r$-diffeomorphism and  $\beta$ is a $\C^r$-differentiable  $k$-form on $M$,
the pullback of $\beta$ is given  by
\begin{equation}\label{eq:pullback}
(f^*\beta)(z)(v_1,\ldots,v_k)=\beta(f(z))(Df_z(v_1),\ldots,Df_z(v_k))
\end{equation}
for  $z \in M $ and $v_1,\ldots,v_k\in T_zM$.

A $k$-cell $\sigma$ on $M$ is the image  of $[0,1]^k$ through an orientation preserving,  $\C^1$-embedding $\sigma: [0,1]^k \to M$.
(We use the same notation for the mapping  and its image. )
A  choice of orientation on  $[0,1]^k$ induces an orientation on the $k$-cell  via $\sigma$.

We will often use an equivalent definition of the pullback of $\beta$ in integral form:
\[
\int_\sigma f^* \beta = \int_{f(\sigma)}\beta, \quad  \forall \  k-\text{cell} \ \sigma.
\]

\subsection{A standing assumption on the manifold}

Throughout the paper we will assume:

 \begin{flalign}\label{U1}\tag{\bf{U1}}
&\textrm{On the manifold $M$, there exists  a uniform $\C^r$ system of}&\\ \nonumber &\textrm{coordinates.}&
\end{flalign}

That is, there exist $\rho>0$ such that for every $x \in M$,  there is a $\C^r$-coordinate system on the  ball $B_\rho(x)\subseteq M$.
The system of coordinates
is uniform in the sense that the coordinate map that  takes the balls $B_\rho(x)\subset M$
onto the ball $B_1(0)\subset \real^d$, as well as the inverse coordinate maps, have  $\C^r$ norms that are bounded uniformly.

\begin{rem}
Assumption \eqref{U1}, under the previous hypothesis that the metric is uniformly $\C^r$ with bounded derivatives, implies that the Riemannian metric on $M$ is complete.

Indeed, let $\gamma$ be an arbitrary geodesic with unit speed.
At any time $t$,  the geodesic has  initial condition $(\gamma(t), \dot\gamma(t))$,  where$\|\dot\gamma(t)\| =1$. Since there is a ball of radius independent
of the point inside of the manifold, the geodesic can be prolonged
by an amount of time independent of the point.
\end{rem}

\smallskip

Using the exponential mapping we can give a concrete
construction of the coordinate systems assumed  in
the standing hypothesis \eqref{U1}.

The following assumption implies \eqref{U1}:
\begin{flalign}
& \textrm{Assume} &
\label{U1'}{\tag {\bf U1'}}
\end{flalign}
\begin{itemize}
\item
The metric $g$ on the manifold $M$  is $\C^{r+1}$;
\item
The metric $g$ on the manifold $M$ has an injectivity
 radius bounded from below away from $0$.
\end{itemize}

\begin{rem}
The fact that \eqref{U1'}
implies \eqref{U1} is standard.
  Since the injectivity radius is bounded
  away from zero implies that the size of
  the neighborhoods covered by the exponential mapping is
  bounded uniformly away from zero.
  The assumption on the metric $g$ to be $\C^{r+1}$ is
  made to obtain that the system of geodesic coordinates is $\C^r$.

Hypothesis {\bf (U1')} is sufficient but not
necessary to obtain {\bf (U1)}.
To obtain that the geodesic
flow (and the exponential mapping) are $\C^r$,
one does not need to assume that all the derivatives of $g$ are bounded.
It suffices to assume only  that the derivatives of the Riemann tensor are bounded.
This property of  a manifold is referred to as having \emph{bounded geometry}
\cite{Cheeger1982finite, Eldering12}.
\end{rem}

\subsection{Symplectic and  presymplectic forms}
\label{sec:symplectic_presymplectic}

\begin{defn}\label{defn:symplectic}
A $2$-form on   a $\C^r$-manifold $M$ of even dimension $d$
\[
\omega: T\, M\times  TM \to \R
\]
is a symplectic form if it is closed, i.e., $d\omega = 0$,
and non-degenerate, i.e., $\iota_v\omega=\omega(v,\cdot)= 0 \implies v = 0$.
\end{defn}

The form $\omega$ is exact if $\omega=d\alpha$ for some
$1$-form\footnote{We refer to $\alpha$ as an action form for $\omega$
following
\cite{Haro00}, but other names are used in the literature
such as \emph{symplectic potential}, \emph{Liouville form}, or \emph{primitive form}.}
$\alpha$ (not necessarily unique).

Throughout the paper, for  a submanifold  $N \subset M$ we denote the form
$\omega _{\mid N}$ as acting in $TN$, that is:
\[
\omega_{\mid N}: TN\times TN \to \R.
\]

\begin{defn}\label{defn:isotropic }
A submanifold $L\subset M$ is said to be  isotropic if $\omega_{\mid L} = 0$,
that is, for each $y\in L$, $T_yL\subseteq T_yL^\omega$,
where
$T_yL^\omega=\{v\in T_yM\,|\, \omega(y)(v,u)=0,\, \forall u\in T_yL\}$ is the symplectic orthogonal of $T_yL$.
A submanifold $L\subset M$ is said to be  coisotropic if  for each $y\in L$, $T_yL^\omega\subseteq T_yL$.
A  submanifold $L \subset M$  is Lagrangian if it is both isotropic and coisotropic, or, equivalently,
 $\omega_{\mid L} = 0$ and
$\textrm{dim}(L) = \textrm{dim}(M)/2$. (See, e.g., \cite{Weinstein71,Weinstein73}.)
\end{defn}

We will also pay attention to forms
that are degenerate and have constant rank (dimension of the kernel)
\cite{Souriau97}.

\begin{defn}\label{defn:presymplectic}
  A $2$-form  $\omega$ on a $\C^r$-manifold $M$ (not necessarily even dimensional)  is presymplectic when
   $d\omega  = 0$ (but $\omega$ may be degenerate).
\end{defn}

The degeneracy of the form can be characterized by its
kernel
\[
K_x(\omega)  = \{v \in T_x M \,| \,  \iota_v(\omega) =0\}.
\]
The form $\omega$ is non-degenerate at $x$ iff $K_x(\omega)  = \{0\}$, in which case the manifold must be even dimensional.

In some works \cite{Souriau97,LibermannM87},  the definition of presymplectic forms  also includes the constant rank condition that $\dim K_x(\omega)=$ const. on open sets.
In general,  the constant rank is not an open condition, since a  small perturbation of the form may decrease the dimension of the kernel.

\begin{rem}
In this paper,  presymplectic forms appear as restrictions of
symplectic forms to invariant manifolds with some rates.
In this case, the constant rank of the kernel
is a consequence of the rate conditions.
Since the kernels are related to the rates, they appear naturally
and are stable under perturbations.
In the opposite direction,  the presymplectic structure
shows that the distributions of vectors having these rates
integrate to a foliation, which is not expected for general distributions.
\end{rem}

\subsection{Conformally symplectic maps}\label{sec:conf_symp}

\begin{defn}\label{def:conformally_symplectic}
A $\C^r$-map  $f: M \to M$
is   conformally symplectic if
\begin{equation}\label{eqn:conformally_symplectic}
  (f^*\omega)(x)=\eta\cdot\omega(x),\textrm { for all } x\in M,
\end{equation}
for some number  $\eta>0$.

The definition of a  conformally symplectic map between two different symplectic manifolds is analogous.
\end{defn}

We will refer to $\eta$ as the conformal factor.

The condition~\eqref{eqn:conformally_symplectic} is equivalent to:
\[
\int_\sigma f^* \omega
= \eta \int_{\sigma} \omega,\,\forall \textrm{ 2-cell } \sigma.
\]

\begin{rem}\label{rem:conformal_factor_constant}
  It is well known(\cite{LibermannM87,Banyaga02}) that
  if $\omega$ is a two form of rank greater or equal  than $4$,
  then, if, for a $1$-form $\beta$,
  we have $\beta \wedge \omega = 0$, we
  conclude $\beta = 0$.

  A consequence is that,
  if ${\rm dim}(M) \ge 4$,  $\omega$ is symplectic,
  $f^* \omega = \eta \omega$ for some function $\eta$,
  implies that  $\eta$ is a constant.

  Applying $d$ to \eqref{eqn:conformally_symplectic},  $0=f^*(d\omega)=d(f^*\omega)=d\eta\wedge \omega$.
  Hence $d\eta = 0$.
\end{rem}

When $\eta=1$ a system satisfying \eqref{eqn:conformally_symplectic} is symplectic, so the results in this paper (that do not include explicitly $\eta \ne 1$) imply the results in \cite{DLS08}.
When $\eta\neq 1$,  the symplectic  volume contracts or expands, and, therefore, invariant manifolds have volume zero or infinite.
The interesting case is when the symplectic  volume is infinite.
In many mechanical applications,
the physical friction satisfies
$\eta<1$, so we will take this as the default case.

\begin{rem}\label{rem:manifold_unbounded}
If $\|\omega\|$ is bounded (which we will assume in \eqref{eqn:bounded_omega}), the fact that a manifold $M$ has infinite symplectic  volume implies that it is unbounded.
 Indeed,  the symplectic volume and Riemannian volume satisfy
$ \int_\sigma\omega^{d/2} \le  C \|\omega\|^{d/2} \int_\sigma d\textrm{Vol}$
for any $d$-dimensional cell $\sigma \subseteq M$,  where $\textrm{Vol}$ is the Riemannian volume and $C>0$.
Therefore, the fact that the symplectic volume is infinite implies that so is  the Riemannian volume.
Then the diameter has to be infinite too.
\end{rem}

Dealing with unbounded manifolds, we will make explicit assumptions on the uniformity of the objects considered.
Particularly, in Section \ref{sec:uniformity}, we will assume that $\omega$ is uniformly bounded in $\C^0$, and
in sections \ref{proofnonclosed} and \ref{sec:iteration} we will assume that $\omega$ is uniformly bounded in $\C^1$.
We will also assume uniform bounds on the derivatives of the map, etc.

\subsection{Exact conformally symplectic maps}

If the symplectic form is exact,  i.e., $\omega=d\alpha$,   then \eqref{eqn:conformally_symplectic}
is equivalent to
\[
d(f^*\alpha-\eta\alpha)=0,
\]
Therefore,
\begin{equation*}\label{conf-symplectic}
 f^* \alpha = \eta \alpha + \beta \textrm{ for some $1$-form $\beta$ with } d\beta  = 0 .
\end{equation*}

In many cases of interest, there are topological
reasons why the symplectic form $\omega$  has  to be
exact; see \cite{ArnaudF24} and Section  \ref{sec:exactness}.

We say that a   map is exact conformally symplectic for the action form $\alpha$ when
the form $\beta$ above is not only closed but also exact. That is:
\begin{equation}\label{eqn:exact_conformally_symplectic}
  f^*\alpha=\eta\alpha+dP^f_\alpha
\end{equation}
for some function $P^f_\alpha:M\to\mathbb{R}$, called the primitive function of $f$.

If a conformally symplectic map has a primitive function, it has
practical consequences.
For instance, under twist conditions, a primitive function
leads  to generating functions and, therefore, variational principles for orbits.
The conformally symplectic systems
have variational principles very similar
to those of the symplectic systems but involving
\emph{discounts}. These discounted variational principles
have a direct interpretation  in finance.  See
Section~\ref{sec:variational}.

\begin{lem} \label{exact_composition}
  Let $(M_i, \omega_i = d \alpha_i)$, $i = 1,2,3$
  be exact symplectic manifolds:

  Let $g:M_1 \rightarrow M_2$,
  $f:M_2 \rightarrow M_3$
  be exact conformally symplectic with respect to
  the corresponding forms:
  \[
  \begin{split}
    & g^* \alpha_2 = \eta_g \alpha_1 + dP^g,\\
    & f^* \alpha_3 = \eta_f \alpha_2 + dP^f .\\
  \end{split}
  \]

  Then, $f\circ g$ is exact conformally symplectic and the primitive of $f\circ g$ is given by:
  \begin{equation}\label{eqn:primitivefg}
  P^{f\circ g}=\eta_f  P^g + P^f\circ g .
  \end{equation}
\end{lem}
\begin{proof}
The proof is an straightforward calculation.
  \begin{equation}\label{eqn:exact_composition}
  \begin{split}
    (f \circ g)^* \alpha_3  &= g^* f^* \alpha_3  = g^*( \eta_f \alpha_2  + d P^f  )=
    \eta_f  ( \eta_g  \alpha_1  +  dP^g ) + (d g)^* P^f \\
    &= \eta_f \cdot \eta_g  \alpha_1 + d( \eta_f  P^g + P^f\circ g) .
\end{split}
  \end{equation}
\end{proof}

The same arguments also show that the inverse of a conformally symplectic exact
map $f$ is also exact with primitive:
  \begin{equation}\label{eqn:primitiveinverse}
  P^{f^{-1}}=-\frac{1}{\eta_f}   P^f\circ f^{-1} .
  \end{equation}

\subsubsection{Gauge transformations}\label{sec:gauge}
The action form $\alpha$ for $\omega$ is non-unique and
the notion of exactness may depend on the form $\alpha$ considered.
See   Example~\ref{ex:standard_map}, which shows that a map may be exact
for some $\alpha$  but not for others.

In general, if a map
is exact for $\alpha$, it will also be exact under the \emph{gauge transformation} of the action form
\footnote{In electromagnetism, where the electromagnetic field
is an exact 2-form exterior derivative of an electromagnetic
1-form (called vector potential),  one often
uses adding gradient
to the vector potential \cite{Thirring, Zangwill}. The term ``gauge" was introduced in \cite{Weyl51} for electromagnetic vector potentials }
\begin{equation}\label{eqn:tilde_alpha}
 \tilde \alpha = \alpha + dG
\end{equation}
for any globally defined function $G$.
Gauge transformations
are changes
to the action form that do not change its cohomology class.

The  gauge transformation \eqref{eqn:tilde_alpha}
induces a change of the primitive function of $f$.
For the form $\tilde \alpha$ in \eqref{eqn:tilde_alpha}, we have
\begin{equation}\label{eqn:Pf_gauge}
P^{f}_{\tilde \alpha} = P^{f}_{\alpha} + G\circ f - \eta G.
\end{equation}

%

\begin{rem}
The theory of exactness and gauge transformations
for maps from one manifold to another has some differences
from the theory  for maps from a manifold to itself.
Since the  action forms in the domain and the range are different,
the theory of maps to different manifolds has more flexibility that
the theory of self-maps.

Let $(M_i,\omega_i)$, $i=1,2$, be symplectic manifolds and $f:M_1\to M_2$   be a conformally symplectic map, that is,
$f^*\omega_2 = \eta \omega_1$.

If $M_2$ is exact symplectic with $\omega_2=d\alpha_2$, then $d(f^*\alpha_2)=\eta\omega_1$, so $M_1$ is also exact symplectic for the action form $\alpha_1=\frac{1}{\eta}f^*\alpha_2$,  and $f$ is exact  with primitive $P_{\alpha_1,\alpha_2}^f=0$.

The condition that $f$ is exact for $\alpha_1$ and $\alpha_2$ is that $f^*\alpha_2=\eta\alpha_1+dP^f_{\alpha_1,\alpha_2}$.

If we change $\alpha_i$ to $\alpha_i + dG_i$, $i = 1,2$, we see that
 $P^{f}_{\alpha_1+dG_1,\alpha_2+dG_2} = P^{f}_{\alpha_1,\alpha_2} + G_2\circ f - \eta G_1$.
Hence, we can make any primitive to be zero, by gauge
transformations either  in the domain  or in  the range.

In the case that $M_2 \subset M_1$, it is natural to
consider that the forms in $M_2$ are restrictions of those in $M_1$
and that $G_2$ is the restriction of $G_1$ to $M_2$.
\end{rem}

As mentioned before, for  unbounded
manifolds, boundedness properties of maps and forms are important.
In most of the results of this paper, we assume  $\omega$ is bounded (see \eqref{eqn:bounded_omega}).
On the other hand,
we will not impose any boundedness assumption on $\alpha$ or on $G$.
In some  cases,  there are lower bounds for any action form in
terms of the Riemannian geometry of the manifold.
These lower bounds grow to infinity as the point goes to infinity.
See Section~\ref{sec:unbounded_action}, Lemma ~\ref{lem:lowerbound}.
The boundedness of $\omega$ happens in many cases of interest, but
boundedness of $\alpha$ may be impossible in some manifolds.

\subsection{Presymplectic maps}
\label{sec:presymp_maps}
Presymplectic maps appear naturally when we consider
NHIMs with rates that are incompatible with having
a symplectic structure (see
Part (A) of Theorem~\ref{thm:main1} and Section \ref{sec:isotropic}).

\begin{defn}
\label{defn:confpresymplectic}
Let $\omega$ be a presymplectic form (see Definition~\ref{defn:presymplectic})
on   $M$.
We say that a $\C^r$-map $f$ is conformally  presymplectic
if there exists a real valued function $\eta:M\to \mathbb{R}$  and constants $\eta_\pm >0$
such that
\begin{equation}
\label{confpresymplectic}
(f^* \omega)(x)   = \eta(x) \omega(x), \quad 0 < \eta_-\le \eta(x) \le \eta_+ < \infty , \textrm{ for all } x\in M.
\end{equation}
\end{defn}

While, in general, the conformal factor $\eta(x)$ of a conformally  presymplectic map is a function,  if the  kernel of the presymplectic form has codimension at least $4$, the   conformal factor  has to be a constant (see Proposition \ref{prop:conformal_factor_constant}), similarly to  the case of conformally symplectic maps (see Remark \ref{rem:conformal_factor_constant}).

\subsection{Normally hyperbolic invariant manifolds}
\label{sec:NHIM}

In this section, we recall the definition of a normally hyperbolic invariant manifold  (NHIM).

\begin{defn} \label{def:nhim}
Let $M$ be a manifold endowed with a smooth Riemannian metric and $f:M\to M$ be a $\C^1$-diffeomorphism.

Let $\Lambda\subset M$ be a unbounded, boundaryless,
and connected submanifold invariant  by $f$.

We say that $\Lambda$ is a NHIM for $f$ if there exists a splitting
\begin{equation}
\label{eq:bundles}\tag{{\bf B}}
T_xM=T_x\Lambda\oplus  E_x^s\oplus E_x^u, \textrm{ for all }x\in \Lambda
\end{equation}
that is invariant under $Df$, and, furthermore, there exist rates
\begin{equation}\label{eq:rates0}\tag{{\bf R}}
0<\lambda_\pm<1, \quad 0\leq  \mu_\pm, \quad \lambda_+\mu_-<1, \quad \lambda_-\mu_+<1
\end{equation}
and constants $C_\pm, D_\pm>0$ such that, for all $x\in M$ we have:
\begin{equation}\label{eqn:NHIM}\tag{{\bf H}}
\begin{split}
v\in T_x\Lambda\Leftrightarrow&\| Df^n(x)(v)\|\leq D_+\mu_+^n\|v\| \textrm{ for all } n\geq 0, \textrm { and }\\
 &\| Df^n(x)(v)\|\leq D_-\mu_-^{|n|}\|v\| \textrm{ for all } n\leq 0,\\
 v\in E_x^s \Leftrightarrow&\| Df^n(x)(v)\|\leq C_+\lambda_+^n\|v\| \textrm{ for all } n\geq 0,\\
 v\in E_x^u \Leftrightarrow&\| Df^n(x)(v)\|\leq C_-\lambda_-^{|n|}\|v\| \textrm{ for all } n\leq 0.
\end{split}
\end{equation}
(See Fig.~\ref{fig:rates}.)
\end{defn}

\begin{center}
\begin{figure}[h]
\begin{tikzpicture}
\draw[blue, line width=0.75mm] (0,0) -- (2,0);
\draw[thick] (2,0) -- (4,0);
\draw[orange, line width=0.75mm] (4,0) -- (8,0);
\draw[thick] (8,0) -- (10,0);
\draw[red, line width=0.75mm] (10,0) -- (12,0);
\node[draw] at (1,-1.25) {Contraction rates};
\filldraw[black] (2,0) circle (2pt) node[anchor=north]{$\lambda_+$};
\filldraw[black] (4,0) circle (2pt) node[anchor=north]{$\frac{1}{\mu_-}$};
\filldraw[black] (8,0) circle (2pt) node[anchor=north]{$\mu_+$};
\filldraw[black] (10,0) circle (2pt) node[anchor=north]{$\frac{1}{\lambda_-}$};
\node[draw] at (11,-1.25) {Expansion rates};
\end{tikzpicture}
\caption{Hyperbolic rates}
\label{fig:rates}
\end{figure}
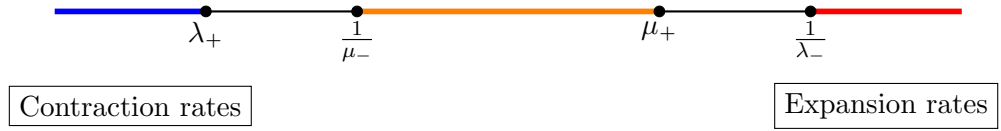
\end{center}

We denote  the dimension of $\Lambda$ by $d_c$, and the dimensions of $E_x^s$, $ E_x^u$ by $d_s$, $d_{u}$, respectively,
where $d_c+d_s+d_u=d$ with $d=\dim(M)$.

Condition \eqref{eq:rates0} implies that  there are gaps between the rates along the stable/unstable bundles and the rates along the NHIM, i.e.,
\begin{equation}\label{eqn:lambda_mu_gap} \lambda_+<\frac{1}{\mu_-} \textrm{ and } \mu_+<\frac{1}{\lambda_-}.\end{equation}
As a consequence of the rate conditions \eqref{eqn:NHIM}, by Lemma \ref{rem:ratesplusminus} we have that \begin{equation}\label{eqn:mu_gap} \frac{1}{\mu_-}\le \mu_+.\end{equation}

Since invariant manifolds for conformally symplectic systems are unbounded (see Remark \ref{rem:manifold_unbounded})
we need to make explicit
uniformity assumptions on the properties of the objects considered.

We have already made the assumption \eqref{U1} (which is implied by  \eqref{U1'}) that on the manifold $M$, there exists  a uniform $\C^r$ system of coordinates.

\medskip

Given a NHIM $\Lambda$, we fix  a uniform $\rho$-neighborhood of  $\Lambda$ in $M$
\begin{equation}\label{eqn:OOrho}
\OO_\rho=\{y\in M\,|\, d(y,\Lambda)<\rho\}, \textrm { for }\rho>0 .
\end{equation}
and we assume:

\begin{flalign}\label{U2}\tag{\bf{U2}}
&\textrm{We assume that  the manifold $\Lambda$ has a uniform tubular neigh-
}&
\\ \nonumber &\textrm{borhood: there is a $\C^r$ diffeomorphism, with a $\C^r$  inverse,}&
\\ \nonumber &\textrm{from a uniform neighborhood of the zero section of $E^s\oplus E^u$ }&
\\ \nonumber &\textrm{to the  uniform neighborhood $\OO_\rho$ of $\Lambda$ (see  \eqref{eqn:OOrho}).}&
\end{flalign}

\begin{rem}
An assumption  that implies assumption \eqref{U2} is  that the exponential  mapping
\begin{equation}\label{eqn:bounded_geometry}
  \mathcal{C}(x,s,u) = \exp_x(s+u),\quad x\in \Lambda, s\in E^s_x, u\in E^u_x,
\end{equation}
defines  a $\C^r$-diffeomorphism from a uniform
neighborhood of the zero section of $E^s\oplus E^u$ to
the  uniform neighborhood $\OO_\rho$ of $\Lambda$.
\end{rem}

Definition~\ref{def:nhim} says that in a small neighborhood of $\Lambda \subset M$ of any point $x\in \Lambda$,
the manifold $M$ is the product of the manifold $\Lambda$ and of the fibers of
the bundles $E^s$ and $E^u$.
For compact manifolds, this gives a tubular neighborhood of $\Lambda$.
See \cite{HPS}.

For unbounded manifolds, it can happen that  the manifold
$\Lambda$ folds  back and comes close to itself.
A concrete example appears in \cite[Example~3.8]{Eldering12}, which we illustrate in Fig.~\ref{fig:manifold_no_unif_tubular}.
The content of assumption \eqref{U2} is that this folding
back of the manifold $\Lambda$ does not happen.

\begin{figure}
\centering
\includegraphics[width=0.5\textwidth]{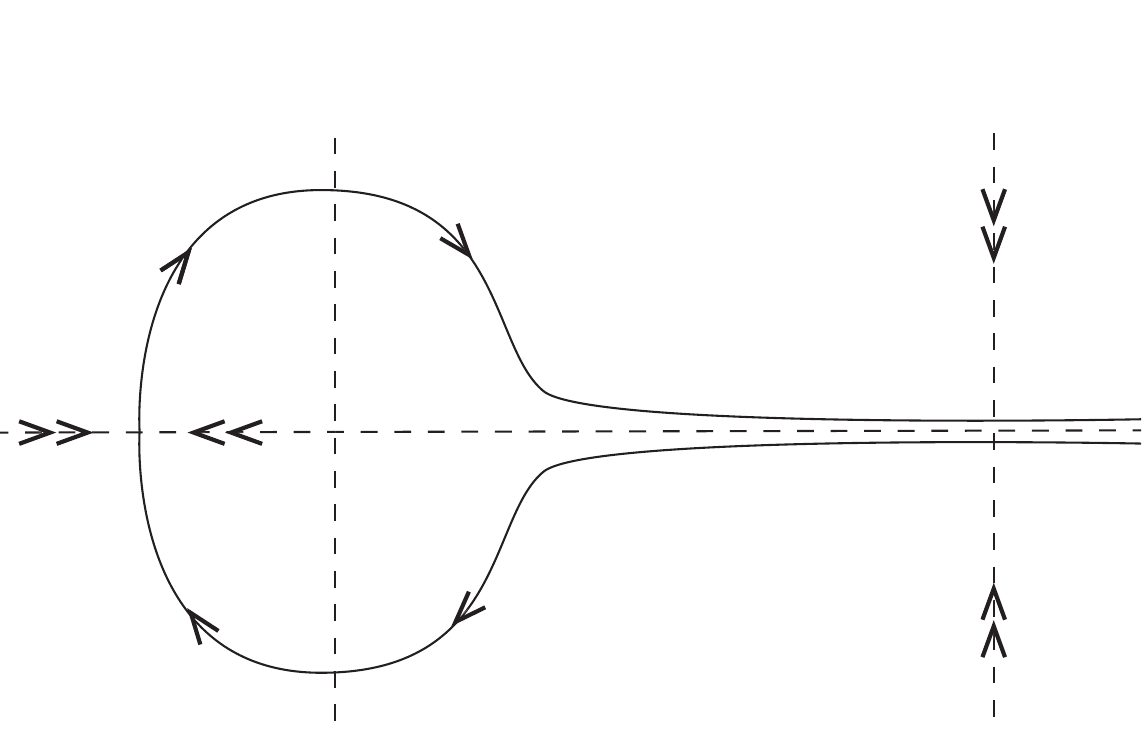}
\caption{A NHIM without a uniform tubular neighborhood. After \cite[Example~3.8]{Eldering12}}
\label{fig:manifold_no_unif_tubular}
\end{figure}

It is important to note that the assumptions \eqref{U1} and \eqref{U2} are of a very different nature, namely:
\begin{itemize}
\item
\eqref{U1} describes  a  local property of $M$;
\item
\eqref{U2} describes a global property of $\Lambda$.
\end{itemize}

In particular, as shown by the previous example,
the assumption  \eqref{U2}  does not follow
from \eqref{U1} or the other assumptions for a NHIM.

We also  assume that $f$, $f^{-1}$  are $r$ times differentiable with uniformly bounded derivatives in $\OO_\rho$:
\begin{equation}\label{U3}\tag{{\bf U3}}
\textrm{$f \in  \C^r(\OO_\rho)$ and $f^{-1}\in  \C^r(f(\OO_\rho))$.}
\end{equation}

A consequence of \eqref{U2} and \eqref{U3}   is that there exist  local stable and unstable manifolds $W^{s,\textrm{loc}}_\Lambda$, $W^{u,\textrm{loc}}_\Lambda$ in $\OO_\rho$,
as well as stable and unstable fibers $W^{s,\textrm{loc}}_x$, $W^{u,\textrm{loc}}_x$, for $x\in\Lambda$:

\begin{equation}\label{decomposition_strong}
W^{u,s, \textrm{loc}}_\Lambda= \bigcup _{x\in\Lambda}W^{s,u,\textrm{loc}}_x .
\end{equation}

The union on the right-hand side of \eqref{decomposition_strong}
is a disjoint union. Therefore, the decomposition
of $W^{u,s, \textrm{loc}}_\Lambda$ into strong stable
leaves is a foliation and, moreover,
the leaves are identified uniquely by their
intersection with $\Lambda$.

These (un)stable and strong (un)stable manifolds and their regularity are described  in Appendix \ref{sec:appendix_NHIM} (see  Theorem \ref{localstable}).
We will take advantage of the remarkable fact that $W^{u,s, \textrm{loc}}_\Lambda$ admit topological
characterizations, but has consequences for rates of convergence and regularity.

We will assume that the rates on the manifolds $\lambda _\pm$, $\mu_\pm$ (see \eqref {eq:rates0}) are such that:
\begin{itemize}
\item[ {\bf (H1)}]
The manifold $\Lambda$  is $\C^1$.
\item[ {\bf (H2)}]
The manifolds $W^{s,\textrm{loc}}_\Lambda$,
  $W^{u ,\textrm{loc}}_\Lambda$ are  $\C^1$.
  \item[ {\bf (H3)}]
  The manifolds $W^{s,\textrm{loc}}_x$,
  $W^{u ,\textrm{loc}}_x$ are  $\C^1$ uniformly in $x$.
  \item[ {\bf (H4)}]
  The  foliations  of the (un)stable manifold
  by strong (un)stable  manifolds
\begin{equation}\label{foliationsstablemanifolds}
W^{s,\textrm{loc}}_\Lambda = \bigcup_{x \in \Lambda}  W^{s,\textrm{loc}}_x,
  \quad
 W^{u,\textrm{loc}}_\Lambda = \bigcup_{x \in \Lambda}  W^{u,\textrm{loc}}_x,
\end{equation}
are of type $\C^{1,1}$. (See Definition~\ref{Ctildemm}.)
\end{itemize}

The hypotheses {\bf (H1)}, {\bf (H2)} {\bf (H3)}
are sufficient to use standard
differential geometry tools and define forms, etc.
The hypothesis {\bf (H4)} is the natural one  to ensure that  the wave maps
$\Omega_\pm$  (projections along the foliation, formally defined in Definition~\ref{wavemap}), are
$\C^1$.

In some proofs we will require  slightly stronger  regularity properties,
which we will make explicit.  These amount to assumptions on
the rates.

\medskip

In this paper,
we will assume that
\begin{equation}
  \label{eqn:mubounds}
  \tag{{\bf N}}
\mu_+\ge 1, \quad \mu_-  \ge 1
\end{equation}

The assumption \eqref{eqn:mubounds}  is to
avoid complicated statements. If $\mu_+ < 1$,
there is a unique fixed point $p \in \Lambda$,  and
$\Lambda$ is a submanifold of the stable manifold of $p$. Many of the regularity statements obtained in the general theory of NHIMs remain true, but they are far from optimal.
Also, in some estimates (derived from \cite[Proposition 15]{DLS08} or similar)
we use bounds like $C_1 \mu_+^n + C_2 \mu_+^{2n}  \le C \mu_+^{2n}$, for $n \ge 0$.
If we do not assume \eqref{eqn:mubounds},
different algebraic expressions for the bounds would be required for $\mu_+\ge 1$ and  for $\mu_+ < 1$, and therefore many statements would need to distinguish between the two different cases.

In some examples, we do not use \eqref{eqn:mubounds} but we mention this
explicitly.

\medskip

We will not use the persistence of NHIMs and their dependence on parameters in this paper,
but we will use the existence and properties of (un)stable manifolds foliated by
strong (un)stable manifolds. Of course, persistence
and dependence on parameters of NHIMs is likely  to  become useful
in future work.

%
%

\subsection{A system of coordinates}
\label{rem:new coordinates}

In $\OO_\rho$ (see \eqref{eqn:OOrho}), by taking the  system of coordinates assumed in \eqref{U1} and restricting it  to
  $W^{s,u,\textrm{loc}}_\Lambda $ we can  define a new system of coordinates on $W^{s,u,\textrm{loc}}_\Lambda$ as below.
  We use that $W^{s,u,\textrm{loc}}_\Lambda$ is foliated by $W^{s,u,\textrm{loc}}_x$, and that the foliation is $\C^{1,1}$.
In a small enough neighborhood, we can consider the foliation
by the $W^{s,u,\textrm{loc}}_x$ as a Cartesian product. Any point in $W^{s,u,\textrm{loc}}_x$ is given by the coordinate $y$ on the strong (un)stable manifold.
Thus, we obtain a new system of coordinates $\varphi$  around any $x \in \Lambda$ such   that
\begin{equation}\label{eqn:varphi_x_y}
  \{ \varphi(x, y ) | \  y \in B_{\tilde \rho}(0) \} = W^{s,u,\textrm{loc}}_x,
\end{equation}
for $B_{\tilde \rho}(0)\subseteq E^{s,u}_x$. Condition \textbf{(H4)} implies
that there exists a constant $C$ so that
\begin{equation}\label{eqn:mixed partials reg 1}
     \| \partial_x  \partial_y \varphi(x,y)\| \le C  .
\end{equation}

If the foliation is $\C^{\tilde m,m}$, then
there exist a constant $C$ so that for all $0\le i \le \tilde m$,
    $0 \le j \le m $,
\begin{equation}\label{eqn:mixed partials reg m}
   \| \partial_x^i \partial_y^j \varphi(x,y)\| \le C  .
\end{equation}

A more geometric version of the system of coordinates $\varphi$ is  obtained
using the exponential mapping.
Given $(x, y)$, $x \in \Lambda$, $y \in E^{s,u}_x$,
we define  $(x,y) \mapsto \exp_x(y)$, where
$\exp$ now denotes the exponential mapping on $W^{s,u, \textrm{loc}}_x$.

\begin{rem}
 The system of coordinates $\varphi$ constructed using the exponential mapping on $W^s_x$
 has remarkable regularity properties.
Given the smoothness of  the $W^{s,\textrm{loc}}_x$  the dependence of $\varphi(x,y)$ is
basically as smooth as the map $f$.
On the other hand, the dependence on $x$ is
not as differentiable and is limited by the hyperbolicity rates.
See Appendix ~\ref{sec:appendix_NHIM}.

Even if the coordinate system $\varphi$ has been constructed without reference
to symplectic forms, the vanishing lemmas (e.g.  Remark~\ref{rem:estimates_omega}, Lemma~\ref{lem:vanishingmanifolds})
show that  the coordinate system $\varphi$  enjoys rather remarkable symplectic properties.
We anticipate that these properties will be crucial in the proofs of Sections~\ref{sec:proof_main_2_b_2},\ref{sec:magic}, \ref{sec:Cartanexact}.
It is also a possibility of one step in Section~\ref{sec:iteration}.

\end{rem}

\subsection{Optimal rates}
\label{sec:optimalrates}
The way we have formulated the rate conditions in \eqref{eqn:NHIM},
$\mu_+$, $\mu_-$, $\lambda_+$, $\lambda_-$  are only bounds on the growth
of vectors and   can be replaced by other rates.
Hence, the only way that one can find
relations among them is
for the  `optimal'  bounds, which we denote by $\mu_+^*$, $\mu_-^*$, $\lambda_+^*$, $\lambda_-^*$, respectively.

More precisely, we define:
\begin{equation}
\label{eqn:optimal_rates}
\begin{split}
\lambda_+^*=&\inf\{\lambda_+\,|\,\|Df^n(x)v\|\le C_+\lambda_+^n \|v\|,\forall n\geq 0,\,  \forall x\in \Lambda,\, \forall v\in E^s_x\},\\
\lambda_-^*=&\inf\{\lambda_-\,|\,\|Df^n(x)v\|\le C_-\lambda_-^n \|v\|,\forall n\leq 0,\,   \forall x\in \Lambda,\,  \forall v\in E^u_x\},\\
\mu_+^*=&\inf\{\mu_+\,|\,\|Df^n(x)v\|\le D_+\mu_+^n \|v\|,\forall n\geq 0,\,  \forall  x\in \Lambda,\,  \forall v\in T_x\Lambda\},\\
\mu_-^*=&\inf\{\mu_-\,|\,\|Df^n(x)v\|\le D_-\mu_-^n \|v\|,\forall n\leq 0,\,  \forall x\in \Lambda,\,  \forall v\in T_x\Lambda\}.
\end{split}
\end{equation}

Notice that, in general, we do not have that
\[
 \exists C_+^*>0\, \textrm{ s.t. } \|Df^n(x)v\|\le C_+^* (\lambda_+^*)^n\|v\|,\,\forall n\geq 0,\,  \forall x\in \Lambda,\,  \forall  v\in E^s_x ,
 \]
but only that
\[
\forall \eps>0 \,  \exists C_+(\eps) \, \textrm{ s.t. } \|Df^n(x)v\|\le C_+(\eps)(\lambda_+^* + \eps)^n\|v\|, \,\forall n\geq 0,\,  \forall x\in \Lambda,\,   \forall v\in E^s_x .
\]
Similar statements hold for the other optimal rates.

\begin{rem}
From \eqref{eqn:lambda_mu_gap} and \eqref{eqn:mu_gap}, we obtain the relations:
\begin{eqnarray}
\label{eq:lambda_ratesplusminus}
\lambda_+^*\lambda_-^*<& 1,\\
\label{eq:ratesplusminus}
\mu_+^*\mu_-^*\ge& 1 .
\end{eqnarray}
Note that \eqref{eqn:mu_gap} (and so \eqref{eq:ratesplusminus}) is trivial if we assume \eqref{eqn:mubounds}.
\end{rem}

\subsection{The scattering map}
\label{sec:SM}
We recall here the scattering map \cite{DLS00, DLS08}.

Assume that $\Lambda$ is a NHIM for $f$ and the conditions  \textbf{(H1-H4)} are satisfied.

\subsubsection{The wave maps}

\begin{defn}\label{wavemap}
For a point $x\in W_\Lambda^{s,\textrm{loc}}$ (resp.\ $x\in W_\Lambda^{u,\textrm{loc}}$), we
denote by $x_+$ (resp.\ $x_-$) the unique point in $\Lambda$ which satisfies $x\in W_{x_+}^{s,\textrm{loc}}$ (resp.\ $x\in W_{x_-}^{u,\textrm{loc}}$).

Consequently the \emph{wave maps}:
\begin{equation}  \label{waveoperators}
\begin{split}
\Omega_{\pm} :W_{\Lambda}^{s,u, \textrm{loc}} & \longrightarrow \Lambda, \\
x & \mapsto x_{\pm},
\end{split}
\end{equation}
are well defined.

\end{defn}

The standing assumption   \textbf{(H4)} implies that the regularity
of the wave maps is at least $\C^1$.
If the foliations of the stable (unstable) manifolds were
$\C^{\tilde m, m}$ as in Theorem~\ref{localstable},
the wave maps would be $\C^{\tilde m}$.

From the definition of $\Omega_\pm$, it follows that the wave maps satisfy the following equivariance relations:
\begin{equation}\label{eqn:intertwining}
  \begin{split}
&\Omega_\pm \circ f^n_{\mid W^{s,u}_\Lambda}=f_{\mid\Lambda}^n\circ \Omega_\pm, \textrm{ for }n\in \Z,
\end{split}
\end{equation}
 where we denote by $f_{\mid N}$ the restriction of the map to a submanifold $N\subseteq M$.

 Notice that \eqref{eqn:intertwining} allows to define  the
 wave maps in the global (un)stable manifolds.
 They will be differentiable as many times as the map $f$ and  the foliation by the strong (un)stable maps.
 In particular, we can define the pullback by the wave maps.
 Nevertheless, it could  happen that  the derivatives of $\Omega_\pm$  are not
 uniformly bounded along the global (un)stable manifolds
 (e.g. if the (un)stable manifold oscillates).

 \subsubsection{Homoclinic channels}
 \label{sec:homoclinic_channel}
 The goal of this section is to define homoclinic intersections between $W^{u,\textrm{loc}}_\Lambda$ and $W^{s,\textrm{loc}}_\Lambda$ that give rise to a smooth family of homoclinic orbits to $\Lambda$.

We assume that   there is  a homoclinic manifold
$\Gamma \subset W_\Lambda^s\cap W_\Lambda^u$ (we require  more conditions on $\Gamma$ below).

More concretely, assume there exist  $N_-$, $N_+$
\[
\Gamma\subseteq f^{N_-}(W^{u,\textrm{loc}}_\Lambda \cap\OO_\rho)\cap f^{-N_+}(W^{s,\textrm{loc}}_\Lambda\cap\OO_\rho),
\]
where $\OO_\rho$ is defined in \eqref{eqn:OOrho},  and, abusing notation, we write:

\[
W^{s,\textrm{loc}}_\Lambda = \bigcup_{0\le n \le  N_+} f^{-n}(W^{s,\textrm{loc}}_\Lambda\cap\OO_\rho)  \textrm{ and }W^{u,\textrm{loc}}_\Lambda=\bigcup_{0\le n \le  N_-} f^{n}(W^{u,\textrm{loc}}_\Lambda\cap\OO_\rho).
\]
consequently
\begin{equation}
\label{eqn:Gamma_O N_plus_N_minus}
\Gamma\subseteq W^{u,\textrm{loc}}_\Lambda \cap W^{s,\textrm{loc}}_\Lambda .
\end{equation}
Since only a finite number of iterates are involved, the regularity of
$W^{s,u,\textrm{loc}}_\Lambda$ (as well as of its foliation) is the same as that for $W^{s,u,\textrm{loc}} \cap \OO_\rho$.

Now, we consider some neighborhoods of the stable
and unstable manifolds:
\begin{equation}
\label{eqn:N_minus_plus}
\OO^{+}_{\rho_+}  = \OO_{\rho_+} (W^{s,\textrm{loc}}_\Lambda),\,
\OO^{-}_{\rho_-} = \OO_{\rho_-}  (W^{u,\textrm{loc}}_\Lambda),
\end{equation}
and let
\begin{equation}\label{eqn:OO}
\OO:=\OO_\rho\cup \OO^+_{\rho_+} \cup \OO^-_{\rho_-},
\end{equation}

We   assume that:
\begin{equation} \label{eqn:f_f_inv_OO}\tag{{\bf U4}}
f \in \C^r( \OO) \textrm{ and }
    f^{-1}  \in \C^r( f(\OO)) .
\end{equation}
Observe that hypothesis \ref{eqn:f_f_inv_OO}  implies hypothesis \eqref{U3}, but we keep  them separated because
   some local results   require only  \eqref{U3} while more global ones require  \ref{eqn:f_f_inv_OO}.

\begin{rem}
Assumptions \eqref{U3} and \eqref{eqn:f_f_inv_OO} are satisfied if we make the simpler assumption that $f$  and $f^{-1}$ are $\C^r$ on $M$.
However, there are examples, for instance in Celestial Mechanics, where $f$ is unbounded (due to singularities of the vector field) but \eqref{U3} and \eqref{eqn:f_f_inv_OO}  hold.
\end{rem}

\subsubsection{Definition of the scattering map}
In this section we define the scattering map. By  assumption \eqref{eqn:Gamma_O N_plus_N_minus},
the stable and unstable manifolds of $\Lambda$, $W_\Lambda^{s,\textrm{loc}}$ and  $W_\Lambda^{u,\textrm{loc}}$, intersect  along the homoclinic manifold $\Gamma$.

We furthermore assume that the intersection between $W_\Lambda^{s}$ and  $W_\Lambda^{u}$  along the homoclinic manifold $\Gamma$ is transversal (see condition \eqref{intersection}) and that $\Gamma$ is transversal to the strong (un)stable foliations \eqref{foliationsstablemanifolds} (see condition
\eqref{gammatransversal}).
More concretely:

\begin{equation}\label{intersection}
  \begin{split}
    & \forall\ x\in\Gamma, \textrm{we have:} \\
&T_x M = T_x W_\Lambda^{s } + T_x W_\Lambda^{u },\\
&T_x W_\Lambda^{s } \cap T_x W_\Lambda^{u} = T_x \Gamma.
\end{split}
\end{equation}

\begin{equation}\label{gammatransversal}
\begin{split}
  & \forall\ x\in\Gamma, \textrm{we have:} \\
&T_x \Gamma \oplus T_x W_{x_+}^{s,\textrm{loc}} = T_x W_\Lambda^{s,\textrm{loc}},\\
&T_x \Gamma \oplus T_x W_{x_-}^{u,\textrm{loc}} = T_x W_\Lambda^{u,\textrm{loc}}.
\end{split}
\end{equation}

Given a manifold $\Gamma$  verifying \eqref{intersection} and  \eqref{gammatransversal}, we can consider the wave maps $\Omega _{\pm}$ of
\eqref{waveoperators} restricted to $\Gamma $.

Under the assumptions \eqref{intersection}, \eqref{gammatransversal} and {\bf (H4)},
we have that $\Gamma$ is $\C^{1}$ and that $\Omega _{\pm}$ are $\C^{1}$ local diffeomorphisms  from $\Gamma $ to $\Lambda$.

\begin{defn}\label{def:channel}
We say that $\Gamma$ is a \emph{homoclinic channel} if:
\begin{enumerate}
\item
$\Gamma \subset W_{\Lambda}^{s,\textrm{loc}}\cap W_{\Lambda}^{u,\textrm{loc}}$ verifies \eqref{intersection} and \eqref{gammatransversal}.
\item
  The wave  map  ${\Omega_-}_{\mid \Gamma}:\Gamma
\to  \Omega _{-} (\Gamma)\subset \Lambda$ is
a  $\C^{1}$-diffeomorphism.
\end{enumerate}
\end{defn}

The last hypothesis in Definition ~\ref{def:channel}, that ${\Omega_-}_{\mid \Gamma} $ is a diffeomorphism from
its domain to its range, can always be arranged by restricting $\Gamma$ to a smaller
neighborhood where the implicit function theorem applies.

\begin{rem}
If $\Gamma$ verifies the definition of a homoclinic
channel, so do subsets of $\Gamma$. Therefore, there is
no loss of generality in considering small enough channels.
One can assume  without loss of generality
that they are bounded.
\end{rem}

We denote by $\Omega_{\pm}^{\Gamma}=({\Omega_{\pm}}) _{ \mid \Gamma }$, and
$H_{\pm}^{\Gamma} = \Omega _{\pm}^{\Gamma}(\Gamma )\subset \Lambda
$, so that
\begin{equation*}\label{essential}
\Omega_{\pm}^\Gamma :\Gamma \longrightarrow H_{\pm }^{\Gamma}
\end{equation*}
are  $\C^{1}$-diffeomorphisms.


We define the scattering map associated to $\Gamma$ as follows:

\begin{defn}\label{def:scattering}
Given a homoclinic channel $\Gamma$ and $\Omega_\pm^{\Gamma}:\Gamma \to H_{\pm}^{\Gamma}$ the associated wave maps, we
define the scattering map associated to $\Gamma$ to be the $\C^1$-diffeomorphism
\[
S : H_-^\Gamma \subset \Lambda  \to H_+^\Gamma \subset \Lambda
\]
 given by
\begin{equation} \label{eq:scattering}
S= S^\Gamma=\Omega_+^\Gamma \circ (\Omega_-^\Gamma)^{-1}.
\end{equation}See Fig.~\ref{fig:scattering_map}
\end{defn}

\begin{figure}
\centering
\includegraphics[width=0.65\textwidth]{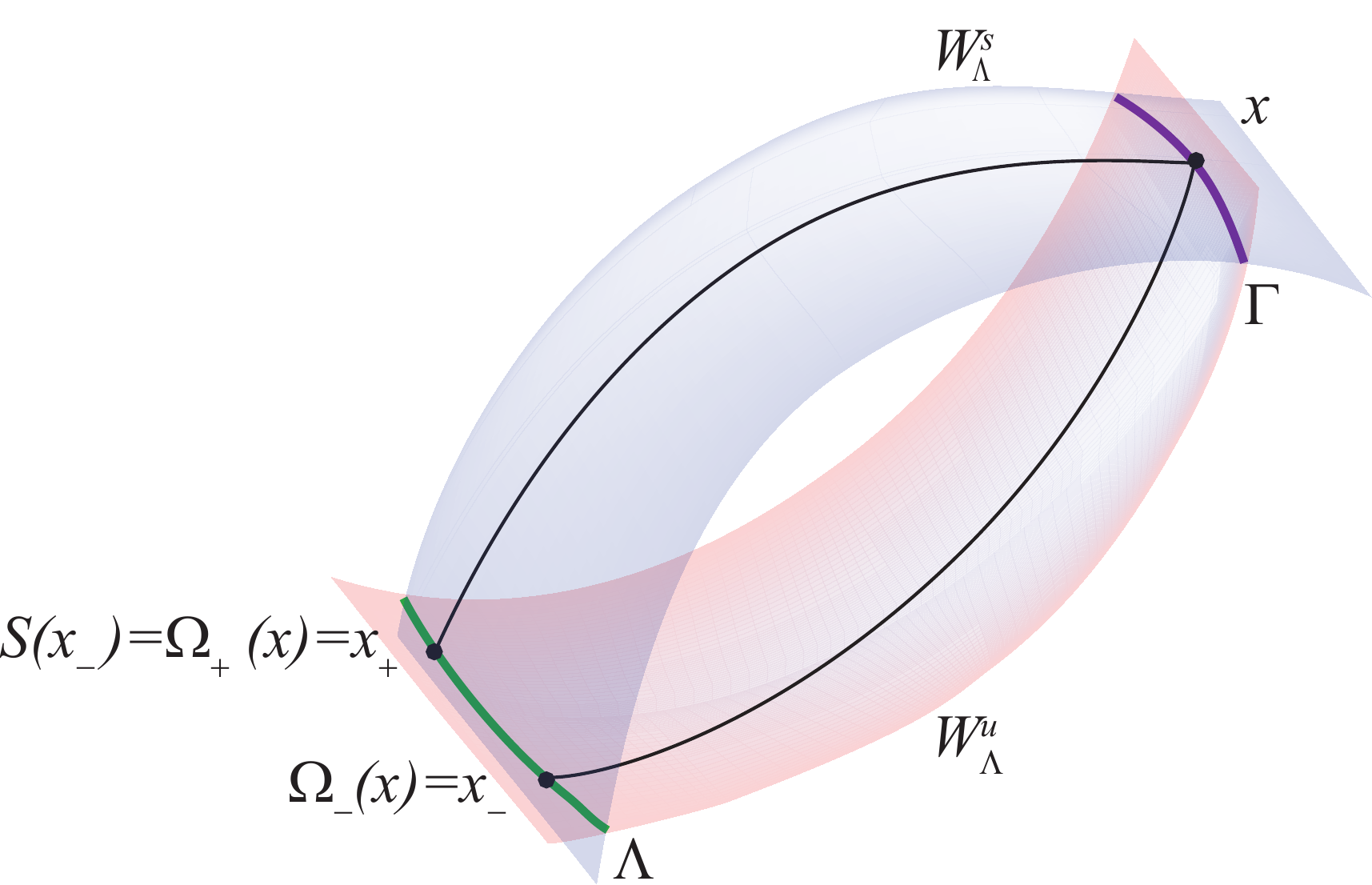}
\caption{The scattering map.}
\label{fig:scattering_map}
\end{figure}

\medskip

Since the projections $\Omega_\pm$ satisfy \eqref{eqn:intertwining}, we have the  $\Omega_\pm^\Gamma$ and $S^\Gamma$ satisfy the equivariance relations
\begin{equation}\label{eqn:iterationsomega}
\begin{split}
\Omega_\pm^{\Gamma}=&f_{\mid\Lambda}^{-n}\circ\Omega_\pm^{f^n(\Gamma)}\circ f^{n},\\
S^{\Gamma}=&f_{\mid\Lambda}^{-n}\circ S^{f^n(\Gamma)}\circ f^{n}.
\end{split}
\end{equation}

\begin{rem}
The fact that the scattering map is $\C^1$ is a consequence of  {\bf (H4)}. If the foliation of the stable (unstable) manifolds are
$\C^{\tilde m, m}$
then the scattering map is $\C^{\tilde m}$.

In general, the scattering map  depends on $\Gamma$ and is only locally defined. In \cite{DLS00} there are examples where the local domain
of the scattering map cannot be extended to a  global  one  (moving along a cycle in $\Lambda$ leads to
lack of monodromy).

The scattering map provides an efficient way to quantify
the effect of  homoclinic trajectories on the NHIM $\Lambda$.
In \cite{DLS00,DelshamsLS06a} it is shown that it can be used to study the heteroclinic intersections between invariant objects in $\Lambda$.
In \cite{GideaLS20,GideaLS20a} it is shown that iterations of the map restricted to the NHIM $\Lambda$ combined with iterates of the scattering map $S$ are closely followed by  true orbits.

In typical situations, we  have many scattering maps (due to the existence of
multiple intersections of the stable and unstable manifolds).
All these scattering maps can be used to generate rich dynamical behavior.

\end{rem}

\subsection{Uniformity assumptions on the symplectic form}
\label{sec:uniformity}

In Section \ref{sec:mainresults} we will assume that $M$ is endowed with a symplectic (presymplectic) form $\omega$.
We assume the  $\omega$ is  $\C^0$ on $\OO$ (see \eqref{eqn:OO}), i.e.:
\begin{equation}\label{eqn:bounded_omega}
\tag{{\bf U5}}  \|\omega _{\mid\OO}\| =\sup _{x\in \OO}\|\omega(x)\|  \le M_\omega<\infty.
\end{equation}

In Section \ref{proofnonclosed} and Section \ref{sec:iteration} we will require the stronger condition  that $\omega$ is $\C^1$ on $\OO$:
\begin{equation}\label{doundedomega}\tag{\bf U5$^\prime$}
  \|\omega_{\mid\OO}\|_{\C^1} =\sup _{x\in \OO}\|\omega(x)\|+ \sup _{x\in \OO}\|D\omega(x)\|\le M_\omega<\infty.
\end{equation}

Note that in Section \ref{sec:unbounded_forms} we will also consider the case when $\omega$ is unbounded.

\section{Main results on NHIMs and scattering maps}\label{sec:mainresults}

In this section we state  results on NHIMs. The proofs appear in Sections~\ref{sec:proofAmain1total} and \ref{sec:proof_main_2}.
Other results which do not involve involve NHIMs appear in Section~\ref{sec:topological}.

\subsection{Standing assumptions}
\label{sec:standing}
Let us summarize the standing assumptions introduced so far:
\begin{itemize}
\item[(i)]
$(M,\omega)$ is an orientable, non-compact, connected, symplectic, Riemannian manifold satisfying condition \eqref{U1},
\item[(ii)]
$f:M\to M$ is a conformally symplectic diffeomorphism of  factor $\eta>0$ (see Definition \ref{def:conformally_symplectic}),
\item[(iii)]
$\Lambda$ is a NHIM for $f$ satisfying  \eqref{eq:bundles} and the rate conditions \eqref{eq:rates0}, \eqref{eqn:NHIM} and \eqref{eqn:mubounds},  the regularity conditions  \textbf{(H1), (H2), (H3), (H4)},
and the uniformity condition \eqref{U2},
\item[(iv)]
$\Gamma$ is a homoclinic channel (see Definition \ref{def:channel}),
\item [(v)]
$f$ satisfies the uniformity conditions  \eqref{U3} and \eqref{eqn:f_f_inv_OO},
\item[(vi)]
The symplectic form $\omega$ satisfies the boundedness condition  {\bf (\ref{eqn:bounded_omega})}.
\end{itemize}

We note that the condition \eqref{eqn:f_f_inv_OO} implies \eqref{U3}, but some of the results only use \eqref{U3}.
Similarly the condition \textbf{(H4)} implies \textbf{(H2)} and \textbf{(H3)}, but some of the results only use \textbf{(H2)} or \textbf{(H3)}.
This is why we list all of these conditions separately.


\subsection{Symplectic properties of NHIMs and pairing rules}
\label{geomNHIM}

The first main result of this paper is:

\begin{thm}\label{thm:main1}
Under  the standing assumptions from Section \ref{sec:standing} we have:

\begin{description}
\item[(A) Symplecticity of the NHIM]
If the conformal factor $\eta$ and the hyperbolic rates  $\lambda_\pm$, $\mu_\pm$  in \eqref{eq:rates0} satisfy the inequalities
\begin{equation}\label{eqn:conformal_rates}\tag{{\bf S}}
\begin{split}
\mu_+ \lambda_+ \eta^{-1}<&1,\\
\mu_- \lambda_-\eta<&1,\\
\end{split}
\end{equation}
then the manifold $\Lambda$  is symplectic and $f_{\mid\Lambda}$ is conformally symplectic of conformal factor $\eta$.

\item[(B) Pairing Rules]
The manifold $\Lambda$ is symplectic if and only if the optimal hyperbolicity rates $\mu_\pm ^*$, $\lambda_\pm ^*$ defined in \eqref{eqn:optimal_rates}
satisfy
\begin{equation}
\label{eqn:pairing_rules}\tag{{\bf P}}
\begin{split}
\frac{\lambda_+^*}{\lambda_-^*}=&\eta, \textrm { and} \\
\frac{\mu_+^*}{\mu_-^*}=&\eta.
\end{split}
\end{equation}
\end{description}
\end{thm}

The proof of Theorem \ref{thm:main1} part {\bf(A)} is given in Section \ref{sec:proofAmain1}; the main ingredient   is a  vanishing lemma (Lemma \ref{lem:vanishing}). Part {\bf(B)} is proved in Section \ref{sec:pairing_rules} using another
vanishing lemma (Lemma \ref{cor:pairing}).

\begin{cor} \label{partial_converse}
Under  the standing assumptions from Section \ref{sec:standing} we have:

If $\Lambda$ is symplectic, then  it  has rates $\lambda_\pm$, $\mu_\pm$ satisfying \eqref{eq:rates0} that moreover satisfy
\eqref{eqn:conformal_rates}.
\end{cor}

\begin{proof}
Since $\Lambda$ is symplectic, by  Theorem \ref{thm:main1} part {\bf(B)}   the optimal rates satisfy the pairing rules \eqref{eqn:pairing_rules}.
By the Definition ~\ref{def:nhim} they also satisfy the rate conditions \eqref{eq:rates0}.
Then we have:
\[
\begin{split}
&\lambda^*_+ < \frac{1}{\mu^*_-}  = \frac{\eta}{\mu^*_+}, \\
&\frac{1}{\lambda^*_-} > \mu^*_+  =  {\eta}{\mu^*_-}.
\end{split}
\]
With algebraic manipulations,   this is precisely \eqref{eqn:conformal_rates}.
Now, if we have that $\lambda_+^*\mu_+^*\eta^{-1}<1$, there exist $\lambda_+>\lambda_+^*$ and $\mu_+>\mu_+^*$ still satisfying
$\lambda_+\mu_+\eta^{-1}<1$.
An analogous reasoning gives the existence of $\lambda_->\lambda_-^*$ and $\mu_->\mu_-^*$  still satisfying
$\lambda_-\mu_-\eta<1$.
This concludes the proof.
\end{proof}

Theorem~\ref{thm:main1} and Corollary~\ref{partial_converse} show that condition \eqref{eqn:conformal_rates}
is necessary and sufficient for the manifold $\Lambda$ to be symplectic.

Note that in Theorem~\ref{thm:main1}, the hypothesis is a condition on the rates \ref{eqn:conformal_rates}, and the conclusion is another condition on the rates \eqref{eqn:pairing_rules}. However, to arrive at this conclusion, we must go through the geometry, by showing that
$\Lambda$ is symplectic.

In Section~\ref{sec:isotropic}, after developing some tools, we show that some conditions on the rates of an isotropic invariant manifold    obstruct normal hyperbolicity.

\subsection{Symplectic properties of scattering maps}
\label{geomScat}

\begin{thm}\label{thm:main2}
Under  the standing assumptions from Section \ref{sec:standing}, assume that the conformal factor $\eta$ and the hyperbolic rates  $\lambda_\pm$, $\mu_\pm$  in \eqref{eq:rates0} satisfy the inequalities \eqref{eqn:conformal_rates}.

Then we have:

\begin{description}
\item[(A) Symplecticity of the homoclinic channel]

The manifold  $\Gamma$ is symplectic.
\item[(B)  Symplecticity of the scattering map]

The wave maps $\Omega_\pm:W^{s,u,\mathrm{loc}}_\Lambda  \to \Lambda$ defined in \eqref{waveoperators}
and \eqref{eqn:intertwining} satisfy:
\begin{equation}
\label{eqn:symplecticwavemaps}
\begin{split}
 (\Omega_+)^* (\omega_{\mid \Lambda}) = \omega_{\mid W^{s,\mathrm{loc}}_{ \Lambda }},\\
 (\Omega_-)^* (\omega_{\mid \Lambda}) = \omega_{\mid W^{u,\mathrm{loc}}_{ \Lambda }} .
\end{split}
\end{equation}
As a consequence, since $\Gamma$ is symplectic $\Omega_\pm^\Gamma =(\Omega_\pm) _{\mid \Gamma} $ are symplectic maps,  and the scattering map  $S=S^\Gamma= \Omega_+^\Gamma \circ (\Omega_-^\Gamma) ^{-1}$
is symplectic:
\[
S^*(\omega_{\mid\Lambda})=\omega_{\mid\Lambda}.
\]
\item[(C)  Exact symplecticity of the scattering map]
Assume further that the symplectic form is exact $\omega = d\alpha$.

Then,
\begin{equation}\label{eqn:wave_exact}
\begin{split}
(\Omega_+)^*(\alpha_{\mid \Lambda})-\alpha_{\mid W^{s,\mathrm{loc}}_\Lambda} =dP^+_\alpha ,\\
(\Omega_-)^*(\alpha_{\mid \Lambda})-\alpha_{\mid W^{u,\mathrm{loc}}_\Lambda} =dP^-_\alpha ,
\end{split}
\end{equation}
where $P^\pm_\alpha$ are  functions on $W^{s,u,\textrm{loc}}_\Lambda$, respectively.

Hence, the scattering map $S$ is exact with respect to $\alpha$, that is
\[
S^*(\alpha_{\mid \Lambda}) =\alpha _{\mid \Lambda}+dP_\alpha^S ,
\]
where $P_\alpha^S$ is  a function on $\Lambda$.

Explicit formulas for $P_\alpha^\pm$ and $P_\alpha^S$ in the case when $f$ is also exact are provided in Lemma~\ref{lem:sum_path_plus_minus}.
\end{description}
\end{thm}

A remarkable aspect of part {\bf(C)} of
Theorem~\ref{thm:main2} is that $\Omega_\pm^\Gamma$   and $S$ are  exact symplectic
for all action forms $\alpha$.
For conformally symplectic systems, one expects that a map could be exact for some action
form but not for others. See Example~\ref{ex:standard_map}.

In Section~\ref{sec:gauge} we have studied the effect of gauge transformations
(changing $\alpha$ into $\alpha + dG$ for some function $G$) on the primitive
functions of exact (conformally) symplectic maps.
A remarkable result (see \eqref{eqn:Pf_gauge}) is that  the primitive of the scattering map
is invariant under normalized gauge changes, that is, gauge functions $G$ that vanish on $\Lambda$.

The proof of Theorem~\ref{thm:main2} is given in Section \ref{sec:proof_main_2}.

In Section \ref{sec:proof_main_2_b} we give seven different proofs of part \textbf{(B)}
of Theorem~\ref{thm:main2}. Some of them require slightly different hypotheses.
For instance, some of the proofs do not use that $\omega$ is closed or non-degenerate, other proofs
assume that $\omega$ is $\C^1$-bounded, and other ones assume
different conditions among  the hyperbolic rates and the conformal factor.
Thus, they  are also applicable in non-symplectic contexts, as in the applications described in the Appendix
\ref{othermodels}.

In Section~\ref{sec:proofpartC} we give two different proofs of part \textbf{(C)} of
Theorem~\ref{thm:main2}.
One  is based on Stokes theorem, and the second one on Cartan's magic formula.

In Section~\ref{othermodels}, we present several problems
that have
appeared in the literature, for which the methods developed
here apply even if the forms playing a role  are not symplectic.

\medskip
To avoid developing complicated language,  when
$\Omega_+^* \omega = \omega$, we will say that $\Omega_+$ is
symplectic even if $\omega$ is not assumed to be closed or non-degenerate.
\medskip

The variety of  proofs shows that the remarkable cancellations leading to
the symplectic properties of the scattering maps are at the crossroads of
several ideas in symplectic geometry. It seems that this paper
has  only started to explore the possibilities.

\subsection{Results for presymplectic systems}
\label{sec:presymplectic}

In this section we present some results analogous to Theorem \ref{thm:main1} and  Theorem \ref{thm:main2} for presymplectic systems.
We assume that $\omega$ is a presymplectic form on $M$ (see Definition \ref{defn:presymplectic}) and $f$ is a conformally presymplectic map (see Section \ref{sec:presymp_maps}).
A motivation for us is to study some NHIMs that appear in quasi-integrable systems near multiple  resonances.

Similarly to the symplectic case (see Remark \ref{rem:conformal_factor_constant}), under conditions on dimensionality, the conformal presymplectic factor needs to be a constant.

\begin{prop}\label{prop:conformal_factor_constant}
If  for any $x\in M$ we have that $\textrm{codim}(K_x(\omega))\ge 4$, then the conformal presymplectic factor $\eta(x)$ is a constant.
\end{prop}

\begin{proof}
We will use the following

\begin{lem}\label{lem:conformal_factor_constant}
Assume $\omega(x)\neq 0$ for any $x\in M$.
Then $d\eta =0$ on $K(\omega)$.
\end{lem}

\begin{proof}
We have
\[
0=f^*(d\omega)=d f^*\omega=d(\eta\omega)=d\eta\wedge \omega.
\]

Then for all $u\in K_x(\omega)$, and all $v,w\in T_xM$ we have
\[
0=d\eta(x)(u)\omega (x)(v,w)+\textrm{ other terms with }\omega (x)(u,\cdot).
\]
Since $u\in K_x(\omega)$, the other terms are $0$.
Since  $\omega(x) \neq 0$, there exist $v,w$ such that $\omega(x)(v,w)\neq 0$.
Therefore $d\eta (x) (u)=0$.
As the result is true for any $x\in M$ the lemma is proved.
\end{proof}

Now we proceed with the proof of Proposition \ref{prop:conformal_factor_constant}.
There exist $A\subset TM$ such that \[TM=K(\omega)\oplus A \textrm { and } \omega_{\mid A} \textrm { is non-degenerate.}\]
Then using the same argument as in the symplectic case  (see Remark \ref{rem:conformal_factor_constant}) we obtain that
\[
\dim A\ge 4 \Rightarrow d\eta_{\mid A}=0,
\]
therefore, using the result of the previous lemma we obtain  $d\eta \equiv 0$ and consequently $\eta=\textrm{const.}$
\end{proof}

Now we examine the integrability of the kernel of  a presymplectic form.
For a fixed $x$, the kernel $K_x (\omega)$ is a linear subspace of $T_x M$, and the family $K_x (\omega)$, $x\in M$, determines a \emph{distribution} on $M$.
Rank-$1$ kernels are just  multiples of a vector field
and can be integrated by solving ODE's. For higher rank kernels, the integrability is
non-trivial (see \cite{LibermannM87,Souriau97,AlishahL12}).

We say that a presymplectic $\C^r$-form $\omega$ has the constant rank property
on an open set $\U \subset M$ if  the dimension of $K_x$, for $x\in \U$,  is constant.

\begin{lem}\label{frobenius}
Let $\omega$ be a $\C^r$ ($r\ge 1$) presymplectic form
with constant rank on $M$.

Then, $K(\omega)$ is an integrable distribution.
That is, there exists a foliation $\F$ whose leaves are
$\C^r$ isotropic manifolds.

The form induced by $\omega$ on the quotient of the manifold $M$
by the foliation $\F$ is  a symplectic form.
\end{lem}

\begin{proof}
For any $\C^r$-vector fields ($r\ge 1$) $u,v,w$, applying the standard formula for  the derivative of $\omega$ yields
\[
\begin{split}
0 = (d\omega)(u,v,w) &=  u \omega(v,w) - v \omega(u,w)   + w \omega(u,v) \\
              & +\omega([u,v],w) - \omega([v,w], u ) + \omega([w,u], v)
\end{split}
\]
If we now assume that $u,v \in K$, we obtain that for any $w$ one of the terms survives.  Hence, for any $w$, we
have $\omega([u,v],w)  = 0$, i.e.
$[u,v] \in K$.

That is, the distribution $K$ is closed under taking commutators.
This is the hypothesis of Frobenius theorem.
A version of Frobenius theorem with low regularity appears in \cite[pp. 123-124]{Hartman}. See also \cite{Yao23}.

Applying now Frobenius theorem, we obtain the existence of a foliation $\mathcal{F}$ integrating the distribution given by the kernel $K(\omega)$.

The fact that the form is non-degenerate follows, quotienting by the kernel we obtain a non-degenerate form.
\end{proof}

Assume that $f$ is conformally presymplectic.
If $u \in K_x(\omega)$ then, for any $v\in T_{x}M$,  $\omega(x)(u,v)=0$ and, consequently:
\[
\omega(f(x)) (Df(x)u, Df(x)v)=f^*\omega(x)(u,v) =\eta \omega(x)(u,v)=0
\]
therefore $Df(x)u \in K_{f(x)}(\omega)$.
That is, a conformally presymplectic map $f$ transforms
$K_x(\omega)$ into $K_{f(x)}(\omega)$.
Consequently, when  the rank is constant and the foliation
$\F$ exists,  the leaves of the foliation $\F$ are
preserved by $f$. It is then natural
to define an induced map $\tilde{f}$ in the space of leaves.

We can obtain a concrete representation  of the space of leaves and the dynamics on it by taking transversal sections
$T_x, T_{f(x)} $
to the foliation $\F$ at $x$ and $f(x)$, respectively.
Each transversal  can be endowed with the restriction of
$\omega$, which is non-degenerate since the transversal
excludes the kernel of the form.
Note that $\omega_{\mid T_x}$ and $ \omega_{\mid T_{f(x)}}$
are closed because the exterior derivative commutes with the restriction.
Then, given  $y \in T_x$,
associate to it $\tilde y =\tilde f(y)\in T_{f(x)}$
defined by $\tilde y = H \circ f(y)$
where $H$ is the holonomy map  sending $x$ to $f(x)$.
If $f$ is conformally presymplectic, then
$\tilde f$ is conformally symplectic from
$(T_x,\omega_{\mid T_x})$ to $(T_{f(x)},\omega_{\mid T_{f(x)}})$.
Similar constructions appear in \cite[p. 106 ff.]{LibermannM87}.
We will apply a similar construction to the
scattering map.

A useful  consequence is the following:
\begin{prop}\label{prop:presymplectic foliation of Lambda}
  Let $\Lambda$ be a NHIM for a conformal
  presymplectic map $f$.

  Assume that $\omega_{\mid\Lambda} $ is constant
  rank presymplectic.

Let $\{\F_x  \}_{x \in U}$ be the leaves of the foliation in $\Lambda$
  integrating $K(\omega_{\mid\Lambda})$.

Then, $\{ W^s_{\F_x } \}_{x \in U}$ is a foliation of $W^s_\Lambda$ integrating $K( \omega|_{W^s_\Lambda})$.
\end{prop}

\begin{proof}[Proof of Proposition \ref{prop:presymplectic foliation of Lambda}]
 Since $\omega_{\mid\Lambda}$ is presymplectic, we can use  Lemma \ref{frobenius} to integrate the kernel $K(\omega_{\mid\Lambda})$,  yielding a foliation $\F$ of $\Lambda$.
 Let $\{\F_x  \}_{x \in \Lambda}$ be the leaves of the foliation.

  We observe that $\F_x$ is an isotropic submanifold in $\Lambda$.
  By applying Proposition~\ref{lem:degeneracy},
  we obtain that  $W^s_{\F_x}$ is an isotropic submanifold in
  $W^s_\Lambda$. It is also a foliation of $W^s_\Lambda$.
\end{proof}

\begin{rem}\label{rem:constant_rank}
  The constant rank property of presymplectic forms is taken as part of the definition
   in some treatments \cite{Souriau97,LibermannM87}.
   As we noted earlier, the constant rank assumption is not an open condition,
   since adding an arbitrary  small perturbation to the
  form may decrease the dimension of the  kernel.

In this paper, the kernel of $\omega$ is obtained by
applying  vanishing lemmas (see Section \ref{sec:vanishing_lemmas}) assuming conditions on the rates.
Since the rates in bundles are continuous under small perturbation, we conclude
that, in such a case, perturbations do not change the dimension of the kernel.
Hence, in such cases,  the constant rank assumption is very natural.
More concretely, the symplectic form $\omega$ restricted to $W^{s,u}_\Lambda$ is presymplectic and
has constant rank; see Proposition~\ref{lem:degeneracy}.
The foliation integrating the kernel of $\omega_{\mid W^{s,u}_\Lambda}$ is $\{W^{s,u}_x\}_{x\in\Lambda}$.
The directions complementary to the kernel integrate to give  symplectic manifolds transverse to the leaves (such an example is a homoclinic channel as given in Definition \ref{def:channel}).

Another example when the constant rank property is implied by the hyperbolic rates is shown in Example~\ref{intermediatenew}.
 \end{rem}

\subsubsection{The scattering map for conformally presymplectic systems}

The main result for conformally presymplectic systems is:
\begin{thm}\label{mainpre}
Assume that $\omega$ is a presymplectic form on $M$, $f$  is a  conformally presymplectic map, and the standing assumptions (i), (iii)-(vi) from Section~\ref{sec:standing} hold for $f$ and $\omega$.
Assume the conformal factor \eqref{confpresymplectic} satisfies the rate conditions:

\begin{equation}\label{eqn:rates_conditions_pre}
\tag{{\bf S'}}
\begin{split}
\mu_+ \lambda_+ \eta_-^{-1}<&1,\\
\mu_- \lambda_-\eta_+<&1.\\
\end{split}
\end{equation}

Then:
\begin{description}
\item [(A) Presymplecticity of  NHIM and of homoclinic channel]{$ $}
$\Lambda$ and $\Gamma$ are presymplectic.
\item[(B) Presymplecticity of scattering map]
The wave maps
\[
\Omega_{\pm}:W^{s,u,\mathrm{loc}}_\Lambda\to\Lambda
\]
 preserve the presymplectic form $\omega$ in the sense of \eqref{eqn:symplecticwavemaps}.

As a consequence, the scattering map $S$ associated  to $\Gamma$  is presymplectic.
\item[(C) Exact presymplecticity of  scattering map]
Assume further that the presymplectic form is exact, i.e., $\omega = d\alpha$.
Then, the scattering map $S$ is exact, that is
\[
S^*\alpha =\alpha+dP^S .
\]
\item[(D) Dynamics in the kernel of the  presymplectic form]
The foliation  $\mathcal{F}$ described in
Lemma~\ref{frobenius} is preserved by both the dynamics $f$ and the scattering map $S$.

When these dynamics are projected onto any transversal section to the foliation, they give  conformally symplectic and symplectic  maps respectively.

\end{description}

\end{thm}

The proof of Theorem \ref{mainpre} is given in Section \ref{sec:proofs_presymplectic}.

\section{Results on topology of manifolds with conformally symplectic dynamics}
\label{sec:topological}

In this section, we show that there are interactions between the
(co)homology of the manifold  and the set of conformally symplectic
factors. In particular, we  present an answer to a question
posed in \cite[p. 160]{ArnaudF24}.

We do not need  many of the analytical assumptions in Section \ref{sec:standing}, particularly those concerning NHIMs.

We assume  that we have  a well defined cohomology theory
and that the $1$ and $2$-cohomology considered are finite dimensional (hence,
we can define pull-back operators and they are finite dimensional).
Then, we  obtain results for maps which are conformally symplectic with respect to  forms in this class.

For unbounded manifolds
there are several possibilities of cohomology theories and they  may give different
obstructions to conformal factors.

When we discuss applications to concrete examples, we will make explicit
the cohomology theory we are using.

\begin{rem}
  For unbounded manifolds, it is very natural to have infinite dimensional cohomology
  (for example, an unbounded cylinder with infinitely many handles attached).
We do not explore these cases in this paper.
\end{rem}

\subsection{Topological obstructions to exactness}
\label{sec:exactness}
For a diffeomorphism $f$, we denote by $f^\# $ the induced map  on
cohomology and by $f_\#$ the induced map on homology.
We will only  consider the action on $1$- and  $2$-(co)homology, and when  we need to
make explicit the order of the cohomology, we will add a number  to the symbol $\#$.

We reserve the notation $f^*$
for pull-back
Denoting by $[ \beta]$
the cohomology class of a closed form,
we have $f^\#[\beta] = [f^* \beta]$.

\begin{lem}
\label{lem:obstruction}
Let  $f$ be a conformally symplectic map for a non-exact form~$\omega$, and  $f^{\#2} :H^2(M)\to H^2(M)$ be the homomorphism induced by $f$ on the cohomology group of order $2$.
Then the conformal factor $\eta$ is an eigenvalue for $f^{\#2}$.
\end{lem}

\begin{proof}
Since $f^*\omega=\eta\omega$, we have
$f^{\#2} [\omega]=\eta[\omega]$. Then  the conformal factor $\eta$ is an eigenvalue for $f^{\#2}$ because for a non-exact form,  $[\omega]\neq 0$.
\end{proof}

The result  below is a converse of Lemma \ref{lem:obstruction}.

\begin{lem} \label{generating}
  Assume $\eta$ is not an eigenvalue of $f^{\#2}$.

  Then,  the symplectic form $\omega$ is exact, $\omega = d \alpha$, for some $1$-form $\alpha$.
\end{lem}
\begin{proof}
  Taking $2$-cohomology on the definition of conformally symplectic map
  \eqref{eqn:conformally_symplectic},
  we obtain:
  \[
  f^{\#2} [\omega] - \eta [\omega] = 0 .
  \]
  Hence, if $\eta$ is not an eigenvalue of $f^{\#2}$, then
 $[\omega] = 0$. Therefore,  there is a $1$-form
  $\alpha$ so that
$\omega = d \alpha$.
\end{proof}

\begin{lem}\label{generating2}
  Assume that
  $\omega = d\alpha$ for some $1$-form $\alpha$.
  \begin{itemize}
  \item [(i)] If $\eta$ is not an eigenvalue of $f^{\#1}$, then  there exists a closed $1$-form $\beta$
  so that $\tilde \alpha = \alpha + \beta$
  satisfies $\omega = d \tilde \alpha$ and
  \begin{equation}\label{newgauge}
  f^* {\tilde \alpha} -\eta {\tilde \alpha} = d P
  \end{equation}
  for some primitive function $P$, and therefore $f$ is exact with respect to $\tilde\alpha$.
  The $1$-form $\beta$  is unique up to
  the addition of an exact form $dG$,  for some function $G:M\to\mathbb{R}$.
  Hence, the function $P$ is unique up to the addition of
  $G \circ f - \eta G$.

  \item[(ii)] If $\eta$ is an eigenvalue of $f^{\#1}$ we have the following dichotomy:
  \begin{itemize}
  \item[(ii.a)] If $[f^*\alpha-\eta\alpha]\not \in \textrm{Range}(f^{\#1}-\eta\textrm{Id})$ then there  is no closed $1$-form $\beta$ such that $f$ is exact for $\tilde\alpha$.
  \item[(ii.b)] If $[f^*\alpha-\eta\alpha]\in \textrm{Range}(f^{\#1}-\eta\textrm{Id})$ then there exists a closed $1$-form $\beta$ such that $f$  is exact for $\tilde\alpha$. Moreover, such $ \beta $ is unique up to   the addition  of   an exact form
$dG$  and of a linear combination of
    $\beta_1,\ldots,\beta_l$, where $[\beta_i]$, $i=1,\ldots,l$, are
    a basis of the space of eigenvectors of $f^{\#1}$  corresponding to  $\eta$  and  $l$ is the algebraic multiplicity of $\eta$
    as an eigenvalue of $f^{\#1}$.
  \end{itemize}

\item [(iii)] If $f$ is symplectic, i.e., $\eta=1$, and $f^{\#1}=\textrm{Id}$, then either $f$ is exact for all actions $\alpha$ or $f$  is not exact
  for any $\alpha$.
  \end{itemize}
\end{lem}
\begin{proof}
Let  $\tilde\alpha=\alpha +\beta$, with $\beta$ closed $1$-form.

The map $f$ is exact with respect to $\tilde \alpha$  if and only if we have \eqref{newgauge}.
Taking $1$-cohomology   we get
  \begin{equation}\label{transformedgauge}
    [f^* {\tilde \alpha} -\eta {\tilde \alpha}]
    = [f^* { \alpha} -\eta {\alpha}] + f^{\#1} [\beta] -\eta [\beta]=0, \end{equation}
which is equivalent to
\begin{equation}\label{eqn:range} (f^{\#1}-\eta\textrm{Id})[\beta]=- [f^*\alpha-\eta\alpha].\end{equation}

  For part (i), as $\eta$ is not an eigenvalue of $f^{\#1}$,  we can find a unique solution $[\beta]$ for \eqref{eqn:range}.
This determines $\beta$ up to
    the addition of  the differential of a function $G$.
    The corresponding change in the primitive function $P$ follows from \eqref{eqn:Pf_gauge}.

Part (ii.a) follows from linear algebra.

For part (ii.b), if $[\beta]$ and $[\beta']$ are two solutions of  \eqref{eqn:range}, then $[\beta]-[\beta']$ is in the eigenspace of $\eta$.

For part (iii),
we have that $\eta=1$ is an eigenvalue of $f^{\#1}=\textrm{Id} $ and   $\textrm{Range}(f^{\#1}-\eta\textrm{Id}) =\{0\}$.

Then,  we observe that for any closed form $\beta$,
\[
f^{\#1}(\alpha +  \beta) - (\alpha + \beta) =
[f^{\#1} \alpha - \alpha] .
\]
Therefore, the map $f$ is exact for all $\alpha + \beta$  --
all possible action forms -- if and only if $f$ is exact for $\alpha$.
\end{proof}

\begin{rem}\label{maybeout}
The assumption in part (iii)   that $f^{\#1}$ is the identity
is natural for  symplectic maps on compact manifolds.
This occurs  when one considers a cohomology theory
which is invariant under homotopy and $f$ lies in
the connected component of the identity in the group of symplectic diffeomorphisms.
This is the case, for instance, of the time-$1$ map  of  an (autonomous) Hamiltonian flow.

The dichotomy in (iii) of $f$ being exact for all actions forms or for none,
is not true non-symplectic maps (see Example~\ref{ex:standard_map}),
or for symplectic maps not in the connected component of identity.
\end{rem}

\begin{rem}
\label{presymplectic_obstruction}
The result in Lemma~\ref{lem:obstruction} does
not depend on  the fact that $\omega$ is non-degenerate.
Hence,  the topological obstruction applies
just as well to the conformal factors for
non-exact pre-symplectic forms,
when the conformal factor is  a constant.
\end{rem}

\begin{rem}
  \label{several}
  In unbounded manifolds there could be several different cohomology theories giving different results.
  Since we think of  Lemma~\ref{lem:obstruction} as providing obstructions for possible conformal factors $\eta$, the existence of several theories   provides different obstructions for each cohomology theory.

  Some of these theories may lead to trivial obstructions,
  and several cohomology theories may lead to the same obstruction.
  In  Example~\ref{concrete} we will apply simplicial cohomology
  to obtain  obstructions for $\mathcal{R}$. For this example, we show that these obstructions
  are optimal  (see Proposition \ref{prop:ArnaudF24}). In Remark~\ref{rem:Borel_Moore} we will apply Borel-Moore homology
obtaining only trivial obstructions. The survey \cite{BanyagaHS24}
presents some other cohomology theories
useful for (locally) conformally symplectic geometry.
\end{rem}

\begin{rem}
  \label{rem:locallyconstant}
  Many (co)-homologies theories are homotopy invariant,  i.e., homotopic maps induce  maps on  (co)-homology that are isomorphic\footnote{However, this is not the case for the Borel-Moore cohomology; see Remark~\ref{rem:Borel_Moore}}.
  In such a case, the  conformal factor $\eta$ for a non-exact
  symplectic  remains
  constant under $\C^1$-homotopy.

  Under assumption \eqref{U1}, if $g,f$ are $\C^1$-close
  (hence $\C^1$ homotopic) conformally
  symplectic mappings for the same non-exact form,
  then the conformal factors have to be the same.

  Hence, for a non-exact form, the set of conformally symplectic factors is locally
  constant for $f$ in the $\C^1$ topology.
\end{rem}

\begin{rem}
  A particular case of Remark~\ref{rem:locallyconstant}
  is that 
  a symplectic diffeomorphism  for a non-exact form cannot be perturbed into a conformally symplectic one.
  The Example~\ref{ex:perturbed} presents a different argument  for the same
  phenomenon.

  In contrast, for exact forms
  $\omega = d\alpha$,
  given a function $H$,
  we can define a vector field  $X$ by
  $i_X\omega = d H + \sigma \alpha$.

  Assuming that the flow of $X$ exists\footnote{This is not obvious in unbounded manifolds since $\alpha$ may be unbounded (see Section~\ref{sec:unbounded_action}) and the standard Cauchy-Picard theorem
  does not apply. Indeed, there are
  examples of action forms with solutions that reach infinity
  in arbitrarily short times.},  the time-$t$ map of the flow   is conformally symplectic for $\omega$
  with a conformal factor $ \exp(\sigma t)$.

  Hence, if the hypothesis of exactness in Lemma~\ref{lem:obstruction}
  does not hold, many of the conclusions fail.
\end{rem}

\begin{rem} \label{rem:homotopic}
Assume that we consider a  cohomology theory that is homotopy invariant.
If $f$ is homotopic to the identity, (which is implied by $f$ being
the time-$1$ map of a time-dependent
flow or by  $f$  being  $\C^1$-close to identity)
then  the induced maps $f^\#$ on $1$ and $2$-cohomology are the identity.
If the conformal factor $\eta\neq 1$,
then $\eta$ is not an eigenvalue of $f^{\#1}$, $f^{\#2}$ and
Lemmas~\ref{generating} and  ~\ref{generating2} (i) apply.
In this case,  we recover \cite[Proposition 9] {ArnaudF24}.

Time-$1$ maps of conformally symplectic vector fields are exact because
there are explicit formulas for the primitive function  \cite{ArnaudF24}.
\end{rem}

\subsection{Topological obstructions to conformal factors}
\label{sec:ArnaudF}
The composition of two conformally symplectic maps with conformal factors $\eta_1$, $\eta_2$, is a conformally symplectic map  with conformal factor  $\eta_1\cdot\eta_2$. Therefore the set of conformal factors forms a multiplicative subgroup $\mathcal{R} \subseteq {\R}_+^*$. The topology of the underlying manifold $M$
presents obstructions to the possible $\mathcal{R}$'s (for instance,  when $M$ is compact $\mathcal{R}=\{1\}$).
Also, conditions on the conformal factor $\eta$ imply conditions on the symplectic structure of $M$.

\medskip

In this section we will address the following question formulated in  \cite[p. 160]{ArnaudF24}:

\begin{center}{\em ``Can $\mathcal{R}$ be strictly between $\{1\}$ and ${\R}_+^*$?''}
\end{center}

We will not attempt a general setting as in \cite{ArnaudF24}, rather  we are
just presenting  some examples in a concrete
manifold: the Cartesian product of a torus and an Euclidean space.

For the obstructions we present,
one also needs to specify a cohomology theory.
The  manifold considered is a deformation retract of the torus, and many
homology and cohomology theories agree.
But there are others, such as Borel-Moore cohomology (see Remark~\ref{rem:Borel_Moore})
that give different answers.
We will use the standard De Rham (co)homology.

Note also that the definition of $\mathcal{R}$ requires specifying
a set of allowed diffeomorphisms (e.g., specifying the growth properties
at $\infty$).
We will
consider diffeomorphisms with uniformly
bounded derivatives and bounded forms.

\begin{prop}\label{prop:ArnaudF24}
 There exists a  symplectic manifold $(M, \omega)$ such that the set \[\mathcal{R} =
\{\eta \in \real_+ \,|\, \exists f:M\rightarrow M \quad \textrm{ s.t. } f^*\omega = \eta\omega\}\] satisfies:
  \begin{itemize}
  \item
    $\mathcal{R}$ contains a number different from $1$;
  \item
    $\mathcal{R}$  consists only of algebraic
    numbers of degree $d(d-1)/2$ that are products of
    two algebraic numbers of degree $d$.
  \end{itemize}
Hence for $(M, \omega)$,  $\mathcal{R}$ is neither $\{1\}$ nor $\mathbb{R}_+$.
\end{prop}

In what follows, we describe several constructions that lead to a proof of Proposition \ref{prop:ArnaudF24}.
The symplectic manifold $(M, \omega)$ claimed in this proposition
  is explicitly constructed in  Example~\ref{concrete}.

  We will show, using cohomology, that all elements of $\mathcal{R}$
  have to satisfy some conditions. Then, we will provide examples
  showing that for any number $\eta$ that satisfies the above restrictions,
  there is a diffeomorphism $f$ that is conformally symplectic for $\omega$
  with conformal factor $\eta$.

\subsubsection{Computation of $f^{\#2}$ on some manifolds}
The matrix elements of $f^{\#2}$ in the standard De Rham cohomology
can be computed explicitly
when $M = \R^d \times \T^d$, where $\T=\R/\Z$ and $d\ge 2$.  We denote the coordinates
on $M$
by $(I,\theta)$. (Later, in  Remark~\ref{rem:Borel_Moore},
we will present results for Borel-Moore  theory).

We note that $\tilde f$, the lift of $f$ to
the universal cover has to satisfy for some $A \in GL(d, \Z)$,
\begin{equation}
\label{lift}
\tilde f(I,\theta + e ) =
\tilde f(I, \theta) + (0, A e), \quad \forall e \in \Z^d .
\end{equation}
Since $A$ is an operator over integers, it is a topological invariant, as small perturbations of $f$ cannot change $A$.  We will
compute $f^{\#2}$ in terms of $A$.

Let:
\[
\{ \sigma_{ij}\}_{1\le i < j \le d}= d \theta_i \wedge d \theta_j
\]
be a basis for the $2$-dimensional de Rham cohomology of $M$. The dimension of the $2$-cohomology
is $d(d-1)/2$.

For $1 \le k \le l \le d$,  we define the $2$-cell
$\gamma_{k,l}(t,s) \subset M$, $(t,s) \in [0,1]^2$,
 by setting $\theta_k = t$, $\theta_l= s$ (mod $1$),
and all the other coordinates in $M$ to zero; that is,
$\gamma_{k,l}$ is a $2$-torus embedded in $M$.

We have:
\[
\int_{\gamma_{kl}}\sigma_{ij} = \delta^i_k \delta^j_l
\]
where $\delta$ is the Kronecker symbol.
Therefore,  we can compute the matrix
elements of $f^{\#2}$ by computing
$\int_{\gamma_{kl}}f^*\sigma_{ij} $.

The following result will be proved next by a direct calculation
as well as a more conceptual argument.

\begin{lem}
  \label{prop:coefficients}
  With the notations in \eqref{lift}, we have:
  \begin{equation}
    \label{coefficients}
    \int_{\gamma_{kl}}f^*\sigma_{ij} = -A_{il} A_{jk} + A_{ik}A_{jl} .
  \end{equation}
\end{lem}
Therefore, in this basis of cohomology, all the coefficients
of the matrix of $f^{\#2}$ are integers and the eigenvalues of $f^{\#2}$
are algebraic numbers of degree equal to the dimension of
$H^2(M)$. In our case, the degree is  $d(d-1)/2$.

\begin{rem}
Since $f^{\#2}:H^2(M)\rightarrow H^2(M)$ is the wedge product of $f^{\#1}: H^1(M) \rightarrow H^1(M)$
with itself, it follows that the eigenvalues of $f^{\#2}$
are the product of two eigenvalues of $f^{\#1}$,
which are algebraic numbers of degree $d$.

In our case,  any conformal factor will be
the product of two algebraic numbers of degree $d$.
This is related to a question raised in \cite[p. 165]{ArnaudF24},
where a very similar  phenomenon is observed in some examples.
\end{rem}

As a corollary, we have:
\begin{cor}
  Let  $M = \R^d \times \torus^d$.

  Assume there is a symplectic form $\omega$ and a conformally
  symplectic diffeomorphism $f$ with a conformal factor $\eta$ that is not
  the product of two algebraic numbers of degree $d$

  Then, $\omega$ is exact.
\end{cor}

\begin{proof}[Proof of Lemma \ref{prop:coefficients}]
To prove \eqref{coefficients} let
\[
\mathcal{I} \equiv
\int_{\gamma_{kl}}f^*\sigma_{ij} =
\int_{\partial f(\gamma_{kl})} \theta_i d\theta_j .
\]
where the second integral is interpreted in the universal cover so
that we can use the variable $\theta$.

Introduce the notation:
\def\p{ {\varphi}}
\[
\p_i(t,s) = f_{\theta_i}(\gamma_{kl}(t,s)),
\]
where the subindex on $f$ indicates the
$\theta_i$-component of $f(\gamma_{kl}(t,s))$.
The periodicity conditions
\eqref{lift} give:
\begin{equation}
  \label{lift2}
\begin{split}
  & \p_i(1,s) - \p_i(0,s) = A_{ik} ,\\
  & \p_i(t,1) - \p_i(t,0) = A_{il} ,\\
  & \partial_2  \p_i(1,s) - \partial_2 \p_i(0,s)  = 0 ,\\
  & \partial_1  \p_i(t,1) -  \partial_1\p_i(t,0)  = 0 ,
\end{split}
\end{equation}
as well as analogous formulas for $j$ taking the place of $i$.

We have $\partial f(\gamma_{kl})= f(\partial \gamma_{kl})$ consists of
four segments and, using \eqref{lift2},

\[
\begin{split}
\mathcal{I} =
  &\int_0^1 \p_i(t,0) \partial_1 \p_j(t,0) \, dt
   +\int_0^1 \p_i(1,s) \partial_2 \p_j(1,s) \,ds \\
    &-\int_0^1 \p_i(t,1) \partial_1 \p_j(t,1) \, dt
  - \int_0^1 \p_i(0,s) \partial_2 \p_j(0,s) \,ds \\
= &\int_0^1 \p_i(t,0) \partial_1 \p_j(t,0) \, dt
   +\int_0^1 \p_i(1,s) \partial_2 \p_j(0,s) \,ds\\
           &-\int_0^1 \p_i(t,1) \partial_1 \p_j(t,0) \, dt
- \int_0^1 \p_i(0,s) \partial_2 \p_j(0,s) \,ds \\
= &\int_0^1 [ \p_i(t,0) - \p_i(t,1)] \partial_1 \p_j(t,0) \, dt +\int_0^1[\p_i(1,s) - \p_i(0,s)] \partial_2 \p_j(0,s) \, ds \\
= &\int_0^1 - A_{il} \partial_1 \p_j(t,0) \, dt +\int_0^1A_{ik}  \partial_2 \p_j(0,s) \, ds \\
= & -A_{il} [\p_j(1,0) - \p_j(0,0)] + A_{ik} [\p_j(0,1) - \p_j(0,0)] \\
= & -A_{il} A_{jk} + A_{ik}A_{jl} .
\end{split}
\]

\end{proof}
  Now, we present a more conceptual (but more sophisticated)  second
proof of Lemma~\ref{prop:coefficients}.

\begin{proof}[Alternative proof of Lemma \ref{prop:coefficients}]
  Observe that we can define a homotopy
  in the space of differentiable maps
  connecting $F^1(I,\theta) = f(I,\theta)$
  and $F^0(I,\theta) = (f_I (I, \theta), A \theta)$.

  Since the action on de Rham cohomology remains constant
  under a homotopy, we obtain that the action of $F^1$ on 2-cohomology
  is the same as that of $F^0$. The latter is just $A^{\wedge 2}$
 (the  wedge product \footnote{We recall that $A^{\wedge 2}$ is defined by
  $A^{\wedge 2}(\alpha \wedge \beta) = (A \alpha) \wedge (A \beta)$ for all
  vectors $\alpha, \beta$} of $A$  with itself),
  which agrees with the direct calculation.

  We work in the lift and we have
  \[
  \begin{split}
    &\tilde F^1(I,\theta)= (\tilde f_I(I,\theta), \tilde f_\theta(I, \theta) ) ,\\
    &\tilde F^0(I,\theta) = (\tilde f_I(I,\theta), A \theta) .
  \end{split}
  \]
  where the subindex  under $f$ indicates taking the component.

  We have for all $e \in \Z^d$,
  \[
  \begin{split}
    & \tilde f_\theta(I, \theta + e) =
    \tilde f_\theta(I, \theta) + Ae ,\\
    & \tilde f_I(I, \theta + e) = \tilde f_I(I, \theta) .
  \end{split}
  \]

  We set for $t \in [0,1]$,
  \[
  \tilde F^t(I,\theta) = (\tilde f_I(I,\theta),  \tilde f^t_\theta (I,\theta)) ,
  \]
  with
  \[
  \tilde f^t_\theta(I, \theta) = t \tilde f_\theta (I,\theta) + (1-t)A e .
  \]

  Note that $\tilde F^t$ remains uniformly differentiable if $\tilde f$ is.

  The function $  \tilde f^t_\theta(I, \theta)$ satisfies
  \[
  \tilde f^t_\theta(I, \theta + e ) =
  \tilde f^t_\theta(I, \theta) + Ae .
  \]

  Therefore, $\tilde F^t$ is the lift of a function on the manifold.
  \end{proof}

\subsubsection{A concrete example}

In the following example, we construct an explicit
conformally symplectic map with a non-exact symplectic form.
Note how it is important that  the conformal factor has to be chosen to be an
eigenvalue of the action in cohomology and, therefore be
an algebraic number.

  \begin{ex}
  \label{concrete}
  Consider $M = \R^d \times \T^d$.

  Let $A \in SL(d, \R) $ be such that it has a complete
  set of eigenvalues/eigenvectors  $\lambda_i, v_i$, some of them real
  and whose product is not $1$ (this is easy to arrange  when $d \ge 3$).

  For a pair of different real eigenvalues of $A$, $\lambda_i$, $\lambda_j$
  denote:
  \[
  \begin{split}
  & \omega_0 = \sum_k dI_k \wedge d\theta_k ,\\
  & \omega_1 = (v_i \cdot d\theta) \wedge (v_j \cdot d\theta) .
\end{split}
  \]
  Denote $\eta_A= \lambda_i \lambda_j$,
  which we assume is not $1$. The number $\eta_A$ is
  an eigenvalue of $A^{\wedge 2}$ which is the same as the action of
  $A$ in $H^2(\T^d) = H^2( M)$.  The dimension of $H^2(\T^d)$ is
  $\ell = d(d-1)/2$.

  We have the following
  elementary facts:
  \begin{itemize}
  \item
    $A^* \omega_1 = \eta_A \omega_1$.
 \item
For $|\eps| \ll 1$,
$\omega = \omega_0 + \eps \omega_1$, is not degenerate.
\item
  For $\eps \ne 0$, $\omega$ is not exact.
\end{itemize}

  We define
  \begin{equation}
\label{fdefined}
  f_A(I,\theta) = ( \eta_A A^{-t } I, A \theta).
  \end{equation}
where  $A^{-t} = (A^t)^{-1}$ is the inverse of the transpose.

  We have that $f_A ^*\omega_0 = \eta_A \omega_0$.
  Similarly, $f^*_A\omega_1 = \eta_A \omega_1$.

  Hence
  \[
  f^*_A \omega = \eta_A \omega .
  \]
\end{ex}

Example~\ref{concrete} depends on the choice
of an automorphism $A$ of the torus.  To emphasize
this we denote the $f$ in \eqref{fdefined}
by $f_{\mid A}$.
Note that in the definition of $f$,
the factor $\eta$ is chosen depending on $A$,
so it will be denoted also $\eta_A$.

\begin{proof}[Proof of Proposition \ref{prop:ArnaudF24}]
 Consider $(M,\omega)$ as in Example \ref{concrete}, where $\omega = \omega_0 + \eps \omega_1$ with $\eps\neq0$.
 Since $\omega$ is not exact, for any conformally symplectic map $f$ on $M$, the symplectic factor $\eta$ is an eigenvalue of $f^{\# 2}$.

 By Lemma \ref{prop:coefficients}, $\eta$ must be an algebraic number of degree $d(d-1)/2$.
 Thus $\mathcal{R}\neq \R+^*$.
 For the conformally symplectic map $f_A$ from Example \ref{concrete}, $\eta_A\neq 1$.
 Thus,  $\mathcal{R}$ is strictly between $\{1\}$ and $\R_+^*$.
\end{proof}

\begin{rem}
  If $\lambda_i$ is an eigenvalue of $A$ of multiplicity $2$,
  the construction of Example~\ref{concrete} leads to
  maps that are conformally symplectic with respect to
  several non-cohomologous symplectic forms.
\end{rem}

\begin{rem}
We will now explore the possibility of
constructing more examples by optimizing the choice of $A$.

The main  observation is that if $B$ is an automorphism of
the torus with simple eigenvalues such that $BA = AB$, then
$A$ and $B$ have a common set of eigenvectors\footnote{
If $Av = \sigma v$, multiplying by $B$ on the left and
commuting, we have $A (Bv)= \sigma(B v)$,
Since we are assuming the space of eigenvectors of $A$ with eigenvalue $\sigma$
is $1$-dimensional, we have that there exists $\nu \in \R$ such that $Bv = \nu v$.}, even though the eigenvalues
may be quite different and, indeed, independent over the rationals.
Now, if $A, B$ have the same eigenvectors, then the maps
$f_A, f_B$ as in \eqref{fdefined}  are conformally
symplectic with factors $\eta_A, \eta_B$ respectively.
Therefore, the set $\cR$ contains the multiplicative group generated
by $\eta_A, \eta_B$. This group could well be dense.

The construction of integer matrices that commute and have simple eigenvalues
is not obvious.
We just point out
that several such  examples, with many
extra properties, appear in \cite[p. 61 ff.]{KatoKN11}
motivated by the theory of Abelian actions on the torus.
For some of these examples,  $\mathcal{R}$ is dense in $\R$.
\end{rem}

\begin{rem}\label{rem:Borel_Moore}
  Example \ref{concrete}  can be also analyzed using the
  Borel-Moore (co)homology theory, which is different from
  the simplicial (co)homology we used so far.

  We recall the Borel-Moore  homology groups of the Euclidean space
  are $H^{BM}_d(\R^d) = \Z$
  and $H^{BM}_k( \R^d) =  \{0\}$ for $k\neq d$. The B-M homology  groups of
  the torus
$H^{BM}_k(\T^d)=\Z^{d\choose{k} }$
  (which is the same as the simplicial  homology).  The  homology
  groups of $\R^d \times \T^d$ can be computed using
  K\"unneth formula.

  Thus, in Example~\ref{concrete}
  the Borel-Moore is different from the usual simplicial
  (co)homology, and the obstructions provided by Borel-Moore to
  $\mathcal{R}$  are all   trivial.

This concrete example that
we have developed is special, but the main
ingredients (finite dimensional cohomology, with some
duality -- via periods -- to homology with integer coefficients)
could hold in greater generality, even if they fail in certain manifolds
(e.g.,  in cylinders with infinitely many handles attached, which
has an infinite dimensional homology).

\begin{conjecture}\label{rem:duality1}

We expect that for \emph{``many''}  manifolds
with finite dimensional (co)homology theories we have:
\begin{center}
\em $\mathcal{R}$ consists of algebraic numbers.
\end{center}

\end{conjecture}

We hope that the precise hypotheses needed could be
well known to experts.
\end{rem}

\section{Examples}\label{sec:examples}

In this section, we present several examples of systems
that illuminate several of the issues addressed by the theorems.
We also note that some of the further  results  depend on
constructing examples.

\medskip

\subsection{An example of a conformally symplectic map that is exact for one action form but not for another}
\label{sec:standard_map}
The next example stresses that the exactness property of a conformally symplectic map may depend on the choice of action form.

\begin{ex}
\label{ex:standard_map}
A paradigmatic example of conformally symplectic system
is the dissipative standard map with two parameters $\eps, \mu$
acting on $\R \times \T$
\begin{equation} \label{DSM}
f(I, \theta) = ( \eta I + \mu + \eps V'(\theta),
\theta +  \eta I + \mu + \eps V'(\theta) )
\end{equation}

The symplectic form considered in \eqref{DSM}   is $\omega=d I \wedge d \theta$.
This is an exact form and we can take the action form $\tilde \alpha_\sigma = (I   +\sigma) d \theta $ where $\sigma$ is any constant.

We have
\[
\begin{split}
f^* &  \tilde \alpha_\sigma   =
(\eta I + \mu + \eps V'(\theta) + \sigma)
(d \theta + \eta d I + \eps V''(\theta)
d \theta )  \\
&= \eta I d\theta  + \eta ^2 I dI  + \eta \eps I   V''(\theta) d \theta
+ ( \mu + \sigma ) d \theta + (\mu + \sigma) \eta dI  \\
&+ (\mu +\sigma)\eps V''(\theta) d \theta
+ \eps V'(\theta) d \theta + \eps \eta V'(\theta) d I
+ \eps ^2 V'(\theta) V''(\theta) d \theta  \\
&= \eta I d\theta + ( \mu + \sigma)  d \theta \\
& + d\left( \frac{\eta^2}{2} I^2 + \eps \eta I V'(\theta) +\eta( \mu + \sigma)
I + \eps (\mu+\sigma) V'(\theta)+ \eps V(\theta)+\eps ^2 \frac{(V'(\theta))^2}{2} \right) .
\end{split}
\]

Note that the term $d \theta$ on the  right hand side
prevents exactness \footnote{The integral of
$d \theta$ over a  non-contractible loop in
the cylinder is not zero; $\theta$ cannot be made into a
continuous variable over the manifold.},
so the map \eqref{DSM} is  exact if and only if  $\eta \sigma = \mu + \sigma$.
Hence, if we choose $\sigma_* = \mu/(\eta -1)$,
we have  that the mapping is exact for  the action form
$\tilde \alpha_{\sigma_*}$, and the primitive function is

\[
P_{\sigma_*}  =  \frac{\eta^2}{2} I^2 + \eps \eta I V'(\theta) +\frac{\eta^2\mu}{\eta-1}
I +\frac{ \eps\eta\mu}{\eta-1}  V'(\theta)+ \eps V(\theta)+\eps ^2 \frac{(V'(\theta))^2}{2}  + C .
\]

In particular, if $\sigma=0$ and the action form is the Liouville form $\tilde \alpha_0=I d\theta$,
the mapping is exact  if and only if $\mu=0$.
\end{ex}

When $\eta=1$ and $\mu=0$ \eqref{DSM}  becomes  the conservative standard map.
The drift parameter $\mu$  is fundamental in applications of  the KAM theorem to conformally symplectic
systems; one needs to adjust a drift parameter $\mu$ to find an invariant torus
of preassigned frequency (see \cite{CallejaCL13}).

When $\eta = 1$, $\eps = 0$, \eqref{DSM} becomes  the integrable area preserving twist map
whose phase space
is foliated by quasi-periodic orbits. For $\eta < 1$, $\eps = 0$,
there is  only one quasi-periodic orbit, which illustrates that quasi-periodic orbits may disappear
under arbitrarily small dissipation.

Note that the limit $\eta \to 1$ is singular. If we fix $\mu \ne 0$,
the $\sigma_*(\eta)$ that makes \eqref{DSM} exact, goes to $\infty$ as $\eta$
converges to $1$.

\medskip
An interesting  variant of \eqref{DSM} (Fr\"oeschl\'e map)
is to take
the variables $I,\theta$ to be $n$-dimensional
and the drift parameter $\mu$ to be also an $n$-dimensional
vector. In such a case, $V'$ means  the gradient.
The calculations presented here show that when $\eta \ne 1$,
for any $\mu$ we can find a unique $n$-vector $\sigma_*(\mu)$
so that the map is exact for $I\cdot d\theta + \sigma(\mu) \cdot d\theta$.

\subsection{Symplectic systems that can not be perturbed into conformally symplectic ones}

The second example  shows  some symplectic maps
that cannot be perturbed into conformally symplectic ones.
This  is due to global properties of the manifold.
If this was a model of a mechanical system,
it would show that if one adds friction, the friction cannot be
just proportional to the velocity due to the global shape of
the manifold.

\begin{ex}\label{ex:perturbed}
  Consider the phase space $ M = \torus^{2n}  \times \real^{m} \times \real^m$
    (with coordinates $(\theta,x,y)$),
    endowed with the symplectic form
    $\omega = \sum_{i=1}^n d\theta_i \wedge d\theta_{i+n} + \sum_{i = 1}^m dx_i \wedge dy_i$.

Define the dynamics on $M$ by
\[
f(\theta, x,y)=
(A\theta, \lambda x ,\lambda^{-1} y )
\]
where $A \in Sp(2n, \Z) $ is a symplectic matrix whose spectrum is contained  between $1/\mu$ and $\mu$, for some $\mu \ge 1$, and  $0 < \lambda < 1 $ is a sufficiently  small number so that $\lambda \mu <1$ and therefore  the set
\[
\Lambda:=\{ (\theta, 0, 0)\,|\, \theta \in \torus^{2n}\} \subset M
\]
is a normally hyperbolic invariant manifold (see Definition \ref{def:nhim}).

The map $f$ is symplectic for $\omega$.

There is no conformally symplectic $\C^1$-small perturbation of
the map $f$ with a conformal factor different from $1$.
\end{ex}

\begin{proof}
If such a perturbation existed, by the theory of NHIM, there
should be an invariant manifold $\C^1$-close to $\Lambda\sim\torus^{2n} $ with rates $\lambda_\pm$, $\mu_\pm$ close to $\lambda$ and $\mu$, and conformally symplectic factor close to $1$.
Then, for small enough perturbations, conditions \eqref{eq:rates0} and  \eqref{eqn:conformal_rates} would be satisfied and, applying  Theorem \ref{thm:main1}, the persistent  NHIM would be
symplectic and the dynamics on  it would be conformally symplectic
Therefore, the NHIM would have infinite volume.

On the other hand the manifold would be $\C^1$-close to
$\Lambda\sim\torus^{2n} $ and hence have finite volume, which
contradicts the fact that the dynamics on it is conformally
symplectic.
%
%
%
\end{proof}

\subsection{Minimal set of constraints  on rates for the existence of a symplectic NHIM}

From \eqref{eq:rates0}, \eqref{eq:ratesplusminus}, \eqref{eqn:pairing_rules}, we see that the minimal set of constraints  for the existence of a  symplectic  NHIM for a conformally symplectic map, in terms of the optimal rates  \eqref{eqn:optimal_rates} and the conformal factor $\eta$, is
\begin{equation}\label{eqn:minimal_constraints}
\begin{split}
 \lambda_+^*\mu_-^*< 1 &\textrm{ and } \lambda_-^*\mu_+^*<1 \\
 \mu_+^*\mu_-^*\ge& 1\\
 \frac{\lambda_+^*}{\lambda_-^*}=\eta &\textrm { and } \frac{\mu_+^*}{\mu_-^*}=\eta .
 \end{split}
\end{equation}

The following example shows that  there are no other constraints besides \eqref{eqn:minimal_constraints} for the existence of a symplectic NHIM for a conformally symplectic map.

\begin{ex}[Flexibility]\label{lem:flexibility}
Suppose that $0<\eta\le 1 $ and we are given a set of positive real numbers $\lambda_+^*$, $\lambda_-^*$, $\mu_+^*$ $\mu_-^*$ satisfying \eqref{eqn:minimal_constraints}.
Then, there exists a conformally symplectic map $f$,  with symplectic factor $\eta$, possessing a symplectic NHIM $\Lambda$,
such that the corresponding rates \eqref{eq:rates0}  are $\lambda_+^*$, $\lambda_-^*$, $\mu_+^*$ $\mu_-^*$, respectively.

We denote by $\textrm{Diag}(a_1,\ldots,a_k)$   a diagonal matrix of entries $a_1,\ldots, a_k$.

Choose  $0<\lambda_n<\ldots <\lambda_1:=\lambda_+^*$ and take
\[
A=\textrm{Diag}(\lambda_1,\ldots,\lambda_n).
\]

Using \eqref{eqn:minimal_constraints},
the condition $\mu_+^*\mu_-^* \ge 1$ in \eqref{eqn:minimal_constraints} can be replaced by the equivalent condition
\begin{equation}\label{eq:hypothesis}
 (\mu^*_+)^2\ge \eta .
\end{equation}
Choose:
$0<\mu_d\le\ldots \le \mu_1:=\mu_+^*$ subject to the following condition
\begin{equation}\label{eq:hypothesis2}
\mu_d\ge \frac{\eta}{\mu_1} = \frac{\eta}{\mu^*_+} .
\end{equation}
Choosing $\mu_d$ as in \eqref{eq:hypothesis2} is possible since  by  \eqref{eq:hypothesis} we have $\mu_1\ge \frac{\eta}{\mu_1}$.
Then take
\[
M=\textrm{Diag}(\mu_1,\ldots,\mu_d ).
\]

Consider the  symplectic form  on $\real^{2n+2d}$:
\[
 \omega=\sum_{i=1}^{n} dy_i\wedge dx_i + \sum_{i=1}^{d} dv_i\wedge du_i .
 \]
Define  the conformally symplectic  map:
\begin{equation}\label{eqn:simple_1}
\begin{split}
f(x,y,u,v)=\left (A x,\eta A^{-1} y, M u, \eta M^{-1} v\right).
 \end{split}
\end{equation}
We have
\[
\begin{split}
\Lambda &=\{(x,y,u,v)\,|\, x=y=0\} \\
E^s_z  &=\{(x,y,u,v)\,|\, y=0, \, u=0, \, v=0\}, \ \forall z=(0,0,u_0,v_0)\in \Lambda \\
E^u_z  &=\{(x,y,u,v)\,|\, x=0, \, u=0, \, v=0\}, \ \forall z=(0,0,u_0,v_0)\in \Lambda \\
T_z \Lambda  &=\{(x,y,u,v)\,|\, x=y=0\}, \ \forall  z=(0,0,u_0,v_0)\in \Lambda .
\end{split}
\]

We now show  that  $\Lambda$ is a  symplectic  NHIM with corresponding optimal rates   $\lambda_+^*=\lambda_1$, $\lambda_-^*=\frac{\lambda_1}{\eta}$, $\mu_+^*=\mu_1$, $\mu_-^*=\frac{\mu_1}{\eta}$. For  $z\in \Lambda$ and $n>0$
\[
\begin{split}
\|Df^n(z)w\| & \le C(\lambda_1)^n \|w\|= C(\lambda_+^*)^n \|w\|, \ \text{for} \ w=(x,0,0,0) \in E^s_z ,\\
\|Df^{-n}(z)w\| & \le C(\frac{\lambda_1}{\eta})^n \|w\|= C(\frac{\lambda_+^*}{\eta})^n \|w\|, \ \text{for} \ w=(0,y,0,0) \in E^u_z .
\end{split}
\]
Moreover, if $w=(0,0,u,v) \in T_z\Lambda$, since $\frac{\eta}{\mu_d}\le\mu_1$ by \eqref{eq:hypothesis} we obtain:
\[
\begin{split}
\|Df^{n}(z)w\| & \le C\left(\max{\left(\mu_1, \frac{\eta}{\mu_d}\right)}\right)^n \|w\|=
C(\mu_+^*)^n \|w\|, \\
\|Df^{-n}(z)w\|  & \le C\left(\max{\left(\frac{1}{\mu_d}, \frac{\mu_1}{\eta}\right)}\right)^n \|w\|
= C\left(\frac{\mu_+^*}{\eta}\right)^n \|w\|.
\end{split}
\]
\end{ex}

\subsection{ Degenerate forms in a manifold and no paring rules}

Now, we present an example illustrating the degeneracy of the forms in invariant manifolds that do
not satisfy the pairing rules.

\begin{ex}
  \label{intermediatenew}
  Fix the conformal factor $\eta > 0$.
  Consider numbers \[0 < a < b < c < d  < 1 <d^{-1} \eta < c^{-1} \eta < b^{-1} \eta < a^{-1} \eta ,\]
  and the map on $\R^8$
  \[
  f(x_1,\ldots,x_4,y_1,\ldots,y_4)
  =  (a x_1,b x_2, c x_3, d x_4,  a^{-1} \eta y_1, b^{-1} \eta y_2, c^{-1} \eta y_3, d^{-1}\eta  y_4 ).
  \]
 The map $f$ is conformally symplectic of factor $\eta$
  for the symplectic form
  \[
  \omega = d y_1 \wedge dx_1 + d y_2 \wedge dx_2 + d y_3\wedge d x_3 +
 dy_4 \wedge d x_4 .
  \]
 We consider the NHIM given by
\[
\Lambda_0 = \{(0,0,0,x_4,0,0,0,y_4)\}
\]
and
$\omega_0 = \omega|_{\Lambda_0}= d y_4 \wedge d x_4 $
is non-degenerate on $\Lambda_0$.
We have that
\[f_{\mid\Lambda_0}(0,0,0,x_4,0,0,0,y_4)=(0,0,0,dx_4,0,0,0,d\eta^{-1}y_4). \]
The NHIM $\Lambda_0$ has $5$-dimensional stable and unstable manifolds given by
\[W^s_\Lambda=\{(x_1,x_2,x_3,x_4,0,0,0,y_4)\} \textrm{ and } W^u_\Lambda=\{(0,0,0,0,x_4,y_1,y_2,y_3,y_4)\}.\]
The optimal rates
are $\lambda_+^*=c$,  $\lambda_-^*=\eta^{-1}c$, $\mu_+^*=\eta d^{-1}$, and $\mu_-^*=d^{-1}$.
Note that $\Lambda_0$ satisfies the pairing rules.

We also consider the NHIM
\[
\Lambda = \{ (0,x_2,x_3,x_4,0, 0,0,y_4)\}.\]
We have that
\[f_{\mid\Lambda}(0,x_2,x_3,x_4,0,0,0,y_4)=(0,b x_2, c x_3,dx_4,0,0,0,d\eta^{-1}y_4). \]
The manifold $\Lambda$  has $5$-dimensional stable and $7$-dimensional unstable manifolds given by
\[W^s_\Lambda=\{(x_1,x_2,x_3,x_4,0,0,0,y_4)\} \textrm{ and } W^u_\Lambda=\{(0,x_2,x_3,x_4,y_1,y_2,y_3,y_4)\}.\]

The corresponding optimal rates  are $\lambda_+^*=a$,  $\lambda_-^*=\eta^{-1}c$, $\mu_+^*= \eta d^{-1}$, and $\mu_-^*=b^{-1}$, and they
do not satisfy the pairing rules.
The form $\omega$ is degenerate on $\Lambda$,  as we have $\omega_{\mid\Lambda} = \omega_0$.
Note that $\Lambda$ is contained in the stable manifold of $\Lambda_0$.

Even if $\Lambda$ is not symplectic, we can identify presymplectic forms in $\Lambda$.

Clearly $\Lambda_0 \subset \Lambda $ and $\Lambda_0$ is NHIM
for the dynamical system $f_{\mid\Lambda}$.

For every $(x_4, y_4)$
we can write the $2$-dimensional leaf:
\[
\L_{(x_4,y_4)} = \{  (0,x_2,x_3,x_4, 0, 0,0,y_4)\, |\, x_2, x_3 \in \R\}.
\]

As $(x_4, y_4)$ range over $\Lambda_0$, the leaves $\L_{(x_4, y_4)}$
foliate $\Lambda$.

Also $\omega_{\mid \L_{(x_4,y_4)}} = 0$. So that the foliation of $\Lambda$
given by the leaves $\L$ is the foliation integrating the
kernel of $\omega_{\mid\Lambda}$ and $\Lambda_0$ is the
symplectic quotient.

From the dynamical point of view,
we can think of $\L_{(x_4,y_4)}$
as the stable  manifold of $(0,0,0,x_4,0,0,0,y_4)$
in $f_{\mid\Lambda}$.

However, considered as a dynamical system in $\R^8$,
$\L_{(x_4,y_4)}$ is only a weak stable manifold for
$(0,0,0,x_4,0,0,0,y_4)$. From the point of view of weak
stable manifolds, integrability of the foliation is surprising (see \cite{JiangPL1995}).
\end{ex}

\subsection{Unbounded forms and no pairing rules}

In the next example we show that the standing assumption \eqref{eqn:bounded_omega} on the boundedness of the symplectic form is essential for the pairing rules. 

\begin{ex}\label{ex:unbounded_symplectic_form}
Let $M= \real  \times \T^1 \times\real ^2$ be a manifold, and let the (unbounded) symplectic form on $M$ be
\[\omega =  e^I  dI \wedge d \theta + dy \wedge dx, \textrm{ for } (I,\theta,x,y)\in  \real  \times \T^1 \times\real ^2.\]
For  $t>0$ define the map
\[ f(I, \theta, y,x) =    \left(I +t, \theta,  10 e^t y, x/10\right).\]
We have \[f^* \omega = e^t \omega\]
that is, $f$ is conformally symplectic with factor $\eta=e^t$.

The set \[\Lambda=\{ (I, \theta, 0, 0)\,|\, ( I,\theta) \in  \real  \times \T^1\}\]
is a NHIM and is symplectic.
The optimal rates are:
\[\mu_+^* = \mu_-^* = 1,\,\lambda_+^* = 1/10, \lambda_-^*  = e^{-t} /10.\]
The pairing rules \eqref{eqn:pairing_rules} do not hold in this example, since
$\frac{\mu_+^*}{\mu_-^*}=1\neq \eta$.
\end{ex}

\section{Vanishing lemmas}
\label{sec:vanishing_lemmas}
This section is devoted to formulating and proving vanishing lemmas, which
are an important ingredient in the proofs of the main results of Section~\ref{sec:mainresults}.

We will show that, under  assumptions on the rates on convergence of the differential of the map and
the conformal factor, several blocks of the symplectic form
in the decomposition corresponding to the invariant spaces have to vanish.
The idea is to use the invariance equation \eqref{eqn:conformally_symplectic} for
sufficiently high iterates.

This idea is very general and applies in many other contexts   and other types  of rates,
for example Sacker-Sell spectrum and Lyapunov exponents,  and for other geometries such as
locally conformal systems.

We call attention to the fact that  the proofs of the lemmas in this section
do not use that the form $\omega$ is closed, and only Lemma  \ref{cor:pairing} uses that the form  $\omega$ is non-degenerate in the tangent bundle of the considered submanifold.
We also include in this section Proposition \ref{lem:degeneracy}, which gives results about isotropic manifolds.
The proof of Proposition \ref{lem:degeneracy} requires that the form is symplectic and is an easy consequence of Theorems \ref{thm:main1} and \ref{thm:main2} whose proof appears later in Section~\ref{sec:proofisotropic}.

\subsection{A basic inequality}
\label{sec:basic}

Most of the vanishing lemmas in this section rely on the following  elementary result:

\begin{lem}\label{lem:iterates}
Let $f$ be  a conformally symplectic  diffeomorphism $f:M\to M$ with conformal factor $0<\eta$ as in  \eqref{eqn:conformally_symplectic}.
Then for all $x\in M$, $n\in\mathbb{Z}$, and $u,v \in T_xM$  we have
\begin{equation}\label{iterates}
\begin{split}
|\omega(x)(u,v)| \le   \eta^{-n}\|\omega (f^n(x))\|\|Df^n(x)u\|\|Df^n(x)v \|.
\end{split}
\end{equation}
\end{lem}

\begin{proof}
Since $f$ is conformally symplectic we have
\begin{equation*}
\begin{split}
\omega (f^n(x))(Df^n(x)u,Df^n(x)v)= \eta^n\omega(x)(u,v) ,
\end{split}
\end{equation*}
so
\begin{equation}\label{eq:basicequatily}
\begin{split}
\omega(x)(u,v) =\frac{1}{\eta^n}\omega (f^n(x))(Df^n(x)u,Df^n(x)v),
\end{split}
\end{equation}
which yields \eqref{iterates}.
\end{proof}

\subsection{General vanishing lemmas}
In this section  we will give two general vanishing lemmas for a conformally symplectic diffeomorphism.
These lemmas will be the main ingredients of the
proof of the pairing rules for a symplectic normally hyperbolic invariant manifold given in Section \ref{sec:pairing_rules}.

\begin{lem}\label{lem:vanishing_pairing}
Let $f$ be  a conformally symplectic  diffeomorphism $f:M\to M$ with conformal factor $0<\eta$ as in  \eqref{eqn:conformally_symplectic}.
Let $L\subseteq M$ be a submanifold invariant under $f$.
Assume that the symplectic form $\omega$ is  uniformly bounded in a neighborhood of $L$.

Take $x\in L$,  and assume that there exist two constants  $C_1=C_1(x),C_2=C_2(x)>0$ and  two vectors $u,v\in T_x L$, such that:

\begin{itemize}
\item [(A1)]
There exists $0<\alpha<1$, such that for all $n\geq 0$
\begin{equation}
\|Df^n(x)u\|\leq C_1\alpha^n\|u\|,
\end{equation}
\item [(A2)]
There exists   $0<\beta$ with $\alpha\beta<\eta$ such that there exists an increasing sequence of positive integers
 $\{n_j\}_{j=1,\ldots,\infty}$ such that
\begin{equation}
\|Df^{n_j}(x)v\|\leq C_2\beta^{n_j}\|v\|, \textrm{ for } j\geq 0.
\end{equation}
\end{itemize}
Then $\omega (x)(u,v)=0$.
\end{lem}

\begin{proof}
Using  that
\begin{equation}\begin{split}
|\omega(f^{n_j}(x))(Df^{n_j}(x)u,Df^{n_j}(x)v)|\le &\|\omega_{\mid L}\| C_1 C_2(\alpha\beta)^{n_j}\|u\|\|v\| ,
\end{split}\end{equation}
\eqref{iterates}  gives:
\[
|\omega(x)(u,v)|\le \|\omega_{\mid L}\| C_1 C_2(\alpha\beta\eta^{-1})^{n_j}\|u\|\|v\|
\]
since $\alpha\beta\eta^{-1}<1$, taking $n_j\to \infty$, we obtain $\omega (x)(u,v)=0$.
\end{proof}

The next result is a converse of Lemma~\ref{lem:vanishing_pairing}.

If we assume that $\omega_{\mid L}$ is non-degenerate, for any $x\in L$, and $0\neq u\in T_xL$ we have $\iota_u(\omega (x))\neq 0$, i.e.,  there exists $v\in T_xL$ such that $\omega(x)(u,v)\neq 0$.
Hence, the hypothesis of the previous lemma have to fail.
If for a point $x$  in an invariant  manifold with
a symplectic form there is a vector decreasing exponentially fast,
there has to be another one growing exponentially with a rate that matches.
This is the key to pairing rules, but one has to fix some details of uniformity.

\begin{lem}\label{cor:pairing}
Let $f$ be  a conformally symplectic  diffeomorphism $f:M\to M$ with conformal factor $0<\eta$ as in  \eqref{eqn:conformally_symplectic}, and
$L\subseteq M$  a submanifold invariant under $f$.

Assume that the symplectic form $\omega$ is  uniformly bounded in a neighborhood of $L$,
and that $\omega_{\mid L}$ is non-degenerate.
Take $x\in L$ and assume that for some $0<\alpha<1$, there exists $u\in T_xL$ that satisfies (A1) of Lemma~\ref{lem:vanishing_pairing}.

Then, for any $\beta$ with $\alpha\beta<\eta$, there exists $v\in T_xL$ that
fails (A2) of Lemma~\ref{lem:vanishing_pairing}
\end{lem}

The negation of (A2) is very strong.
It means that, for the point $x\in L$ and for the vector $v$, we have that  for every $C_2>0$ there are only finitely $n_j$'s such that $\|Df^{n_j}(x)v\|\leq C_2 \beta^{n_j}\|v\|$.
Therefore, there exists $n_0(C_2)$ such that
\[
\|Df^{n}(x)v\|\ge  C_2
\beta ^n \|v\| \ \ \forall  n\ge n_0(C_2).
\]
Increasing  the constant $C_2$, we obtain the previous  inequality for all  $n\ge 0$.
\begin{equation} \label{lowbound}
\|Df^{n}(x)v\|\ge  \tilde C_2(x)
\beta ^n \|v\| \ \ \forall n\ge 0.
\end{equation}

A subtle point is that the $\tilde{C}_2$ appearing in \eqref{lowbound}
may depend on the point $x \in L$  even if the assumptions in  (A1)
hold with uniform constants. The reason is that the failure of (A2)
may happen for different sequences depending on the point.

This will be enough for our purposes in Section~\ref{sec:pairing_rules}
which only need the lower bounds for large enough $n$ and some $x$.

\begin{rem}
  For the experts in Fenichel theory,  we point out that
  the \emph{uniformity lemma} \cite{Fenichel74} allows to go from \eqref{lowbound} to
  bounds with uniform constants. Unfortunately, one of the assumptions of
the uniformity lemma is compactness of the manifold, which is not true in our
setting.
\end{rem}

\subsection{Vanishing lemmas on NHIMs}

This subsection gives a useful vanishing lemma for a NHIM under assumptions \eqref{eqn:conformal_rates} on the rates and the conformal factor, and assuming the form $\omega$ is bounded.
It will be crucial to prove that the NHIM   is symplectic in Section~\ref{sec:proofAmain1}.
\begin{lem}[Infinitesimal Vanishing Lemma]
\label{lem:vanishing}

We take the standing assumptions from Section~\ref{sec:standing}.

Let $x\in\Lambda$.
Then,  we have:
\begin{equation}
\label{eqn:vanishing}
\begin{split}
\mu_+ \lambda_+ \eta^{-1}  < 1 &\implies  \omega(x)(v_t,v_s)= 0, \quad \forall v_t\in T_x\Lambda,
\, \forall v_s\in E_x^s ,
\\
\mu_- \lambda_- \eta  < 1 & \implies  \omega(x)(v_t,v_u)= 0, \quad \forall v_t\in T_x\Lambda,
\, \forall v_u\in E_x^u,
\\
\lambda_+^2 \eta^{-1}<1 & \implies  \omega(x)(v^1_s,v^2_s)=0, \quad \, \forall v^1_s,  v^2_s\in E_x^s,
\\
\lambda_-^2 \eta < 1 & \implies  \omega(x)(v^1_u,v^2_u)=0, \quad  \forall v^1_u , v^2_u\in E_x^u.
\end{split}
\end{equation}
\end{lem}

\begin{proof}

For $v_t\in T_x\Lambda$, $v_s\in E^s_x$, by\eqref{eqn:NHIM} we have
\begin{equation}\begin{split}
\|Df^n(x) v_t\|\leq& D_+\mu_+^n\|v_t\|\textrm { for } n\geq 0,\\
\|Df^n(x) v_s\|\leq& C_+\lambda_+^n\|v_s\|\textrm { for } n\geq 0,
\end{split}
\end{equation}
so, by \eqref{iterates} and \eqref{eqn:bounded_omega},
\[
|\omega(x)(v_t,v_s)| \leq M_\omega  (C_+D_+)(\lambda_+\mu_+\eta^{-1})^n\|v_t\|\|v_s\| \textrm { for } n\geq 0.
\]
As $ \lambda_+\mu_+\eta^{-1} <1$, and $n$ is arbitrary, we obtain $\omega (x)(v_t,v_s)=0$.

Analogously,  from
\begin{equation}\begin{split}
\|Df^n(x) v_t\|\leq& D_-\mu_-^{|n|}\|v_t\|\textrm { for } n\leq 0,\\
\|Df^n(x) v_u\|\leq& C_-\lambda_-^{|n|}\|v_u\|\textrm { for } n\leq 0,
\end{split}
\end{equation}
it follows
\[
|\omega(x)(v_t,v_u) | \leq M_\omega  (C_-D_-)(\lambda_- \mu_- \eta )^{|n|}\|v_t\|v_u\|\| \textrm{ for } n\leq 0
\]
and, since by assumption, $ \lambda_- \mu_- \eta   <1$,
we obtain $\omega(x)(v_t,v_u)=0$.

Similarly
\[
|\omega(x)(v^1_s,v^2_s)| \leq\ M_\omega (C_+)^2(\lambda_+^2 \eta^{-1} )^{n }\|v^1_s\|\|v^2_s\| \textrm{ for } n\geq 0,
\]
and $ \lambda_+^2 \eta^{-1} <1$ imply
$\omega(x)(v^1_s,v^2_s)=0$.

An analogous argument shows that $ \lambda_-^2 \eta   < 1$ implies $\omega(x)(v^1_u,v^2_u)=0$.
\end{proof}

A corollary of Lemma~\ref{lem:vanishing} is that
the manifold $W^{s,u, \textrm{loc}} $ are co-isotropic.

\begin{rem}
\label{rem:estimates_omega}

In the neighborhood $\OO_\rho$ (see \eqref{eqn:OOrho}), it is natural to obtain a system of coordinates
in $W^{s,\mathrm{loc}}_\Lambda$ to a neighborhood of
the zero section of $E^s_\Lambda$.

We could use any system of coordinates whose coordinate is tangent to
the stable bundle. For example the coordinate system in Section \ref{rem:new coordinates}, which has  useful geometric properties.
In a sufficiently small neighborhood, we can trivialize the stable bundle.

We want to express the symplectic form $\omega$ in
the $x,y$ system of coordinates. Note that $(x,0)$ corresponds to
points in $\Lambda$.

Assuming \eqref{eqn:conformal_rates},
the Lemma~\ref{lem:vanishing} gives us that $\omega(x,0)$
  -- the form $\omega$ in  the manifold $\Lambda$, has the representation
  \[
  \omega(x,0) = \begin{pmatrix}
    \omega_{xx}(x)  & 0 \\
    0 & 0
  \end{pmatrix}.
  \]
  Therefore, if $\omega$  is differentiable and satisfies \eqref{doundedomega}
  we have, by the mean value theorem:

\begin{equation}
  \label{eqn:decomposition2}
  \omega(x,s)= \begin{pmatrix}
    \omega_{xx}(x)  & 0 \\
    0 & 0
  \end{pmatrix}
  + e(x,y), \textrm{ with } \|e(x,y) \| \le C M'_\omega \|y\| .
\end{equation}

We will use these estimates in Section~\ref{sec:iteration}.
\end{rem}

\begin{rem}\label{rem:ds=du}
As a consequence of Lemma \ref{lem:vanishing}, we obtain that the stable and unstable bundles have equal dimension, i.e., $d_s=d_u$.
The reason is that relative to $T_\Lambda M\oplus E^s\oplus E^u$ we have
\[\omega=\begin{pmatrix}
      \omega_{xx} & 0 & 0 \\
      0 & 0 & \omega_{y_s y_u} \\
      0 &\omega_{y_uy_s}& 0 \\
    \end{pmatrix}
    \]
and $\omega_{y_sy_u}$, $\omega_{y_uy_s} $ are non-degenerate.
Therefore, $E^s$ is isomorphic to $E^u$.
\end{rem}

\subsection{Vanishing lemmas on stable/unstable manifolds}

The following lemma can be considered as
an analogue of Lemma~\ref{lem:vanishing} for the stable and unstable manifolds
of a NHIM.
It will be used in the proof of part {\rm (B)} of Theorem~ \ref{thm:main2}.

\begin{lem}[Vanishing Lemma]
\label{lem:vanishingmanifolds}
We adopt the standing assumptions from Section \ref{sec:standing}.

Let $y\in W^{s,\mathrm{loc}}_x$ for some $x \in \Lambda$.
Then,  we have:
\begin{equation}
\label{eqn:vanishingstable}
\begin{split}
\mu_+ \lambda_+ \eta^{-1}  < 1 &\implies  \omega(y)(v_t,v_s)= 0, \quad \forall v_t\in T_yW^{s,\mathrm{loc}}_\Lambda,
\, \forall v_s\in T_y W_x^{s,\mathrm{loc}} ,\\
\lambda_+^2 \eta^{-1}<1 & \implies  \omega(y)(v^1_s,v^2_s)=0, \quad \, \forall v^1_s,  v^2_s\in T_yW^{s,\mathrm{loc}}_x.
\end{split}
\end{equation}

Analogously, let $y\in W^{u,\mathrm{loc}}_x$ for some $x \in \Lambda$.
Then,  we have:
\begin{equation}
\label{eqn:vanishingunstable}
\begin{split}
\mu_- \lambda_- \eta  < 1 & \implies  \omega(y)(v_t,v_u)= 0, \quad \forall v_t\in T_yW^{u,\mathrm{loc}}_\Lambda,
\, \forall v_u\in T_yW^{u,\mathrm{loc}}_x,
\\
\lambda_-^2 \eta < 1 & \implies  \omega(x)(v^1_u,v^2_u)=0, \quad  \forall v^1_u , v^2_u\in T_yW^{u,\mathrm{loc}}_x.
\end{split}
\end{equation}
Consequently:
\begin{equation} \label{eqn:vanishing_manifolds}
\begin{split}
\lambda_+^2 \eta^{-1}<1 & \implies W^{s,\mathrm{loc}}_x \textrm{ is isotropic},\ \forall x\in \Lambda,\\
\lambda_- ^2 \eta  < 1 &  \implies W^{u,\mathrm{loc}}_x \textrm{ is isotropic}, \ \forall x\in \Lambda.
\end{split}
\end{equation}
\end{lem}

Note that in \eqref{eqn:vanishingstable}, we have $\lambda_+\mu_+\eta^{-1}<1$ implies $\lambda_+^2\eta^{-1}<1$, since $\lambda_+<\mu_+$.

The conclusions in
the first lines of  \eqref{eqn:vanishingstable},
\eqref{eqn:vanishingunstable} can
be stated geometrically as saying that,
for all $x \in \Lambda$:
\begin{equation}\label{vanishing_geometric}
  \begin{split}
& \forall y \in W^s_x,  v_s  \in T_y W^s_x
\implies i_{v_s} (\omega_{\mid W^s_\Lambda}) = 0 ,\\
    & \forall y \in W^u_x,  v_u \in T_y W^u_x
\implies i_{v_u}(\omega_{\mid W^u_\Lambda}) = 0 .
\end{split}
\end{equation}
In other words, $W^{s,u}_\Lambda$ is presymplectic and the foliation
given by the kernel is  the foliation of strong stable/unstable manifolds.

\begin{proof}[Proof of Lemma \ref{lem:vanishingmanifolds}]
We consider the case of the stable manifold. For points in the unstable manifold we proceed analogously.
Take $y \in  W^{s,\textrm{loc}}_\Lambda $ and $x \in \Lambda$ such that $y\in W^{u,\mathrm{loc}}_x$.

For $v_t\in T_y W^{s,\mathrm{loc}}_\Lambda$, $v_s\in T_y W^{s,\mathrm{loc}}_x$ we have, using the bounds \eqref{eqn:ratesinmanifolds} in Lemma \ref{lem:globalrates}, we have that there exist $D_+, C_+$, such that:
\begin{equation}\begin{split}
\|Df^n(y) v_t\|\leq& D_+\mu_+^n\|v_t\|\textrm { for } n\geq N,\\
\|Df^n(y) v_s\|\leq& C_+\lambda_+^n\|v_s\|\textrm { for } n\geq N.
\end{split}
\end{equation}
Using \eqref{iterates} we obtain:
\[
|\omega(y)(v_t,v_s)| \leq M_\omega  (C_+D_+)(\lambda_+\mu_+\eta^{-1})^n\|v_t\|\|v_s\| \textrm { for } n\geq N.
\]
Since $ \lambda_+\mu_+\eta^{-1} <1$ by assumption, and $n\geq N$ is arbitrary, we obtain $\omega (y)(v_t,v_s)=0$.

Analogously,  for  $v^1_s, v^2_s\in T_yW^{s,\mathrm{loc}}_x$:
\[
|\omega(y)(v^1_s,v^2_s)|\le C_+^2 M_\omega (\lambda_+^2\eta^{-1})^n\|v^1_s\| \|v^2_s\|, \ \forall n\ge N,
\]
and therefore, if $ \lambda_+^2\eta^{-1} <1$, we have $\omega(y)(v^1_s,v^2_s)=0$ for any $v^1_s,v^2_s \in T_y{W^{s,\mathrm{loc}}_x}$.
As this is true for any
$y\in W^{s,\mathrm{loc}}_x$, we obtain that $\omega_{\mid W^{s,\mathrm{loc}}_x}=0$, and therefore $W^{s,\mathrm{loc}}_x$ is isotropic.

\end{proof}

\subsection{Results on isotropic and coisotropic manifolds}

The next Proposition~\ref{lem:degeneracy} gives results which ensure that the form $\omega$ vanishes on some  manifolds.
This result will not be used in the  proofs of Theorems \ref{thm:main1} and \ref{thm:main2}. The proof of Proposition~\ref{lem:degeneracy}
is given in Section \ref{sec:proofisotropic} after the proofs of these theorems.

\begin{prop}[(Co)isotropic submanifolds]
  \label{lem:degeneracy}
 We take the standing assumptions from Section \ref{sec:standing}.
\begin{itemize}
\item[(i)]
If $N\subset \Lambda$ is an  isotropic submanifold  (not necessarily invariant),  that is, $\omega_{\mid N}=0$,  then we have:
\begin{equation} \label{eqn:vanishing_lagrangianmanifolds}
\begin{split}
\mu_+\lambda_+ \eta^{-1}<1 & \implies W^s_N \textrm{ is isotropic},  \ \text{that is} \ \omega_{\mid W^s_N} = 0 ,
\\
\mu_-\lambda_- \eta  < 1 &\implies  W^u_N  \textrm{ is isotropic}, \ \text{that is} \  \omega_{\mid W^u_N}= 0.
\end{split}
\end{equation}
\item[(ii)]
The stable and unstable manifolds of $\Lambda$ satisfy:
\begin{equation}\label{eqn:vanishing_2}
\begin{split}
\mu_+ \lambda_+ \eta^{-1}  < 1 &   \implies W^s_\Lambda  \textrm{ is coisotropic},\\
\mu_- \lambda_- \eta  < 1 & \implies  W^u_\Lambda  \textrm{ is coisotropic}.
\end{split}
\end{equation}
\end{itemize}
\end{prop}

We note that $\omega_{\mid W^{s,u}_\Lambda}$ is presymplectic and its kernel $K_x(\omega)$ has constant rank equal to $d_u=d_s$; see Remark~\ref{rem:constant_rank}. Recall that $d_s=d_u$ from Remark~\ref{rem:ds=du}.

\subsection{Some properties of rates in isotropic invariant manifolds}
\label{sec:isotropic}
In this section we start to explore
the interaction between rates and isotropic invariant manifolds.

In this section we assume that $\Lambda$ is an invariant manifold satisfying conditions \eqref{eq:bundles} and \eqref{eqn:NHIM}, but not necessarily \ref{eq:rates0}.
In particular, $\Lambda$ is not necessarily normally hyperbolic. It is easy to see that Theorem \ref{thm:main1} {\bf(A)} still holds under these hypotheses\footnote{Its proof uses Lemma \ref{lem:vanishing} which only requires \eqref{eq:bundles}, \eqref{eqn:NHIM} and \eqref{eqn:conformal_rates}}.
A  consequence  of Corollary~\ref{partial_converse} is
that if  $\Lambda$ is not symplectic,
then it does not satisfy \eqref{eqn:conformal_rates}.
One interesting case is when $\Lambda$ is an isotropic manifold.
By the vanishing lemmas, there are assumptions on
rates that imply that $\Lambda$ is  isotropic. Hence, we have
some inequalities on rates (involving only $\mu_\pm, \eta$) that imply other inequalities on rates
(involving $\mu_\pm, \lambda_\pm, \eta$)
by passing through isotropic manifolds.

This reveals some relation between rates, isotropy, normal hyperbolicity
that we illustrate in an example. A fuller theory is being developed
incorporating other ingredients.

\begin{cor}
Assume the setting in Section \ref{sec:standing} without (iii) and (iv), and that  $\Lambda$ is an invariant manifold satisfying \eqref{eq:bundles} and \eqref{eqn:NHIM}.

  Then:
  \begin{equation}
    \label{corrates}
    \begin{split}
      &   \ (\mu_+^* )^2 \eta^{-1} < 1   \quad  \textrm{OR} \quad
       \ (\mu_-^* )^2 \eta < 1  \\
      & \phantom{AAAAA} \implies \\
      & \omega|_\Lambda = 0 \\
      & \phantom{AAAAA} \implies \\
      &   \mu_+^* \lambda_+^*  \eta^{-1} \ge 1    \quad \textrm{OR} \quad
       \mu_-^*  \lambda_-^* \eta \ge 1.
    \end{split}
  \end{equation}

In particular, $\Lambda$ is not normally hyperbolic.
\end{cor}

  \begin{proof}
    We start by proving  the first implication in  \eqref{corrates}

  If  $ (\mu_+^* )^2 \eta^{-1} < 1 $,
 then for some $\mu_+$  with
     $(\mu_+ )^2 \eta^{-1} < 1 $ and some $C>0$, we
  have for all $v \in T_x \Lambda$,
  \[
  \| Df^n(x) v\|  \le C \mu_+^n  \|v\| \quad \textrm{ for } n > 0 .
  \]
  Hence, using \eqref{iterates} and taking the limit as $n \to \infty$,
  we have that $\omega(x)(u,v) = 0$, $ \forall u,v \in T_x\Lambda$, $\forall x \in
  \Lambda$.

  The identical argument for
$ \ (\mu_-^* )^2 \eta < 1$, taking $n \to -\infty$ is left to the reader.

  The second implication in  \eqref{corrates}
  is just the failure  of  \eqref{eqn:conformal_rates}.

  Finally, we note that if $\Lambda$ is a NHIM, then by \eqref{eqn:lambda_mu_gap} and \eqref{eqn:mu_gap},   we have $\lambda_+^*<\mu_+^*$, therefore
  $(\mu_+^* )^2 \eta^{-1}< 1$   implies  $\mu_+^* \lambda_+^*  \eta^{-1} < 1$;
  similarly, $(\mu_-^* )^2 \eta < 1$   implies  $\mu_-^*  \lambda_-^* \eta < 1$. This contradicts the last conclusion of \eqref{corrates}.
  Hence, $\Lambda$ cannot be normally hyperbolic.
\end{proof}

  Isotropic, specially Lagrangian manifolds have extra properties
  among rates that are incompatible with normal hyperbolicity.
  In this paper, we will only mention an example,
  and postpone a fuller exploration involving other
  concepts to future work.

  \begin{ex}
    Consider $\R^{10}$ endowed with the form
    $\omega = \sum_j d x_j \wedge d y_j$
and the map $f$ given by:
\[
\begin{split}
    f &(
    x_1,
    x_2,
    x_3,
    x_4,
    x_5,
    y_1,
    y_2,
    y_3,
    y_4,
    y_5)
    \\ = &
    \left( \lambda_+ x_1,
      \mu_-^{-1} x_2,
      \mu x_3,
      \mu_+ x_4,
      \lambda_-^{-1} x_5, \lambda_+^{-1} \eta y_1,
                    \mu_- \eta y_2,
                    \mu^{-1} \eta y_3,
                    \mu_+^{-1} \eta y_4,
                    \lambda_- \eta y_5 \right).
          \end{split}
\]

We just assume
\[
   \lambda_+ < \mu_-^{-1} \le  \mu  \le \mu_+ <
   \lambda_-^{-1} .
   \]

The $3$-D manifold $\Lambda$  corresponding to the variables $x_2,x_3,x_4$ (and the other variables set to
zero) is an invariant manifold which is isotropic.

The key point of the example is that we have introduced an intermediate rate $\mu$ in $\Lambda$.
The presence of a rate $\mu$ along the manifold forces
the presence of a rate $\mu^{-1}\eta$ in the normal bundle.

If $\mu_-^{-1} < \mu^{-1} \eta< \mu_+$ (which in our case could well happen)
we obtain that the presence of vectors with an intermediate rate $\mu$ as above is incompatible
with $\Lambda$ being normally hyperbolic.

Similar phenomena appear in the use of automatic reducibility in whiskered tori
\cite{CallejaCL20}.

\end{ex}

\subsection{Vanishing lemmas for derivatives of a general $2$-form $\omega$
}

For a general $2$-form $\omega$ (which may be non-closed or
be degenerate)  we have that
\[
f^*\omega = \eta \omega \quad \implies \quad
f^* (d\omega)  = \eta (d \omega).
\]

Hence, procedures similar to those  used
to prove  Lemma~\ref{lem:vanishingmanifolds}
can be applied to obtain a  vanishing Lemma~\ref{derivative_vanish}, where we assume that  $d\omega$ is bounded and  some adequate assumptions on rates, and we conclude that   $d\omega$ vanishes on several blocks.

This will be enough to give a proof of a variant of part (B) of Theorem~\ref{thm:main2} in  Section~\ref{proofnonclosed} under the assumptions of Lemma~\ref{derivative_vanish}.
In particular, the proof in Section~\ref{proofnonclosed}  does not assume that $\omega$ is closed and can be
extended to cases studied in \cite{wojtkowski1998conformally}.

\begin{lem} \label{derivative_vanish}

We make the standing assumptions from Section \ref{sec:standing} with {\textbf{(U5')}} instead of {\textbf{(U5)}}.
We assume only  that $\omega$ is a $2$-form not necessarily closed or non-degenerate.

We assume that $f$ satisfies \eqref{eqn:conformally_symplectic} and  hence
satisfies also:
\begin{equation}\label{conformalderivative}
f^* (d\omega)  = \eta (d \omega).
\end{equation}

Then we have:
\begin{itemize}
\item[(A)]
If
  \begin{equation} \label{tts_rates}
    \mu_+^2 \lambda_+ \eta^{-1}  < 1,
    \end{equation}
  then  for every $y \in W^s_\Lambda$ and $x\in\Lambda $ such that $y\in W^s_x$,
  for all
  $v_t, w_t \in  T_y W^s_\Lambda$, and for all $u_s \in T_y  W^s_x$,
  we have:
  \begin{equation} \label{vanish_matrix_tts}
    d\omega (y) (v_t, w_t, u_s) = 0.
  \end{equation}

  Analogously, if:
  \begin{equation} \label{ttu_rates}
    \mu_-^2 \lambda_- \eta  < 1.
  \end{equation}
    then,   for every $y \in W^u_\Lambda$ and $x\in\Lambda $ such that $y\in W^u_x$,
  for all
  $v_t, w_t \in  T_y W^u_\Lambda$, and for all $u_u \in T_y  W^u_x$,
  we have:
  \begin{equation} \label{vanish_matrix_ttu}
    d\omega (y)(v_t, w_t, u_u) = 0.
  \end{equation}
\item[(B)]
  \begin{equation}
\begin{split}
\lambda_+^3 \eta^{-1} < 1 & \implies  \forall x \in \Lambda, \
d(\omega_{\mid W^s_x})= 0, \\
\lambda_-^3 \eta < 1 & \implies \forall x \in \Lambda, \
d(\omega_{\mid W^u_x})= 0.
\end{split}
\end{equation}
\end{itemize}
\end{lem}

\begin{proof}
 As in the previous lemmas, we observe that, because of
  \eqref{conformalderivative}, we have
  for all $n \in \Z$ (and for all $y$ and all $(v,w,u)$),
  \[
  d\omega(y) ( v, w, u) = \eta^{-n} d\omega(f^n(y))(
  Df^n(y)v, Df^n(y) w, Df^n(y)u )
  \]
  With the respective assumptions on rates and  uniform boundedness
  of the derivative $d\omega$, we obtain the desired result taking the limit
  $n \to \pm \infty$ in the different cases, as we  did in the proofs
  of Lemma~\ref{lem:vanishing} and Lemma~\ref{lem:vanishingmanifolds}.
\end{proof}

\subsection{Vanishing lemmas for some examples of unbounded symplectic forms}
\label{sec:unbounded_forms}
In this section, we develop a  result Lemma~\ref{prop:unbounded},which is very
similar to Lemma~\ref{lem:vanishing}, but which applies to some unbounded
symplectic forms. The system is assumed to have a compact invariant set
$\mathcal{A}$ -- that serves as the origin to measure distances, e.g.
$\mathcal{A}$ could be a fixed  point.
We also assume that the symplectic form  at a point $x$  is bounded
by a power $\alpha$ of the distance from $x$ to $\mathcal{A}$ and that the hyperbolicity
rates together with $\alpha$ and $\eta$ satisfy some relations.

We hope that Lemma~\ref{prop:unbounded}  indicates the ingredients needed in a systematic theory dealing with
unbounded forms. However, this result will not be used in this paper.

\begin{lem}[Infinitesimal Vanishing Lemma for Some  Unbounded Forms]\label{prop:unbounded}
  {$ $}

  We take as granted the standing assumptions from Section \ref{sec:standing} without \eqref{eqn:bounded_omega}.

  We  assume that   there exists a compact invariant set $\A \subset \Lambda$ such
  that for some $A,B , \alpha > 0$
    we have for all $x\in\Lambda$
    \begin{equation} \label{form_assumption}
      \| \omega(x) \| \le B+A \cdot d(x,\A)^\alpha ,
    \end{equation}
    where  $d$ is the Riemannian distance measured along $\Lambda$.

Then, for all $x\in\Lambda$, $v^1_s,v^2_s\in E^s_x$, $v^1_u,v^2_u\in E^u_x$,, $v^1_t,v^2_t\in T_x\Lambda$, we have:
\begin{equation}
\label{eqn:vanishing_all}
\begin{split}
\mu_+^{1+\alpha} \lambda_+ \eta^{-1}  < 1 &\implies  \omega(x)(v^1_t,v^1_s)= 0,
\\
\mu_-^{1+\alpha} \lambda_- \eta  < 1 & \implies  \omega(x)(v^1_t,v^1_u)= 0,
\\
\mu_+^\alpha \lambda_+^2\eta^{-1}<1 & \implies  \omega(x)(v^1_s,v^2_s)=0,
\\
\mu_-^\alpha \lambda_-^2 \eta < 1 & \implies  \omega(x)(v^1_u,v^2_u)=0,
\\
\mu_+^{2+\alpha} \eta^{-1}  < 1 &\implies  \omega(x)(v^1_t,v^2_t)= 0,\\
\mu_-^{2+\alpha} \eta   < 1 &\implies  \omega(x)(v^1_t,v^2_t)= 0.
\end{split}
\end{equation}
\end{lem}

\begin{rem}
One may wonder whether the assumptions of  Lemma~\ref{prop:unbounded} are contradictory.
An example in $M = \real \times \torus$ is obtained by choosing any function $h:\real \to \real$, satisfying  $|h(I)|\le |I|^\alpha$, and setting $\omega = h(I) dI \wedge d \theta$.
We consider a map of the form $f(I,\theta) = (g(I), \theta)$.
If  $h$, $g$ satisfy the separable differential equation $h(g(I)) g'(I) =\eta h(I)$, with $g(0)=0$
then the map $f$ is conformally symplectic for $\omega$ and the set $\A=\{ (0,\theta)\}$ satisfies the hypotheses of the lemma.
We need to choose $h$ so that the solution $g$ gives a diffeomorphism.

\end{rem}

\begin{proof}
Since $\A$ is a compact set, for every $x\in\Lambda$ we have
\[
d(x,\A)=\inf_{y\in \A}d(x,y)<+\infty .
\]
Condition \eqref{eqn:NHIM} implies that, for some constants $\tilde{D}_+, \tilde{D}_->0$ independent of
$x$, we have
\begin{equation} \label{orbit_ growth_assumption}
    \begin{split}
       & d(f^{n} (x), \A) \le \tilde{D}_+\mu_+^n d(x, \A), \quad n \ge 0, \\
       & d(f^{-n} (x), \A) \le \tilde{D}_-\mu_-^n d(x,\A), \quad n \ge 0 . \\
      \end{split}
    \end{equation}
From \eqref{orbit_ growth_assumption} and \eqref{form_assumption}, for $n\ge 0$, we have
\[
\begin{split}
\sup_{x \in \Lambda} \| \omega (f^n(x)) \|  \le& B+A\cdot d(f^n(x),\A)^\alpha\le B+A (\tilde{D}_+) ^\alpha \mu_+^{\alpha n}d(x,\A)^\alpha\\
  =&B+A^+_x  \mu_+^{\alpha n},\\
\sup_{x \in \Lambda} \| \omega (f^{-n}(x)) \|  \le& B+A\cdot d(f^{-n}(x),\A)^\alpha\le B+A (\tilde{D}_-) ^\alpha \mu_-^{\alpha n}d(x,\A)^\alpha\\
   =&B+A^-_x  \mu_-^{\alpha n},
\end{split}\]
where $A^+_x=A(\tilde{D}_+ )^\alpha  \cdot d(x,\A)^\alpha$ and $A^-_x=A(\tilde{D}_- )^\alpha  \cdot d(x,\A)^\alpha$.

    Using \eqref{iterates}, for $v^1_s, v^2_s  \in E^s_x$, $v^1_t,v^2_2\in T_x\Lambda$,
    we have for $n \ge 0$:
    \begin{equation}\label{stable_estimates}
    \begin{split}
      & | \omega(x)(v^1_s, v^2_s)| \le \left [\tilde B (\eta^{-1}\lambda_+^2)^n +\tilde A _x(\eta^{-1} \lambda_+^2 \mu_+ ^ \alpha)^ n \right]\|v^1_s\| \|v^2_s\| ,\\
      & | \omega(x)(v^1_s, v^1_t)| \le \left[\tilde B (\eta^{-1}\lambda_+ \mu_+)^n +  \tilde A _x(\eta^{-1}\lambda_+ \mu_+^{1+\alpha})^n \right] \|v^1_s\| \|v^1_t\| , \\
      & | \omega(x)(v^1_t, v^2_t)| \le \left[\tilde  B (\eta^{-1}\mu_+^2)^n +  \tilde A _x(\eta^{-1} \mu_+^{2+\alpha})^n \right] \|v^1_t\| \|v^2_t\| ,\\
      & | \omega(x)(v^1_t, v^2_t)| \le \left[\tilde B (\eta \mu_-^2)^n +  \tilde A _x(\eta  \mu_-^{2+\alpha})^n \right] \|v^1_t\| \|v^2_t\|,
    \end{split}
    \end{equation}
    for some constant  $\tilde B>0$ and   some   $\tilde A _x>0$  that  is independent of $n$ but depends on  $d(x,\A)^\alpha$.

    Under the hypothesis on the rates and using  \eqref{eqn:mubounds}, the limit
   as $n \rightarrow +\infty$
    of the right-hand-size of
    \eqref{stable_estimates}  is zero.

   A similar argument gives the vanishing of the symplectic form in the case when one of the tangent vectors is unstable (or both are unstable).
\end{proof}

\section{Proof  of Theorem \ref{thm:main1}}\label{sec:proofAmain1total}

\subsection{Proof of Theorem \ref{thm:main1}  (A) on the symplecticity of the NHIM}\label{sec:proofAmain1}
We want to show that
$\omega_\Lambda:=\omega_{\mid\Lambda}$
is a symplectic form and hence $\Lambda$ is symplectic.

Since the exterior derivative commutes with restriction to submanifolds,
we have:
\[
d(\omega_{\mid \Lambda})=(d\omega)_{\mid\Lambda}=0,
\]
so to prove that $\omega_\Lambda$ is symplectic we only have to prove that $\omega_\Lambda$ is non-degenerate.

For $x\in\Lambda$, if $v_*\in T_x\Lambda$ and $\omega(x)(v_*,v_t)=0$ for all $v_t\in T_x\Lambda$, then,
as \eqref{eqn:conformal_rates} is satisfied, we can apply Lemma~\ref{lem:vanishing}, and  we have
$\omega(x)(v_*,v)=0$ for all $v\in T_xM$. By the non-degeneracy of $\omega$ in $TM$  we conclude $v_*=0$.

The dynamics on $\Lambda$ is conformally symplectic because
 $f^*\omega=\eta\omega$ and the restriction commutes with the pullback. Hence,
$(f_{\mid\Lambda})^*\omega_\Lambda=\eta\omega_\Lambda$.

\subsection{Proof of  Theorem \ref{thm:main1} (B) on pairing rules}
\label{sec:pairing_rules}

In this section we show  that the geometry imposes certain symmetries on the possible rates.
In the case of symplectic maps, these symmetries (and their proofs) have been folklore
 but we have not been able to locate a specific reference.
Here we derive the symmetries for conformally symplectic maps, and note that the proof also
applies to the symplectic case.
For conformally symplectic systems, there are arguments for periodic orbits and for
Lyapunov exponents \cite{dettmann1996proof,wojtkowski1998conformally}, but the argument here is different and is based on the vanishing lemmas.

We are under the assumption that $\Lambda$ is symplectic and therefore $\omega (x)$  is non-degenerate for any $x \in \Lambda$.

For the optimal rate  $\mu_+^*$ we have that $\forall\eps>0, \,\exists  D_+=D_+(\eps)$ such that:
\[
\forall x \in\Lambda\, \forall u\in T_x\Lambda\,
\|Df^n(x)u\|\leq D_+(\mu_+^*+\eps)^n\|u\|, \ \forall n\geq 0.
\]

Taking $x \in \Lambda$ and
applying Lemma  \ref{cor:pairing} for $L=\Lambda$, $\alpha=\mu_+^*+\eps$ and $\beta=\frac{\eta-\eps}{\mu_+^*+\eps}$, as $\alpha\beta<\eta$,  we obtain  there exists
$v\in T_x\Lambda$ where $\omega(x)(u,v)\ne 0$ and  there exists $D_2 > 0$  such that
\[
\|Df^n(x)v\|\geq D_2 \left (\frac{\eta-\eps}{\mu_+^*+\eps}\right)^n\|v\|, \ \forall n\ge 0.
\]

Since $\mu_-^*$ is defined as an optimal rate, by Lemma~\ref{lem:rates_inverses} we have
\[
\frac{1}{\mu_-^*}\geq \frac{\eta-\eps}{\mu_+^*+\eps},
\]
and, since this holds for all $\eps>0$, we obtain
\[
\frac{1}{\mu_-^*}\geq \frac{\eta}{\mu_+^*}.
\]

Applying the same argument for the inverse map $f^{-1}$   we also have
\[
\frac{1}{\mu_+^*}\geq \frac{\eta^{-1}}{\mu_-^*}.
\]
We conclude
\begin{equation}\label{eqn:MUMU}
 \frac{\mu_+^*}{\mu_-^*}=\eta.
\end{equation}

A similar argument, which we now detail,
 yields
\[
\frac{\lambda_+^*}{\lambda_-^*}=\eta.
\]

For the optimal rate  $\lambda_+^*$ we have that $\forall\eps>0, \,\exists C_+=C_+(\eps)$ such that:
\[
\forall x \in\Lambda\, \forall u \in E^s_x \,
\|Df^n(x)v\|\leq C_+(\lambda_+^*+\eps)^n\|v\|, \ \forall n\geq 0.
\]

Taking $x\in \Lambda$ and applying Lemma \ref{cor:pairing} for $\alpha=\lambda_+^*+\eps$   and $\beta=\frac{\eta-\eps}{\lambda_+^*+\eps}$, as $\alpha\beta<1$,
we conclude that there is a vector $w \in T_x M$ such
that $\omega(x)(v, w) \ne  0 $ and
a constant  $C_2>0$ such that
\begin{equation}\label{eqn:PR}
 \|Df^n(x) w  \| \ge C_2 \left ( \frac{\eta-\eps}{\lambda_+^* +\eps} \right)^n \|w\|, \  \forall n\ge 0.
\end{equation}
Let $w=w^u+w^{ts}$, where $w^u\in E^u_x$ and $w^{ts}\in T_xW^s_\Lambda$.

Using \eqref{eqn:PR}, \ref{eqn:ratesinmanifolds} and \eqref{eqn:MUMU} we obtain
\begin{equation*}
\begin{split}
 \|Df^n(x) w^u  \| \ge & \|Df^n(x) w  \| - \|Df^n(x) w^{ts}  \| \\
 \ge& C_2 \left ( \frac{\eta-\eps}{\lambda_+^* +\eps} \right)^n \|w\| - D_+(\mu^*_+ +\eps)^n\|w^{ts}\|\\
 = &C_2 \left ( \frac{\eta-\eps}{\lambda_+^* +\eps} \right)^n \|w\| - D_+(\eta\mu^*_- +\eps)^n\|w^{ts}\| \\
 \ge & C_3 \left ( \frac{\eta-\eps}{\lambda_+^* +2\eps} \right)^n \|w^u\|,
\end{split}
\end{equation*}
for $n\ge 0$ sufficiently large and $\eps$ sufficiently small, where the last inequality is due to the fact that $\lambda^*_+\mu^*_-< 1$.
Note that $w^u \ne 0$ because we have upper bounds for the growth of $w^{ts}$ which are incompatible with the
lower bounds for the growth of $w$.

Since any uniform
bound $\lambda_-$ with
\[
\|Df^n(x) w^u\| \ge \tilde C (\lambda_-)^{-n} \|w^u\|
\]
has to satisfy
\[
\lambda_-^*\le \lambda_-
\]
we conclude,
by the same argument as before (letting $\eps\to 0$) that
\[
\lambda_-^*  \le \eta^{-1}    \lambda_+^* .
\]
Applying this result to $f^{-1}$ in place of $f$ we obtain the desired result.
\qed

\begin{rem}
  It is interesting to compare the proofs of pairing rules for rates
  above with the proofs of pairing rules for periodic orbits or
  for Lyapunov exponents in
  \cite{Dessler88,dettmann1996proof,wojtkowski1998conformally}.
  The proofs in the above references  are based on defining the operator
  $J_x: T_xM \rightarrow T_x M$ by
  $\omega(x)(u,v) = g_x(u, J_x v) $ where $g$ is the Riemannian metric.

  Then, the conformal symplectic property of the map  is translated into
  $ Df^n (x)^T J_{f^n(x)} Df^n(x)= \eta^n J_x$,
  where the transpose is with respect to the metric.
  Hence,
  \begin{equation}\label{relation}
    Df^n(x)  = \eta^{n}
    J^{-1}_{f^n(x)}  (Df^n (x))^{-T} J_x.
  \end{equation}

We think of \eqref{relation} as a relation
among linear operators in tangent spaces. In the literature,
sometimes, \eqref{relation}
is  described as a  relation among matrices using
a global frame (introduced already in the setup).
We emphasize that  \eqref{relation} has an intrinsic meaning
without a global frame.

  The equation \eqref{relation} relates  the
  rates of growth of $Df^n(x)$ and those of $(Df^n(x))^{-T}$
leading to pairing rules.
 Using \eqref{relation} to relate asymptotic rates,  seems to require
    that $J^{-1}_{f^n(x)}$ is uniformly bounded.

    The method we use here to obtain pairing rules  does not require
    that $\|J^{-1}_{x}\|$ is uniformly bounded nor the existence off a global frame.
  \end{rem}

\section{Proof of  Theorem \ref{thm:main2}}\label{sec:proof_main_2}

\subsection{Proof of Theorem \ref{thm:main2}  (A) on symplecticity of the homoclinic channel}

We first prove  that,
if $\Gamma$ is a homoclinic channel (see Definition~\ref{def:channel}), then
$\omega_{\mid\Gamma}$ is non-degenerate, hence
$(\Gamma,\omega_{\mid\Gamma})$ is a symplectic manifold.

Conditions \eqref{eqn:conformal_rates} allow to apply part \textbf{(A)} of Theorem \ref{thm:main1} obtaining that $\Lambda$ is symplectic.

If $\Gamma$ is sufficiently $\C^1$-close to $\Lambda$,
from $\omega_{\mid\Lambda}$ being non-degenerate we deduce
$\omega_{\mid\Gamma}$ is non-degenerate.

If $\Gamma$ is not  $\C^1$-close to $\Lambda$, by the Fiber Contraction Theorem  (see Lemma \ref{convergence_rates})  we have
\begin{equation}\label{eqn:lambdalemma}
  d_{\C^1}(f^n(\Gamma),\Lambda)\leq C   (\lambda_+\mu_-)^{n}, \textrm { for } n\geq 0.
\end{equation}
Then, there exists $N>0$ such that $f^N(\Gamma)$ is sufficiently $\C^1$-close to $\Lambda$
so that
$\omega_{\mid f^N(\Gamma)}$ is non-degenerate as in  the previous case.
Since $f$ is conformally symplectic we have
\[
\omega_{\mid f^N(\Gamma)}=(f^*)^N\omega_{\mid\Gamma}=\eta^N \omega_{\mid\Gamma}.
\]
Since $\omega_{\mid f^N(\Gamma)}$ is non-degenerate it follows that $\omega_{\mid  \Gamma}$ is non-degenerate.

\subsection{Proof of Theorem \ref{thm:main2}  (B) on symplecticity of the scattering map}\label{sec:proof_main_2_b}

In this section we give seven different proofs of Theorem~\ref{thm:main2}  (B)  or of some  versions of it
(some versions do not assume that $\omega$ is closed or use  different assumptions on rates or boundedness
of the derivatives of $\omega$).
We note that some of these proofs do not use that $\omega$ is non-degenerate, so they work without change in  the presymplectic case (see Section \ref{sec:proofs_presymplectic}) just taking into account that the conformal factor can be a function.

The first proof, given in Section~\ref{sec:main_2_B_proof_v} is based on vanishing lemmas.
A proof  adapting the  one from \cite{DLS08} from the symplectic case to the conformally symplectic case  is given in Section \ref{sec:original_proof}.
In Section~\ref{sec:proof_main_2_b_2} we give a proof based on the system of coordinates defined in Section \ref{rem:new coordinates}.
In Section \ref{sec:magic} we give a proof based on Cartan's magic formula.
These four proofs use the standing assumptions from Section
\ref{sec:standing} and conditions \eqref{eqn:conformal_rates}.
They use strongly that $\omega$ is closed, but they do
not use that $\omega$ is non-degenerate, so that these proofs
apply to presymplectic forms.

A fifth proof, given in Section~\ref{sec:proof_main_2_b_L}, uses the
study of graphs, but requires \eqref{doundedomega} and different rate conditions.
We give a sixth proof based on vanishing lemmas in Section~\ref{proofnonclosed}, which does not use that $\omega$ is a closed form,  but also requires \eqref{doundedomega} and different rate conditions.
In Section~\ref{sec:iteration}  we give a seventh proof, based on iterations, which also uses \eqref{doundedomega}.

We also give two proofs  of part (C)  in Section~\ref{sec:proofpartC}.
We note that they are based on vanishing lemmas. The first one, given in Section \ref{sec:proofpartCStokes} and based on Stokes' Theorem, uses that $\omega$
is exact (hence closed) but it does not use that $\omega$ is non-degenerate.
The second one, given in Section \ref{sec:Cartanexact} also uses that $\Omega$ is exact and uses Cartan's magic formula.

\begin{rem}
To prove that
\[
(\Omega_+)^*(\omega_{\mid \Lambda})=\omega_{\mid W^{s,\textrm{loc}}_\Lambda}
\]
it is enough to work on $W^{s,\textrm{loc}}_\Lambda \cap \OO_{\rho}$.
The reason is that, taking $n>0$ big enough but fixed, $f^n(W^{s,\textrm{loc}}_\Lambda) \subset \OO_\rho$ and,  by \eqref{eqn:intertwining} we obtain that:
\[
(\Omega_+)_{\mid W^{s,\textrm{loc}}_\Lambda}=f^{-n}_{\mid \Lambda}\circ
(\Omega_+)_{\mid f^n( W^{s,\textrm{loc}}_\Lambda)} \circ f^n_{\mid W^{s,\textrm{loc}}_\Lambda}
\]
and therefore we obtain the equality in all $W^{s,\textrm{loc}}_\Lambda$
and indeed on $W^s_\Lambda$.
\end{rem}

\subsubsection{A proof of Theorem \ref{thm:main2}  (B)  based on Stokes' theorem}
\label{sec:main_2_B_proof_v}
We prove that the $2$-form $\omega$ is invariant under the pullback of $\Omega _+$, as the proof for $\Omega _-$ is analogous.

\begin{figure}
\centering
\includegraphics[width=0.5\textwidth]{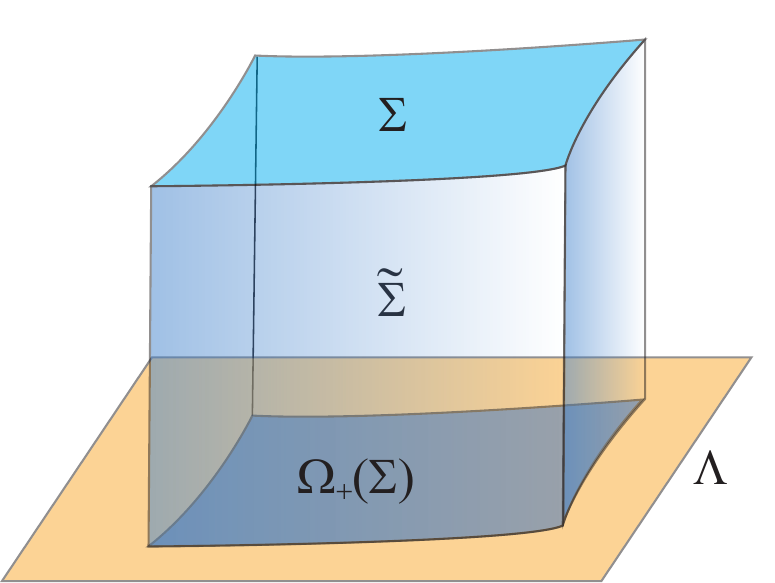}
\caption{}
\label{stokes_2_cell}
\end{figure}

It is enough to take  any  $y\in W^{s,\textrm{loc}}_{ \Lambda}$,
and any two tangent vectors $v^1,v^2\in T_yW^s_\Lambda$. We will prove that
\begin{equation}
\label{eqn:pullback_Omega}(\Omega_+^* \omega)(y)(v^1,v^2) = \omega(y)(v^1,v^2).
\end{equation}

We  define a $2D$-cell\footnote{A concrete, but slightly more costly
in the regularity  is to write explicitly the cell
as
$\Sigma( t_1,t_2)=\exp_y(\eps (t_1v^1+t_2v^2)),\quad 0\leq t_1,t_2\leq 1$,
for $0<\eps$ sufficiently small, where now $\exp$ denotes the exponential
mapping for the metric restricted to $W^s_\Lambda$.} $\Sigma\subseteq W^{s,\textrm{loc}}_\Lambda$
in such a way that it is tangent to  $v_1$ and
$v_2$ using the uniform system of coordinates  assumed
to exist in {\bf(U1)} and restricting it to
$W^{s,\mathrm{loc}}_\Lambda$.

By the transversality assumption on the vectors $v^1,v^2$,
it follows that different points $a\in  {\Sigma}$ project  onto different  points $\Omega_+(a)\in \Lambda$.

Define a $3D$-cell $\tilde{\Sigma}$ in $W^{s,\textrm{loc}}_\Lambda$ by
 \begin{equation}\label{eqn:3Dcell}
    \tilde{\Sigma}(t,t_1,t_2)=\gamma(t; \Sigma (t_1,t_2),\Omega_+(\Sigma(t_1,t_2))), \quad 0\leq t, t_1,t_2\leq 1
 \end{equation}
where
\begin{equation}
\label{eqn:gamma_def}
\gamma(\cdot, a,\Omega_+(a)) \textrm{ is a   a path in } W^s_{\Omega_+(a)}\textrm{ from } a \textrm{ to }\Omega_+(a) .
\end{equation}
The family of paths is chosen so that they depend smoothly on $a\in  \Sigma$ and
that $\tilde{\Sigma}$ forms  a   $3D$-cell inside  $W^s_\Lambda$.
The projection $\Omega_+(\Sigma)$ is a $2D$-cell inside $\Lambda$.
See Fig.~\ref{stokes_2_cell}.
Let $c(s)$ be a piecewise smooth parametrization of $\partial  \Sigma$, for $s\in\partial([0,1]^2)$.

Since $\omega$ is a closed form, using Stokes Theorem  we   compute

\begin{equation}\label{Stokes}
0=\int_{\tilde{\Sigma}} d\omega =
\int_{\partial \tilde{\Sigma}}\omega =
\int_\Sigma \omega-\int_{\Omega_+(\Sigma)}\omega+\int_\Upsilon \omega,
\end{equation}
where $\Upsilon$ is the $2D$-cell completing the boundary of $\tilde\Sigma$.
We consider it parameterized by
\[
\Upsilon(t,s)=\gamma(t,c(s)),\textrm{ for } (t,s) \in [0,1]\times \partial([0,1]^2).
\]

Now we compute the integral along $\Upsilon$:
\[\begin{split}
\int_\Upsilon \omega=&\int _{s\in \partial([0,1]^2) }\int _{t\in[ 0,1] }\omega (\Upsilon(t,s)) (\partial _t \Upsilon(t,s), \partial s\Upsilon(t,s))dt ds=0,
\end{split}
\]
where we have used that
\[
\left(\partial _t \Upsilon(t,s), \partial s\Upsilon(t,s)\right) \in T_{ \Upsilon(t,s)   }W^s_{\Omega_+(\Upsilon(t,s))} \times T_{ \Upsilon(t,s)   } W^{s,\mathrm{loc}}_\Lambda ,
\]
and therefore, conditions \eqref{eqn:conformal_rates} allow to apply Lemma \ref{lem:vanishingmanifolds} obtaining that  $\omega (\Upsilon(t,s)) (\partial _t \Upsilon(t,s), \partial_s\Upsilon(t,s))=0$.

In conclusion
\begin{equation}\label{cell_vanish}
\int_{\Sigma} \omega=\int_{\Omega_+(\Sigma)} \omega.
\end{equation}

Since \eqref{cell_vanish} holds for any $2$-cell  $\Sigma$ that is  transverse to the fiber
$W^s_{\Omega_+(y)}$, and for any $y\in W^{s,\textrm{loc}}_\Lambda$, it follows that
\[
\Omega_+: W^{s,\textrm{loc}}_\Lambda\to \Lambda
\]
satisfies  $\Omega_+^* (\omega_{\mid \Lambda})=\omega_{\mid W^{s,\textrm{loc}}_\Lambda}$.

Finally,
as $\Omega _+^\Gamma$
is a restriction of this map to the symplectic manifold $\Gamma$, it is  symplectic.
\bigskip

An analogous reasoning gives that $\Omega_-^\Gamma$ is symplectic, and then,  $(\Omega _-^\Gamma)^{-1}$ is also symplectic and therefore $S=\Omega _+^\Gamma\circ (\Omega _-^\Gamma)^{-1}$ is  symplectic.
%
%
%

\subsubsection{A proof of Theorem \ref{thm:main2}  (B)  by adapting the proof of \cite{DLS08} from the symplectic case to the conformally symplectic case}
\label{sec:original_proof}
The proof of \cite{DLS08} uses a similar geometric  construction as the proof in
Section~\ref{sec:main_2_B_proof_v}.
The paper \cite{DLS08} starts from the same cell depicted in
Fig.~\ref{stokes_2_cell} and obtains the desired result  by showing that the integral
of $\omega$ over $\Upsilon$
is zero.

The vanishing of this 2D integral is obtained using that for every $n > 0$
\[
\int_{\Upsilon} \omega = \eta^{-n} \int_{\Upsilon} ({f^n})^* \omega =
\eta^{-n} \int_{f^n (\Upsilon)} \omega.
\]
We now observe that the Riemannian area of $f^n( \Upsilon)$ is bounded from above by $C(\lambda_+ \mu_+)^n$.
Since $\omega$ is bounded, under the rate conditions \eqref{eqn:conformal_rates}, we obtain that
$ \eta^{-n} \left| \int_{f^n (\Upsilon)} \omega \right|$ can be made
as small as desired by taking $n$ large.

The proof in Section~\ref{sec:main_2_B_proof_v}, can be considered as a \emph{``disintegration''}
of the  argument in \cite{DLS08}. We can think of the vanishing lemma as dividing
$\Upsilon$ into infinitesimal cells and showing that each infinitesimal integral is exactly zero.
Proving first the infinitesimal result gives more flexibility
and the vanishing lemmas are used also to prove the pairing rules.

\subsubsection{A proof  of Theorem~\ref{thm:main2}  (B)  based on a system of coordinates }
\label{sec:proof_main_2_b_2}

A more explicit version  of Lemma~\ref{lem:vanishingmanifolds}
can be obtained by  using the system of coordinates defined in Section \ref{rem:new coordinates}, which exists thanks to hypothesis \eqref{U1}.
Using these coordinates we can make the symplectic form $\omega$ explicit.

Recall that the coordinate system $\varphi$ is so that
\[
\{\varphi(x,y)\,|\, y\in B_\rho(0)\}=W^{s,\mathrm{loc}}_x.
\]

Note that the representation of $\Omega_+$ in this system of
coordinates is  given  by:
\begin{equation}\label{eqn:Omega_plus coords}
\Omega_+(x,y) = (x, 0)
\end{equation}
The coordinate $x$ is defined precisely as taking the
projection $\Omega_+$, therefore $\Omega_+$ is represented
by setting the $y$ coordinate to $0$.

We can identify:
\[
T_{(x,y)}W^s_\Lambda=\{(t,w),\, t \in \R ^{d_c},\, w \in \R ^{d_s}\} ,
\]
so we can choose a basis $(t_1,0), \dots (t_{d_c},0), (0,w_1),\dots (0, w_{d_s})$ of $T_{(x,y)} W^s_\Lambda$ independent of the point $(x,y)$.

It is important to remark that, in these  coordinates, we can use the results of Lemma \ref{lem:globalrates} so that
\begin{equation}\label{eq:vsn}
\begin{split}
\|Df^{n}(x,y)(t,0) \| &\le D_+\mu_+^n \|t\|,\\
\|Df^{n}(x, y)(0,w) \| & \le C _+\lambda _ +^n \|w\|,
\end{split}
\quad n\ge 1 .
\end{equation}

By hypotheses \eqref{eqn:conformal_rates},  we can  use \eqref{eqn:vanishingstable} of  the Vanishing Lemma~\ref{lem:vanishingmanifolds},
to
obtain that the symplectic form can be represented as:
\begin{equation}
  \label{decomposition3}
  \omega (x,y) = \begin{pmatrix}
    \omega_{xx}(x,y)  & 0 \\
    0 & 0
  \end{pmatrix} .
\end{equation}

When  $\omega_{\mid \Lambda}$ is non-degenerate, the kernel of $\omega|_{W^{s,\textrm{loc}}_\Lambda}$
in this neighborhood is the tangent to the  $W^s_x$ leaves of the strong stable
foliation
\footnote{
This is consistent with Lemma~\ref{frobenius} that shows
that the kernel of a presymplectic form integrates to a foliation.}.

Now we proceed to the  proof  of Theorem~\ref{thm:main2}  (B).

We have the representation of $\omega$ by \eqref{decomposition3}.

The following is the key  observation:
if the symplectic form $\omega$ is closed, expressing the
differential in coordinates we then have:
\begin{equation}\label{eqn:partial_omega}
\partial_{y} \omega_{xx}(x ,y)  = 0 .
\end{equation}

To show this, take a sufficiently small patch $U\subset \Lambda$,
where we can trivialize   $E^s$. Since $d\omega=0$, we have $d(\omega_{\mid W^s_\Lambda})=0$ which expressed in coordinates gives
\[\begin{split}
0=\sum_i\sum_{j<k} \partial_{x_i}\omega_{x_j x_k} dx_i\wedge dx_j\wedge dx_k+\sum_l\sum_{j<k} \partial_{ y_l }\omega_{x_j x_k} d y_l \wedge dx_j\wedge dx_k.
\end{split}
\]

As the terms in the above sum are linearly independent, it follows that $\partial_{ y_l }\omega_{x_j x_k}(x, y)  = 0$ for all $l$ and  $j< k$.
This shows \eqref{eqn:partial_omega}.

Therefore, $\omega_{xx}$ depends only on $x$.
Since in this system of coordinates we have $\Omega_+(x,y) = (x, 0)$ (see \eqref{eqn:Omega_plus coords}),
  we obtain directly that, if  $n$ sufficiently large and $\mathscr{C} \subset f^n(\Gamma)$ is a $2$-cell, then
  \[
  \omega( \Omega_+(\mathscr{C})) = \omega(\mathscr{C}) .
  \]
\qed

\subsubsection{Proof of part (B) of Theorem~\ref{thm:main2} based on Cartan's magic formula}
\label{sec:magic}
The following proof is similar in spirit to the one in Section \ref{sec:proof_main_2_b_2}
 but avoids the construction of a system of coordinates.

 To prove that $\Omega_+:  W^{s,\mathrm{loc}}_\Lambda \to \Lambda$
  satisfies
  $(\Omega_+)^*(\omega_{\mid \Lambda})=\omega_{\mid W^{s,\mathrm{loc}}_{\mid \Lambda}}$, we
proceed as follows.
Take any section $\Psi \subset W^{s,\mathrm{loc}}_\Lambda \cap \OO_\rho$ transversal to the foliation \eqref{foliationsstablemanifolds} (see condition  \eqref{gammatransversal}) and consider the restricted  wave map
\[
\Omega_+^\Psi\equiv (\Omega_+)_{\mid \Psi}:\Psi\to \Lambda.
\]
We will see that $\Omega_+^\Psi$
satisfies:
$ (\Omega_+^\Psi)^* (\omega_{\mid \Lambda}) = \omega_{\mid \Psi}$.

As usual, we can assume that $\Psi$ is $\C^1$ close to $\Lambda$ and
use a finite number of iterates to get to others.

For   $x \in \Omega_+(\Psi)\subset \Lambda $, using the implicit function theorem,
 we can associate  unique $\gamma(x)  \in \Psi$ and $v(x) \in T_x W^s_x = E^s_x$
in such a way that
 $\gamma(x) = \exp_x(  v(x) )$, where the exponential is along $W^s_x$
 and $v(x)$ is required to be in a sufficiently small ball.
Both $\gamma$ and $v$   depend on $x\in\Omega_+(\Psi )\subset  \Lambda$
in a continuously differentiable way. As we mention,
we can always restrict $\Psi$ so that $\Omega_+(\Psi)$ is bounded.

Consider the $\C^1$ family  of mappings $\phi_t: \Omega_+(\Psi ) \to  W^s_\Lambda$,  indexed by  $t \in [0,1]$:
\[
\phi_t(x) = \exp_x (t v(x) ).
\]
Clearly,
$\phi_0(x) = x, \phi_1(x)  = \gamma(x)$,
and, more  succinctly,
$\phi_0 = \Id$ and $\phi_1 = (\Omega_+^\Psi )^{-1}$.

We let $\frac{d}{dt} \phi_t = V\circ \phi_t$, where   $V(\phi_t (x)) $ is tangent to $W^s_x$ at $\phi_t(x)$.
This defines $V$ as a $\C^1$ vector field on some domain in
$W^s_\Lambda$.

We now compute,  using Cartan's magic formula
\[
\frac{d}{dt} \left(\phi^*_t\omega \right)=  \phi_t^* \left[   i(V) d \omega +  d i(V) \omega\right].
\]
 The first term above is zero  because  $\omega$ is closed.
The second term is also zero  by the  Vanishing Lemma \ref{lem:vanishingmanifolds}.
Therefore
\[
\omega_{\mid\Lambda} = \phi_0^*( \omega_{\mid\Lambda}) =
\phi_1^*( \omega_{\mid \Psi})=  ((\Omega_+^\Gamma)^{-1})^* (\omega_{\mid\Psi}).
\]
\qed

\subsubsection{Proof of part (B) of Theorem~\ref{thm:main2} based on
  graphs in products of  manifolds
 }  \label{sec:proof_main_2_b_L}

In this section we present another proof of part {\bf(B)} of
Theorem~\ref{thm:main2}, with the  standing assumptions from Section \ref{sec:standing}
but with {\textbf{(U5')}}.

We also assume that the rates \eqref{eq:rates0} do not  satisfy \eqref{eqn:conformal_rates} but they satisfy a different condition:
\begin{equation} \label{lagrangianrates}
 \mu_-\mu_+^2 \lambda_+ \eta^{-1} < 1, \quad  \mu_+\mu_-^2\lambda_- \eta < 1 .
\end{equation}

This proof is based on the  study of graphs.

First, we recall some standard results.
Given a pair of
manifolds $(M_1, \omega_1)$, $(M_2, \omega_2)$ where $\omega_i$ are two $2$-forms,
and a pair of maps $g_1:M_1\to M_1$, $g_2:M_2\to M_2$.
We define
\begin{equation}\label{doubling}
  \begin{split}
    &\tilde M   = M_1 \times M_2 ,
\quad  \tilde \omega   = (-\omega_1) \oplus \omega_2, \\
&\textrm{ that \  is, \ for } \ x_1\in M_1, \ x_2\in M_2, \ v_1,w_1\in T_{x_1}M_1, \ v_2,w_2\in T_{x_2}M_2:
\\
    &  \tilde\omega (x_1,x_2)( (v_1,v_2), (w_1,w_2) )= -\omega_1(x_1)(v_1,w_1) + \omega_2(x_2)(v_2,w_2) ,\\
  & \tilde g : \tilde M  \rightarrow \tilde M,\\
    & \tilde g (x_1,x_2)  = (  g_1(x_1),   g_2(x_2) ),
\end{split}
\end{equation}

Given a map $f: M_1 \rightarrow M_2$, we
define its graph $\G(f) \subset \tilde M$ by:
\[
\G(f) = \{(x, f(x) )\, |\,  x \in M_1 \}.
\]

The following result is well known:

\begin{lem}\label{lem:lagrangiangraph}
  \label{lagrangiangraph}
With the notations above, the diffeomorphism  $f$ satisfies $f^*(\omega_2)=\omega_1$ if and only if $\tilde \omega$ vanishes on $\G(f) \subset \tilde{M}$.
\end{lem}

\begin{proof}
The standard and easy proof of Lemma~\ref{lagrangiangraph}
is just to observe that $T_{(x,f(x))}  \G(f)= \{(u, Df(x)u ) \, | u \in T_x M_1 \}$.
Hence, $\tilde \omega$ vanishes on $\G(f) \subset \tilde{M}$  is the same as  having for all
$x \in M_1$,
$u,v \in T_x M_1$,
\[
\begin{split}
0 &= \tilde{\omega}(x,f(x)) ((u, Df(x)u), (v, Df(x)v) )\\
&  =
- \omega_1 (x)(u,v) + \omega_2 (f(x))( Df(x)u, Df(x)v)
\end{split}
\]
\end{proof}

If $f^*(\omega_2)\ne \omega_1$,  the  form ${\tilde \omega}_{\mid\G(f)}$ does not vanish.
The size of  $\|\tilde {\omega}_{\mid\G(f)}\|$ is a measure of the failure of
$f^*\omega_1 = \omega_2$.

We now proceed with the proof of part {\bf(B)} of
Theorem~\ref{thm:main2}.
In this case, $M_1=M_2=M$, $\tilde{\omega}=(-\omega)\oplus\omega$,  and $g_1=g_2=f:M\to M$ which satisfies:
$f^*(\omega)=\eta \omega$,
and therefore
  \[
  \tilde f^* \tilde \omega = \eta \tilde \omega .
  \]

  To prove that $\Omega_+:  W^{s,\mathrm{loc}}_\Lambda \to \Lambda$
  satisfies
  $(\Omega_+)^*(\omega_{\mid \Lambda})=\omega_{\mid W^{s,\mathrm{loc}}_{\mid \Lambda}}$, we
proceed as we did in section \ref{sec:magic}, taking any section $\Psi \subset W^{s,\mathrm{loc}}_\Lambda$ transversal to the foliation \eqref{foliationsstablemanifolds} (see condition  \eqref{gammatransversal}) and considering the restricted  wave map
\[
\Omega_+^\Psi\equiv (\Omega_+)_{\mid \Psi}:\Psi\to \Lambda.
\]

We will see that $\Omega_+^\Psi$
satisfies:
$ (\Omega_+^\Psi)^* (\omega_{\mid \Lambda}) = \omega_{\mid \Psi}$.

To prove it,
we  take the  graph of $\Omega_+^\Psi$
\[
\G(\Omega_+^\Psi)\subseteq \Psi\times \Lambda\subseteq M\times M =\tilde M
\]
and prove that
$\tilde \omega_{\mid \G(\Omega_+^\Psi)}=0$.

To this end, we also consider
 $\textrm{Id}_{\Lambda}:\Lambda\to\Lambda$
 whose graph
 $\G(\textrm{Id}_{\Lambda})=\tilde\Lambda=\Lambda\times\Lambda\subseteq \tilde M$,
 and note that
$\tilde\omega_{\mid\tilde\Lambda}=0$.

Observe that the equivariance relation for the wave maps \eqref{eqn:intertwining} when restricted to a transversal manifold $\Psi$ to the foliation gives a relation similar \eqref{eqn:iterationsomega}, that is:
\begin{equation}\label{eqn:intertwiningupsilon}
\begin{split}
&\Omega_+^{\Psi}=f_{\mid\Lambda}^{-n}\circ\Omega_+^{f^n(\Psi)}\circ f^{n}, \quad n\ge 0
\end{split}
\end{equation}
If we reformulate the equivariance relation \eqref{eqn:intertwiningupsilon}  in
terms of graphs, we obtain:
\[
{\tilde f}^n( \G(\Omega_+^\Psi)) =
  \G(\Omega_+^{f^n(\Psi)}),\quad n\ge 0
\]
where
$\G(\Omega_+^{f^n(\Psi)})\subseteq  f^n(\Psi)\times\Lambda \subseteq \tilde M$ is the graph of $\Omega^{f^n(\Psi)}_+$.

Therefore, we have for all $n\ge 0$
  \begin{equation}\label{goodidentity}
 \tilde \omega_{\mid\G(\Omega_+^ \Psi)}=\eta^{-n}({\tilde f}^{n})^* (
  \tilde{\omega}_{\mid\G(\Omega_+^{f^n(\Psi)})} ) .
  \end{equation}

  Using Lemma~\ref{convergence_rates},
  we have that
  $d_{\C^1}( f^n(\Psi), \Lambda)  \le C (\lambda_+ \mu_-)^n $,
  hence
  \[
d_{\C^1}\left( \G(\Omega_+^{f^n(\Psi)}), \G( \textrm{Id}_{\Lambda})\right )
    \le  C (\lambda_+ \mu_-)^n .
    \]
    Using \eqref{doundedomega} and that
    $\tilde{\omega}_{\G( \textrm{Id}_{\Lambda})} = 0 $,
    we have:
    \[
    \begin{split}
  \| \tilde{\omega}_{\mid\G(\Omega_+^{f^n(\Psi)})}  \|_{\C_0}
  & =   \| \tilde{\omega}_{\mid\G(\Omega_+^{f^n(\Psi)})} -
\tilde{\omega}_{\mid\G( \textrm{Id}_{\Lambda})} \|_{\C_0} \\
  & \le\|\tilde \omega\|_{\C^1}
  d_{\C^1}( \G(\Omega_+^{f^n(\Psi)}), \G( \textrm{Id}_{\Lambda}))\\
    & \le C (\lambda_+ \mu_- )^n
    \end{split}
    \]

    Hence, we estimate  \eqref{goodidentity},
    using the obvious estimates for $f^*$ and the
    previous  estimates:
    \[
    \|\tilde{ \omega}_{\mid\G(\Omega_+^ \Psi)}\|_{\C^0} \le C \eta^{-n} \mu^{2n}_+ (\lambda_+ \mu_-)^n
    = C (\mu^{2}_+ \lambda_+ \mu_-\eta^{-1})^n, \quad n \ge 0 .
    \]

    We conclude that, under the assumptions \eqref{lagrangianrates}, the left hand side of the above vanishes and we obtain that
the form $\tilde \omega$ vanishes on $\G(\Omega_+^ \Psi)$,
or, equivalently by Lemma \ref{lem:lagrangiangraph}, that $(\Omega_+^ \Psi)^* (\omega)=\omega$.
As this is true for any section $\Psi$ the map $\Omega_+$ satisfies:
    \[
    \Omega^*_+(\omega_{\mid \Lambda}) = \omega_{\mid W^{s,\mathrm{loc}}_\Lambda}
    \]
    An analogous proof works for the map $\Omega_-$
  \qed

  \begin{rem}
The proof in this section,  based on
 the study of graphs, as well as the proof
 based in iteration  given in Section \ref{sec:iteration},
use only the convergence of $f^n(\Psi)$ to $\Lambda$. We use only the
most elementary bounds.

One advantage of the use of  elementary bounds is that the
proofs
work for largely arbitrary forms.  This allows us
to obtain results for more models. See Section~\ref{othermodels},
in  particular Section~\ref{thermostat}.

When $\omega$ is indeed a symplectic form, $\G(\Omega_+^ \Psi)$ is a Lagrangian manifold and
this gives extra properties to study the convergence of approximations and perturbation theory.
\end{rem}

\subsubsection{Proof of Theorem~\ref{thm:main2}  (B) for non-closed
  forms based on vanishing lemmas}\label{proofnonclosed}

In this section we present a version of part (B) of Theorem \ref{thm:main2} assuming the standing assumptions  of Section \ref{sec:standing}  without the hypothesis that the form $\omega$ is  closed but assuming  that it satisfies  \eqref{doundedomega}.
We also assume that the rates \eqref{eq:rates0}
satisfy
\eqref{tts_rates} and  \eqref{ttu_rates}.

The main tool will be the second item in  Lemma~\ref{derivative_vanish}, which claims that $d\omega$ vanishes on the leaves of the stable and unstable manifolds of $\Lambda$.

\begin{thm}\label{thm:mainvanishder}
  In the setup of Theorem~\ref{thm:main2}, do not assume   $d\omega = 0$.
  Assume that
$\omega$ satisfies {\textbf{(U5')}}.
  Assume that the hyperbolicity rates of $\Lambda$
  satisfy \eqref{tts_rates} and  \eqref{ttu_rates}.
  Then, we have
  \[
  \Omega_+^* (\omega_{\mid\Lambda}) = \omega_{\mid W^s_\Lambda}, \
  \Omega_-^* (\omega_{\mid\Lambda})  = \omega_{\mid W^u_\Lambda}.
  \]
\end{thm}

  \begin{proof}
  We do the proof for $\Omega_+$.
    We start as in the proof of part (B) of
    Theorem~\ref{thm:main2} in Section~\ref{sec:main_2_B_proof_v}

    Using the same notation, an application of
    Stokes' Theorem gives identity \eqref{Stokes}, that we write here:
\[ \int_{\tilde\Sigma } d\omega =
\int_{\partial \tilde\Sigma}\omega =
\int_\Sigma \omega-\int_{\Omega_+(\Sigma)}\omega+\int_\Upsilon \omega,
\]
where $\tilde \Sigma$ is the $3D$-cell defined in \eqref{eqn:3Dcell}.
As we are not assuming $\omega$ is closed, we need an argument to show that the left hand side of this equality vanishes.

We have that
\begin{equation} \label{explicit_integral}
\int_{\tilde \Sigma} d\omega = \int_{[0,1]^3}
d\omega \left(\tilde \Sigma(\tilde t)\right) \left( \partial_t \tilde \Sigma(\tilde t),
\partial_{t_1}\tilde \Sigma(\tilde t), \partial_{t_2} \tilde \Sigma(\tilde t)\right)
dt \, dt_1 \, dt_2 ,
\end{equation}
where we denote $\tilde t=(t,t_1,t_2)$.
Here $\tilde \Sigma(\tilde t)$ represents a point $y\in W^s_\Lambda$ and
  $\tilde \Sigma(1,t_1,t_2)$ represents  $\Omega_+(y)=x\in\Lambda$.

As we assume \eqref{tts_rates}, we can apply Lemma~\ref{derivative_vanish}
observing that
\[
\begin{split}
 \partial_t \tilde \Sigma(\tilde t) &\in T_yW^s_{x} , \\
\partial_{t_1}\tilde \Sigma(\tilde t), \partial_{t_2} \tilde \Sigma(\tilde t) &\in   T_{y}W^s_{\Lambda} ,
\end{split}
\]

Therefore, by \eqref{vanish_matrix_tts} of Lemma~\ref{derivative_vanish},  the integrand in
\eqref{explicit_integral} vanishes and we
obtain that $\int_{\tilde \Sigma}  d\omega = 0$.

From there, the proof does not need any change from
the proof in Section~\ref{sec:proof_main_2_b}.

The proof for $\Omega_-$  is analogous.
\end{proof}

\subsubsection{A proof   of Theorem~\ref{thm:main2} (B) for non-closed forms
  based on iteration}
\label{sec:iteration}

In this section we present a proof of
a version of part (B) of Theorem \ref{thm:main2} assuming the standing assumptions  of section \ref{sec:standing}, without the hypothesis that the form $\omega$ is  closed but assuming that the form $\omega$ satisfies  \eqref{doundedomega}.
We also assume that the rates \eqref{eq:rates0}  satisfy
\eqref{tts_rates} and  \eqref{ttu_rates}.

We will also assume that the manifold $W^{s,\textrm{loc}}_\Lambda$ is a $\C^2$-manifold, which is stronger than the standing assumption {\bf (H2)}.

We hope that this new proof provides some insights
that can be used  to develop perturbation theories or to
extend the theory to other contexts
involving non-closed forms  as in \cite{wojtkowski1998conformally}.

To prove that the wave maps  $\Omega_\pm$  preserve the  form $\omega$, as we did in section \ref{sec:proof_main_2_b_L}, we take any section $\Psi$ transversal to the foliation \eqref{foliationsstablemanifolds}, and prove that $\Omega^\Psi_\pm$  preserve the form $\omega$.

We use again that $\Omega^\Psi_\pm$ satisfy  \eqref{eqn:intertwiningupsilon}, to relate the projection on
$\Psi$ to the projection on $f^n(\Psi)$.
We will focus in $\Omega_+^\Psi$.
The heuristic idea is  that, by Lemma~\ref{convergence_rates},
$f^n(\Psi)$ approaches
$\Lambda$ for $n\ge N_+$ sufficiently large, so that we can approximate  $\Omega_+^{f^n(\Psi)}$ by the identity map.

The errors of symplecticity  in the approximation
amazingly wash away when put through the \eqref{eqn:intertwining}.
Let us emphasize that \eqref{eqn:intertwining} is an exact formula for every $n$
and that we do not need to take limits in the formula, only on
the estimates obtained  by applying it.

For a 2-cell $\mathscr{D}$, we denote
$\omega(\mathscr{D})= \int_\mathscr{D} \omega$ and $|\mathscr{D}| = \textrm{Area}(\mathscr{D})$  the Riemannian area.

\begin{rem}
In general, we have $\omega(\mathscr{D}) \le C |\mathscr{D}|$.
The converse inequality
\begin{equation}\label{converse}
  |\mathscr{D}| \le C |\omega(\mathscr{D})|
\end{equation}
is   true in bounded neighborhoods   when the  dimension of
$\Lambda$ is $2$,
but \eqref{converse} is false when $\Lambda$ has dimension $\ge 4$.

The fact that \eqref{converse} is true when the dimension of
$\Lambda$ is $2$, will be developed in Section~\ref{sec:2DLambda}.
\end{rem}

We use the coordinate system $(x,y)$ on $W^s_\Lambda$ described in Section \ref{rem:new coordinates}, and the approximation of $\omega$ near $\Lambda$ given by \eqref{eqn:decomposition2}.
More concretely, by \eqref{doundedomega} and \eqref{eqn:conformal_rates} -- which is implied by \eqref{tts_rates} --  we can apply \eqref{eqn:decomposition2}, obtaining
\[
|(\Omega^\Psi_+)^* \omega_{\mid\Lambda} (x,y)-\omega_{\mid \Psi}(x,y)|\le  M_\omega \|y\|.
\]
Then,  for any $2$-cell $\mathscr{D}$  in a neighborhood given
by $\|y\| \le \rho$ we have:
\begin{equation}
  \label{Omega_approx_symplectic}
| \omega( \Omega_+^\Psi (\mathscr{D}) )  - \omega(\mathscr{D}) |= \left|\int _\mathscr{D}  (\Omega^\Psi_+)^* \omega-\omega \right| \le C \rho |\mathscr{D}| .
\end{equation}

Given a $2$-cell $\mathscr{D}$ contained in $\Psi$, we
have, by the conformally symplectic property,
\[
\omega( f^n(\mathscr{D}) ) = \eta^n  \omega(\mathscr{D})
\]
and, by  \eqref{Omega_approx_symplectic} applied to $f^n(\mathscr{D}) $ using \eqref{eq:vsn} and \eqref{eqn:decomposition2}:
\begin{equation}\label{eqn:area}
\omega( \Omega_+^{f^n(\Psi)} (f^n(\mathscr{D}))) =
\eta^n  \omega(\mathscr{D}) + e_n |f^n(\mathscr{D})|
\end{equation}
with
\[
|e_n| \le C C_+ \rho  \lambda_+^n, \quad |f^n(\mathscr{D})| \le D_+^2\mu_+^{2n} |\mathscr{D}| .
\]
Finally, using \eqref{eqn:intertwining}:
\[
\begin{split}
\omega(\Omega_+^{\Psi}(\mathscr{D})) &=
\omega( f_{\mid\Lambda}^{-n} \circ \Omega_+^{f^n(\Psi)} \circ  f^n(\mathscr{D}) ) =
\eta^{-n}\left(  \eta^n  \omega(\mathscr{D})  +  e_n  |f^n(\mathscr{D})| \right)\\
&= \omega(\mathscr{D}) + \eta^{-n}e_n  |f^n(\mathscr{D})| .
\end{split}
\]
Using the previous bounds we obtain:
\[
 \eta^{-n} e_n  |f^n(\mathscr{D})| \le D_+^2 C_+ \rho (\lambda_+\mu_+^2\eta ^{-1})^n |\mathscr{D}| .
\]

Under the assumption \eqref{tts_rates},
by taking the limit as $n \to \infty$ we conclude that
$\omega(\Omega^{\Psi}_+(\mathscr{D})) = \omega(\mathscr{D})$ for all 2-cells $\mathscr{D}$ in $\Psi$.

\begin{rem}
Note that the only ingredients entering in the
proof are the equivariance relation \eqref{eqn:intertwining},
the infinitesimal vanishing lemma (Lemma~\ref{lem:vanishing}), and the fact that the stable manifold
is tangent to the stable bundle. None of those use
the fact that the form $\omega$ is closed nor that it is
non-degenerate.
\end{rem}

\paragraph{\em The case when $\Psi$ is $2$-dimensional:   Proof  of part \textbf{(B)}  of Theorem~\ref{thm:main2} without  assumption \eqref{eqn:conformal_rates}}
\label{sec:2DLambda}

When $\Psi$ is $2$-dimensional,
if we assume \eqref{converse} with a $C$ uniform on the whole
manifold, we can obtain a stronger result.

Suppose that $\omega(\mathscr{D}) \ge 0$  (otherwise change $\mathscr{D}$ into
the cell with opposite orientation). Then our assumption \eqref{converse} can be written
\begin{equation}\label{eqn:2D}
\forall \mathscr{D} \textrm{ $2$-cell},\, |\mathscr{D}|\le C\omega(\mathscr{D}).
\end{equation}

By the conformal symplectic property, we have:
\[
\omega( f^n(\mathscr{D}) ) = \eta^n  \omega(\mathscr{D}),
\]
and, by  \eqref{eqn:area} and \eqref{eqn:2D},
\[
\begin{split}
\omega( \Omega_+^{f^n(\Psi)} \circ  f^n(\mathscr{D}) ) &\le
\eta^n  \omega(\mathscr{D}) +  C e_n \omega(f^n(\mathscr{D}))\\
&=\eta^n(1+ C e_n) \omega(\mathscr{D})
\end{split}
\]
with $ |e_n| \le C C_+ \rho  \lambda_+^n$.

Finally:
\[
\begin{split}
\omega(\Omega_+^{\Psi}(\mathscr{D})) &=
\omega( f^{-n} \circ \Omega_+ ^{f^n(\Psi)}\circ  f^n(\mathscr{D}) ) \\
&\le
\eta^{-n}\left(  \eta^n  (1+e_n) \omega(\mathscr{D}) \right) = (1+e_n) \omega(\mathscr{D}) .
\end{split}
\]
Taking the limit as $n\to \infty$ we get the result.

So, we proved the desired result with the same rate assumptions \eqref{eqn:conformal_rates}
(but of course we need \eqref{eqn:2D} with uniform
bounds, which are guaranteed by \eqref{doundedomega}, as well that the manifold $\Upsilon$ is $2$-dimensional).

\medskip

\paragraph{\em Systematic construction of approximations to scattering map}
\label{systematic}

The method developed above
shows that  we can reconstruct the scattering
map through approximations that get washed away by \eqref{eqn:intertwining}.
Using the procedure above we can pass some properties of the
approximations of $\Omega_+ ^{f^n(\Gamma)}$ from the approximations to $\Omega_+^\Gamma$.

Even if Lemma~\ref{lem:basic} is  not used in this paper, we could use
it  to control perturbations (some version of
this was used in \cite{GideaLM21}) or to
generate numerical approximations.

\begin{lem}  \label{lem:basic}
  Assume that $\Upsilon^n: f^n(\Gamma) \rightarrow \Lambda, n> 0$
  satisfies
  \[
  \begin{split}
    &\lim_{n \to \infty}
(\mu_-^{n} )    \| \Upsilon^n -\Omega_+^{ f^n(\Gamma)} \|_{\C^0}  = 0 \\
    \end{split}
  \]

Then,
  \[
  \lim_{n \to \infty}
  \| \Omega_+^\Gamma - f_{\mid\Lambda}^{-n}\circ\Upsilon^n  \circ f^n \|_{\C^0} = 0
  \]
\end{lem}

\begin{proof}
  The proof of  Lemma~\ref{lem:basic} is immediate using  formula \eqref{eqn:intertwining} and the
  mean value theorem.
\end{proof}

\begin{lem}  \label{lem:basic2}
  Assume that $\Upsilon^n: f^n(\Gamma) \rightarrow \Lambda, n> 0$ is $\C ^1$ and
  satisfies
  \[
  \begin{split}
    &\lim_{n \to \infty}
  (\mu_-^2 \mu_+)^n  n^2   \| \Upsilon^n -\Omega_+^{ f^n(\Gamma)} \|_{\C^1} = 0 \\
    \end{split}
  \]
Then,
  \[
  \lim_{n \to \infty}
  \| \Omega_+^\Gamma - f_{\mid\Lambda}^{-n}\circ\Upsilon^n  \circ f^n \|_{\C^1} = 0
  \]
\end{lem}

\begin{proof}
  Using the estimates on the rates of growth of higher derivatives from
  \cite[Proposition 15]{DLS08} -- recall that in this paper
  we are assuming condition~\eqref{eqn:mubounds} -- and condition \eqref{U3}
  we obtain:
  \[
  \| D^2 f_{\mid\Lambda}^{-n} \| \le C_1   \mu_- ^{2  n} n^2
  \]
  and, therefore,
  \[\begin{split}
  \| f_{\mid\Lambda}^{-n} \circ  \Omega_+^{f^n \Gamma} \circ f^n
  - f_{\mid\Lambda}^{-n}\circ\Upsilon^n  \circ f^n \|_{\C^1}
  &\le C_1 \mu_-^{2n} n^2 \|   \Omega_+^{f^n \Gamma} \circ f^n -
  \Upsilon^n  \circ f^n \|_{\C^1}\\ \le&
  C \mu_-^{2n} n^2 \|   \Omega_+^{f^n \Gamma} - \Upsilon^n \|_{\C^1}  C_2 \mu_+^n
  \end{split}\]
\end{proof}

Now, we prove that if the approximated map $\Upsilon^n$ is approximately symplectic
in the weak sense this implies that the map $\Omega_+^\Gamma$ is symplectic.
The following norm is natural
\[
  [[ (\Upsilon^n)^* \omega - \omega ]] \equiv
  \sup_{\A} \frac{| \int_{\Upsilon^n (\A)} \omega - \int_\A \omega|}{|\A|}
  \]
where the supremum is taken over all $\A$, $\C^1$ 2-cells in $\Gamma$,
and $|\cdot|$ is the Riemannian area.

\begin{lem}
  \label{lem:basic3}

  Assume that we are in the conditions of Lemma~\ref{lem:basic2}

  Assume furthermore:
  \begin{equation}\label{eq:aprox}
  \begin{split}
    &\lim_{n \to \infty}
   (\eta ^{-1} \mu_+ ^{2})^n [[( \Upsilon^n)^* \omega - \omega ]] = 0
       \end{split}
  \end{equation}

Then
$\Omega_+^\Gamma$ is symplectic.
\end{lem}
\begin{proof}
Let $\A$ be  a 2-cell in $ \Gamma$.
We have  $\omega( f(\A ) )= \eta \omega(\A)$.

Denote  $\omega(\A) = a$. We have:
  \begin{itemize}
  \item
   $\omega(f^n(\A)) =  \eta^n a$,
  \item
  By hypothesis \ref{eq:aprox},
  $\omega(\Upsilon^n\circ f^n(\A)) = \eta^n a +  \eps_n|f^n(\A)|$
  with $\eps_n \rightarrow 0$.
  \item
Then,
  $\omega( f_{\mid\Lambda}^{-n} \circ
\Upsilon^n\circ f^n(\A)) =\eta^{-n} ( \eta^n a +  \eps_n |f^n(\A)|)
= a +\bar \eps_n
$
with $|\bar \eps_n| \le C  \eta ^{-n} \mu_+ ^{2n}\eps_n  |\A|$.
By the assumption $ (\eta ^{-1} \mu_+ ^{2})^n \eps_n\to 0$,  we obtain that $\omega( f_{\mid\Lambda}^{-n}\circ \Upsilon^n\circ f^n(\A)) \to a$  as $n\to\infty$.

\item
Since by Lemma \ref{lem:basic2} we have $\| \Omega_+^\Gamma - f_{\mid\Lambda}^{-n}\circ\Upsilon^n  \circ f^n \|_{\C^1} \to 0$ this implies that $\omega(\Omega_+^\Gamma(A))=a=\omega(A)$, which is the integral version of $\Omega_+^\Gamma$ being symplectic.
\end{itemize}
\end{proof}

\subsection{Proof of part (C) of Theorem \ref{thm:main2} on the exact
  symplecticity of the scattering map}\label{sec:proofpartC}

We give two proofs that the scattering map is exact symplectic (even if the map $f$ is not).
The first proof is based on Stokes theorem, and the second one on Cartan's magic formula.

\subsubsection{Proof of part (C) of Theorem \ref{thm:main2} based on Stokes' theorem}\label{sec:proofpartCStokes}

To prove  that the scattering map is
exact symplectic, we prove this property for  $\Omega_+$ (a similar argument applies to $\Omega_-$).

We perform a construction similar
to that in the proof of   part {\bf(B)} of Theorem \ref{thm:main2} given in section \ref{sec:main_2_B_proof_v}.

Let $\sigma \subset W^{s,\mathrm{loc}}_\Lambda$ be a 1D-cell 
parameterized by
\[
\begin{split}
\sigma:[0,1]&\to W^{s,\mathrm{loc}}_\Lambda \\
u &\mapsto \sigma(u).
\end{split}
\]

We complete $\sigma$  to a 2D-cell $\tilde{\sigma}$ contained in $W^{s,\mathrm{loc}}_{\Omega_+(\sigma)}$  by
\[
\tilde{\sigma}(t,u)=\gamma(t;\sigma(u), \Omega_+(\sigma(u))), \quad (t,u)\in [0,1]\times [0,1],
\]
where
\[
\tilde{\sigma}(0,u)=\sigma(u), \
\tilde{\sigma}(1,u)= \Omega_+(\sigma(u))
\]
and the path $\gamma$ is defined as in \eqref{eqn:gamma_def}.
See Fig.~\ref{stokes_1_cell}.

\begin{figure}
\centering
\includegraphics[width=0.5\textwidth]{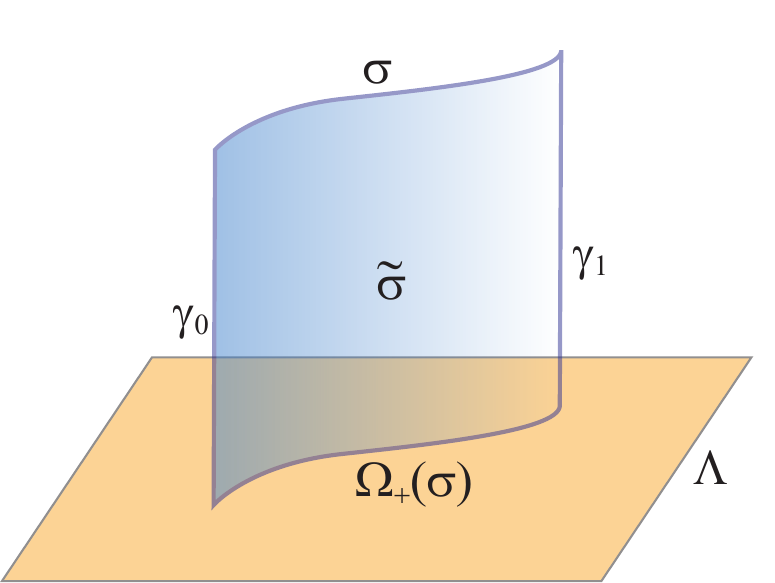}
\caption{A 1-cell $\sigma$ and its completion to a 2D-cell $\tilde{\sigma}$.}
\label{stokes_1_cell}
\end{figure}

We note that,
 $\tilde \sigma$ is contained in $W^{s,\mathrm{loc}}_{\Omega_+(\sigma)}$
and,  by Proposition \ref{lem:degeneracy},
we know that $\omega$ vanishes on $W^{s,\mathrm{loc}}_{\Omega_+^\Gamma(\sigma)}$, so
\[
\int_{\tilde\sigma}\omega= 0 .
\]
Then, using Stokes' Theorem, we obtain
\begin{equation}\label{eq:stokes1}
0=\int_{\tilde\sigma}\omega=\int_{\tilde\sigma}d\alpha=\int_{\partial\tilde\sigma}\alpha=\int_\sigma \alpha -\int_{\Omega_+(\sigma)}\alpha
+\int_{\gamma_1}\alpha
-\int_{\gamma_0}\alpha
\end{equation}
where
\[
\begin{split}
\gamma_1(t)=\tilde \sigma(t,1)=\gamma(t;\sigma(1),\Omega_+(\sigma(1)), \\
\gamma_0(t)=\tilde \sigma(t,0)=\gamma(t;\sigma(0),\Omega_+(\sigma(0)), \ t\in[0,1].
\end{split}
\]

Now, we define the following function on $W^{s,\mathrm{loc}}_\Lambda$:
\begin{equation}\label{eq:primitiveomega+}
P^+(x)=\int_{\gamma(\cdot;x,\Omega_+(x))}\alpha, \textrm { for } x\in W^{s,\mathrm{loc}}_\Lambda .
\end{equation} Note that
$P^+=0$ on $\Lambda$.
The function $P^+$, clearly depends on $\alpha$, but when $\alpha$ is fixed,
we will not include it in the notation. See Remark~\ref{rem:primitiveomega+} for the effects of
changing $\alpha$.

\begin{lem}\label{functiondefined}
The integral defining  $P^+(x)$
in \eqref{eq:primitiveomega+} does not depend on the path $\gamma(\cdot;x,\Omega_+(x))$ in
$W^{s,\mathrm{loc}}_{\Omega_+(x)}$ chosen to connect $x$ to $\Omega_+(x)$.
\end{lem}

Lemma~\ref{functiondefined}  shows that $P^+$ is indeed well defined as a
function on  $W^{s,\mathrm{loc}}_\Lambda$.

\begin{proof}
  Take another path $\tilde \gamma(\cdot;x,\Omega_+(x))$
contained in $W^{s,\mathrm{loc}}_{\Omega_+(x)}$, and call
\[
\tilde P(x)=\int_{\tilde \gamma(\cdot;x,\Omega_+(x))}\alpha .
\]
We know that
$d\alpha_{\mid W^s_{\Omega_+(x)}}=0$.
Since $W^s_{\Omega_+(x)}$ is simply connected  (by Theorem \ref{localstable}(IV)-(ii)) we have that $\gamma \cup \tilde \gamma$
bounds a 2D-cell $\mathcal{B}$   in $W^{s,\mathrm{loc}}_{\Omega_+(x)}$. Applying
Stokes theorem we get:
\[
\begin{split}
0= \int _{\mathcal{B}}d\alpha=
\int _{\gamma \cup \tilde \gamma}\alpha=P^+(x)-\tilde P(x) .
\end{split}
\]
\end{proof}

Using the function $P^+$  in \eqref{eq:primitiveomega+}, by \eqref{eq:stokes1} we obtain:
\[
\int_{\sigma}\Omega_+^*(\alpha)=\int_\sigma \alpha
+P^+(\sigma(1))-P^+(\sigma(0))= \int_\sigma \alpha +\int_{\sigma}dP^+ .\]

As this is true for any $1$-cell $\sigma$, we therefore have proved that
\begin{equation}\label{locally_exact}
\Omega_+^*\left(\alpha_{\mid \Lambda}\right) = \alpha_{\mid W^{s,\textrm{loc}}_ \Lambda} + dP^+ ,
\end{equation}
where  $P^+$ given by the formula  \eqref{eq:primitiveomega+} is the  primitive function of $\Omega_+$
with respect to $\alpha$ which satisfies $P^+_{\mid \Lambda} = 0 $.

\medskip

More important for our purposes,  we can restrict \eqref{locally_exact}
to the homoclinic channel $\Gamma$ and obtain:

\[
\left(\Omega_+^\Gamma\right)^*\left(\alpha _{\mid \Lambda}\right) = \alpha_{\mid \Gamma} + dP^+_{\mid \Gamma} ,
\]

A similar argument yields $\Omega_-$ is exact, with the primitive function $P^-$,
given by an integral formula similar to \eqref{eq:primitiveomega+}.

\medskip

Using the elementary calculation from Lemma \ref{exact_composition}, it follows
that the scattering map $S=\Omega^\Gamma_+\circ (\Omega^\Gamma_-)^{-1}$ is exact.
The derivation below (which uses \eqref{eqn:primitivefg} and \eqref{eqn:primitiveinverse}) gives a formula for the primitive function for $S$:
\begin{equation*}
  \begin{split}
  P^S &=
  P^{
  \Omega_+^\Gamma \circ (\Omega_-^\Gamma)^{-1}}=
P^ {({\Omega_-^\Gamma})^{-1}}+P^{{\Omega_+^\Gamma}}\circ  ({\Omega_-^\Gamma})^{-1} \\
  & = - P^{\Omega_-^\Gamma }\circ  (\Omega_-^\Gamma)^{-1}
+ P^{\Omega_+^\Gamma}\circ ({\Omega_-^\Gamma})^{-1},
\end{split}
    \end{equation*}
so, using the notation from above,
\begin{equation*}\label{eqn:primitive_scattering}
  P^S=(P^+-P^-)\circ(\Omega_-^\Gamma)^{-1}.
\end{equation*}

\begin{rem}
Assume that the symplectic form is exact $\omega=d\alpha$.
Since $\omega$ vanishes on $W^{u,s}_x$, and  $W^{u,s}_x$ is simply connected (see Theorem \ref{localstable}),
by applying the Poincar\'e Lemma we obtain that the restrictions of $\alpha$ to the stable/unstable fibers  $\alpha_{\mid W^{u,s}_x}$ are exact.
Then  \eqref{eq:primitiveomega+} shows that $-P^\pm_{\mid W^{s,u}_x}$ is a primitive of $\alpha_{\mid W^{s,u}_x}$.
\end{rem}

\begin{rem}\label{rem:primitiveomega+}
The integral formulas \eqref{eq:primitiveomega+} are for a fixed  action function $\alpha$.

  For the primitive $\tilde \alpha$
  related to the $\alpha$ by a gauge transformation, that is
  $\tilde \alpha = \alpha + dG$ with $G$ a real valued function on $M$,
   on $W^{s}_\Lambda$  we have:
  \begin{equation}\label{eqn:gauge_Omega}
  P^+_{\tilde \alpha} = P^+_\alpha + G \circ \Omega_+ - G.
  \end{equation}
  A similar formula holds for $P^-$ on $W^u_\Lambda$.

  The generating function $P^S_{\tilde \alpha}$ satisfies the following on the domain  $H^- \subset \Lambda$ of $S$:
  \[
  P^S_{\tilde \alpha} = P^S_\alpha + G \circ S   - G.
  \]
\end{rem}

\subsubsection{Proof of part (C) of Theorem \ref{thm:main2} based on Cartan's magic formula}\label{sec:Cartanexact}

With the same notation as in Section \ref{sec:magic}  we compute:
\[
\frac{d}{dt} \phi_t^* \alpha = \phi^*_t\left[   i_V d\alpha +   d (i_V\alpha)  \right]
 =  \phi_t ^* (d i_V \alpha)  =  d ( \phi^*_t ( i_V\alpha)),
\]
where we have  used $i_V (d\alpha)   = i_V (\omega) = 0$.
The last equality is because of the Vanishing Lemma~\ref{lem:vanishingmanifolds}.

Therefore we have
\[
((\Omega_+^\Gamma)^{-1})^*\alpha - \alpha  = \phi_1^* \alpha -  \phi^*_0 \alpha
= d   \int_0^1   \phi^*_t ( i_V\alpha)\, dt,
\]
showing that $(\Omega_+^\Gamma)^{-1}$  is exact. The fact that $\Omega_+^\Gamma $ is exact follows from applying \eqref{eqn:primitiveinverse} to the map $(\Omega_+^\Gamma)^{-1}$   for $\eta=1$. We also obtain the direct formula:
\[(\Omega_+^\Gamma)^*\alpha - \alpha=-d\left[   \int_0^1  \left (\phi_t\circ (\Omega_+^\Gamma))^* ( i_V\alpha)\, dt\right)\right].\]

\begin{rem} Note that this proof also gives that, if we have
a presymplectic manifold and consider the
foliation by the kernel, the holonomy maps between
two transversals (which are symplectic manifolds since the kernel
is excluded) are symplectic maps.\end{rem}

\subsection{Proof of Proposition \ref{lem:degeneracy} }\label{sec:proofisotropic}

The first item of Proposition \ref{lem:degeneracy} is a direct consequence of part {\bf(B)} of Theorem \ref{thm:main2}.
For $\Omega_{+\mid W^s_ N}:  W^s_ N \to N$ with $N\subset \Lambda$, as
$ (\Omega_+)^* (\omega_{|\Lambda}) = \omega_{|W^{s,\mathrm{loc}}_\Lambda}$, using that $\omega_{\mid N}=0$,  we get:
\[
0=(\Omega_{+\mid W^s_ N})^* \omega_{\mid N} = \omega_{\mid W^s_N}.
\]

To prove the second item, let  any $y \in W^{s,\mathrm{loc}}_\Lambda$. We have that $y \in W^{s,\mathrm{loc}}_x$ for
$x =\Omega_+(y)\in \Lambda$.
Using that $\Lambda$ is
symplectic (Part (A) of Theorem~\ref{thm:main1})
we can construct as isotropic manifold $N \subset \Lambda$ of
dimension $\frac{d_c}{2}$ (using, for example Darboux theorem) with $x\in N$.
Using Proposition~\ref{lem:degeneracy} part (i),
we obtain that $W^s_N$ is isotropic.
We note that $y \in W^s_N \subset W^s_\Lambda$
and that
\[
\text{dim}( W^s_N)  = \frac{d_c}{2} + d_s = \frac{1}{2}(d_c + d_s +d_u).
\]
Therefore $W^s_N$ is a Lagrangian submanifold of $M$.

Therefore,  $T_y W^s_N$ is a Lagrangian subspace of
$T_yM$.  Since $T_y W^s_N \subset T_y W^s_\Lambda$,
we conclude that  $T_y W^s_\Lambda$ is coisotropic,
and since $y$ was an arbitrary point, the manifold $T_y W^s_\Lambda$
is coisotropic but not Lagrangian.

\section{Formulas for the primitive functions of wave maps and  scattering map when $f$ is exact}
\label{sec:primitive_series}

When $f$ is exact conformally symplectic,  in this section
we obtain formulas for the primitive functions of  the wave maps
and the scattering map in terms of the primitive function of $f$.
The variational formulation for conformally symplectic systems
is given in \eqref{Aformal}.
The formulas  \eqref{sum_path_plus}
\eqref{sum_path_minus} and \eqref{eqn:primitive_SM_good_action}
provide a link with the calculus of variations for
conformally symplectic systems.
In the symplectic case,
\cite{Angenent93, Lomeli97,AmbrosettiB98}
develop variational descriptions of heteroclinic
connections. Such formulas have been used in numerical
calculations of orbits homoclinic (or heteroclinic) to periodic orbits  \cite{Tabacman95}
of twist maps.

\medskip

Fixing an action form $\alpha$, we rearrange the definition of an exact conformally symplectic map \eqref{eqn:exact_conformally_symplectic}
as
\begin{equation}\label{exact_cs}
\alpha = \eta^{-1} f^* \alpha -d \eta^{-1}P^f_\alpha .
\end{equation}

Applying formula \eqref{exact_cs}
repeatedly,   we obtain for any $N \in \N$
\begin{equation}\label{explicit_forward}
	\alpha = \eta^{-N} (f^*)^N  \alpha -
  d \left(\sum_{j= 0}^{N-1}  \eta^{-j-1} P^f_\alpha \circ f^{j} \right).
\end{equation}

Similarly, rearranging \eqref{eqn:exact_conformally_symplectic} as
\[
\alpha = \eta (f^*)^{-1} \alpha + d P^f_\alpha\circ f^{-1}
\]
and
iterating we have:
\begin{equation}\label{explicit_backward}
  \alpha  = \eta^{N} (f^*)^{-N}\alpha +
  d \left(\sum_{j= 1}^N  \eta^{j-1} P^f_\alpha \circ f^{-j} \right).
 \end{equation}

 Integrating \eqref{explicit_forward}  and \eqref{explicit_backward} over a path $\sigma$ and
 remembering that the integral of a differential over a path  is just the difference of the values at the ends, we obtain for
 any path $\sigma$:
 \begin{equation} \label{explicit_forward_path}
   \int_\sigma \alpha =  \eta^{-N} \int_{f^N (\sigma)} \alpha
   - \sum_{j= 0}^{N-1}  \eta^{-j-1} \left( P^f_\alpha \circ f^{j}(\sigma(1)) -
     P^f_\alpha \circ f^{j}(\sigma(0)) \right),
 \end{equation}
 \begin{equation} \label{explicit_backward_path}
   \int_\sigma \alpha =  \eta^{N} \int_{f^{-N} (\sigma)} \alpha
   +\sum_{j= 1}^N  \eta^{j-1} \left( P^f_\alpha \circ f^{-j}(\sigma(1)) -
     P^f_\alpha \circ f^{-j}(\sigma(0)) \right).
 \end{equation}

Given $x \in W^{s,\mathrm{loc}}_{\Lambda}$, when  $\sigma$ is chosen to be the path $\gamma_+(\cdot;x,\Omega_+(x)) \subset W^{s,\mathrm{loc}}_{\Omega_+(x)}$ given in \eqref{eqn:gamma_def}, and denoting by
\[
\gamma^N_+:=f^N(\gamma_+(\cdot;x,\Omega_+(x))=
\gamma_+(\cdot;f^N(x),\Omega_+(f^N(x))),\]
using the formula \eqref{eq:primitiveomega+} for the primitive function $P^+_\alpha$
(and, similarly, in the analogous formula for $P^-_\alpha$), we obtain:

\begin{lem}\label{lem:sum_path_plus_minus}
  The primitive functions $P^\pm_\alpha$ of $\Omega_\pm$  for the action   form $\alpha$ are given by:
\begin{equation}
  \label{sum_path_plus}
\begin{split}
P^+_\alpha(x)=&\eta^{-N} \int_{\gamma^N_+}\alpha +\sum_{j=0}^{N-1}\eta^{-j-1}[P^f_\alpha\circ f^j(\Omega_+^\Gamma(x))-P^f_\alpha\circ f^j(x)],
 \end{split}
\end{equation}
\begin{equation}\label{sum_path_minus}
\begin{split}
P^-_\alpha(x)=&\eta^{N} \int_{\gamma^N_-}\alpha +\sum_{j=1}^{N}\eta^{j-1}[P^f_\alpha\circ f^{-j}(\Omega_-^\Gamma(x))-P^f_\alpha\circ f^{-j}(x)].
 \end{split}
\end{equation}

The primitive function $P^S_\alpha$ of the scattering map is given by (see \eqref{eqn:primitivefg} and \eqref{eqn:primitiveinverse}):
\begin{equation}
  \label{sum_path_plus_minus}
\begin{split}
P^S_\alpha(x)&=\left[\eta^{-N} \int_{\gamma^N_+}\alpha +\sum_{j=0}^{N-1}\eta^{-j-1}[P_\alpha^f\circ f^j(\Omega^\Gamma_+)-P_\alpha^f\circ f^j]\right]\circ (\Omega^\Gamma_-)^{-1}(x)\\
&-\left[\eta^{N} \int_{\gamma^N_-}\alpha +\sum_{j=1}^{N}\eta^{j-1}[P_\alpha^f\circ f^{-j}(\Omega^\Gamma_-)-P_\alpha^f\circ f^{-j}]\right ]\circ (\Omega^\Gamma_-)^{-1}(x)\\
=&\eta^{-N} \int_{\gamma^N_+}\alpha\circ (\Omega^\Gamma_-)^{-1}(x) +\sum_{j=0}^{N-1}\eta^{-j-1}[P_\alpha^f\circ f^j\circ S(x) -P_\alpha^f\circ f^j\circ (\Omega^\Gamma_-)^{-1}(x)] \\
&- \eta^{N} \int_{\gamma^N_-}\alpha \circ(\Omega^\Gamma_-)^{-1}(x) -\sum_{j=1}^{N}\eta^{j-1}[P_\alpha^f\circ f^{-j}(x)-P_\alpha^f\circ f^{-j}\circ(\Omega^\Gamma_-)^{-1}(x) ] .
 \end{split}
\end{equation}
\end{lem}

\begin{rem} \label{rem:converegence_equivalent}
The function  $P^+_\alpha$ in \eqref{sum_path_plus} is well defined for all $x\in W^s_\Lambda$. Therefore, the series on the right-hand side of \eqref{sum_path_plus} is convergent if and only if  the sequence $\eta^{-N}\int_{\gamma^N_+}\alpha$
  converges to zero. Analogously for  $P^-_\alpha$ in \eqref{sum_path_minus}.
\end{rem}

The convergence of the series in \eqref{sum_path_plus}
and \eqref{sum_path_minus}  is very easy if the orbits of $f$ in $\Lambda$ are bounded. In such
a case we have uniform bounds on $ \alpha $ and the
paths $\gamma^N_\pm$ have lengths
bounded by $\lambda_\pm^{|N|}$ (when $N \rightarrow \pm \infty$).

However, even if $\omega$ is uniformly bounded,  $\alpha$ may
be unbounded (see  Section \ref{sec:unbounded_action} for lower bounds
for all forms).
If $f^N(x)$ escapes to infinity,
it could happen that $\|\alpha_{f^N(x)}\|$ grows
so fast that it overtakes the decrease of the length of $\gamma^N_\pm$ (and the powers of $\eta$).

In Section~\ref{convergence},
we show that there is a gauge in which the
formulas \eqref{sum_path_plus} and \eqref{sum_path_minus}
converge very fast.
Indeed, with the construction of Section~\ref{convergence}, the
formulas become finite  sums.

In Section~\ref{divergence} we will show
that, if there are orbits that escape to infinity, there is
always a gauge that makes
\eqref{sum_path_plus} and \eqref{sum_path_minus}
divergent.

\subsection{Construction of gauges yielding  convergence of the series
   for the primitives of the wave maps and the scattering map}
\label{convergence}

\begin{lem}\label{lem:construction_G}
Given an action form $\alpha$ for $\omega$, there exists an action form $\tilde\alpha=\alpha+dG$, for some smooth function $G:M\to\mathbb{R}$, such that
\begin{equation}
\label{eqn:primitive_SM_good_action}
\begin{split}
P^S_{\tilde\alpha}=& \sum_{j=0}^{\infty}\eta^{-j-1}[P^f_{\tilde\alpha}\circ f^j\circ S -P^f_{\tilde\alpha}\circ f^j\circ(\Omega_-)^{-1}] \\
& -\sum_{j=1}^{\infty}\eta^{j-1}[P^f_{\tilde\alpha}\circ f^{-j} - P^f_{\tilde\alpha}\circ f^{-j}\circ (\Omega_-)^{-1}] .
\end{split}
\end{equation}
\end{lem}

\begin{proof}
  Let $G^+=P^+_\alpha$  given in \eqref{eq:primitiveomega+} ,defined on $W^{s,\textrm{loc}}_\Lambda$, and note that $G^+=0$ on $\Lambda$,
  so, by \eqref{eqn:gauge_Omega} we have
  $P^+_{\alpha+dG^+}=P^+_\alpha  - G^+=0$ on $W^{s,\textrm{loc}}_\Lambda$.

  Note also that \eqref{sum_path_plus}can be written as:
  \[
  P^+_\alpha(x)= P^+_\alpha(F^N(x))+\sum_{j=0}^{N-1}\eta^{-j-1}[P^f_\alpha\circ f^j(\Omega_+^\Gamma(x))-P^f_\alpha\circ f^j(x)].
\]

 Applying this formula for $\alpha+dG^+$ we obtain that the sum is zero.

  Similarly, for $G^-=P^-_\alpha$ on  $W^{u,\textrm{loc}}_\Lambda$, we have $P^-_{\alpha+dG^-}=P^-_\alpha - G^-=0$    on $W^{u,\textrm{loc}}_\Lambda$.

    So, we have accomplished our goal of making the series convergent (in fact zero) using gauge
    functions $G^\pm$ defined on $W^{s,\textrm{loc}}_\Lambda \cup  W^{s,\textrm{loc}}_\Lambda$.
    What remains is to show that this partially defined function can be extended to
    the whole $M$ so that it is a well defined gauge function and makes the series convergent (but not zero).
    We now give the details, which are fairly standard.

Let $\OO_\rho$ be the uniform neighborhood of $\Lambda$ defined in \eqref{eqn:OOrho}.
Choose $\rho'>\rho$ such that   $\OO_{\rho'}$ is disjoint from the homoclinic channel $\Gamma$.

Now we construct a function $G^{ext}:M\to\mathbb{R}$ which agrees with $G^+$ on $W^{s,\textrm{loc}}_\Lambda\cap \OO_\rho$
and with $G^-$ on $W^{u,\textrm{loc}}_\Lambda \cap \OO_\rho$.
We now give the details
of this  extension by using partitions of unity,  and this finishes the proof of
Lemma~\ref{lem:construction_G}.

We can cover $\OO_\rho$  by a countable collection of uniform balls $\B_i$,  $i=1,\ldots,\infty$,  such
that:
\begin{itemize}
\item[(C1)] $
\OO_\rho\subseteq \bigcup_{i} \B_i\subseteq \OO_{\rho'}$;
\item[(C2)]
Each point $x\in \OO_\rho$ is contained in only finitely many balls $\B_{ i_1} ,\ldots, \B_{ i_{L+1} }$  (here $L$ is the covering dimension of the manifold $M$, which equals the dimension of the manifold);
\item [(C3)]
On each $\B_i$ we have a  local trivialization  of $E^s\oplus E^u$, that is,
 \begin{equation}\label{eqn:trivialization}
   (E^s\oplus E^u)_{\Lambda\cap\B_i}\simeq (\Lambda\cap \B_i)\oplus E^s\oplus E^u.
\end{equation}
\end{itemize}

By (C3), on each  open set $\B_i$ we can choose a system of coordinates $(c_i,s_i,u_i)$ such that
\begin{equation}\label{eqn:trivial_coords}
\begin{split}
\Lambda\cap\B_i      =&\{(c_i,s_i,u_i)\,|\, s_i=u_i=0\},\\
W^s_\Lambda\cap\B_i =&\{(c_i,s_i,u_i)\,|\, u_i=0\}, \\
W^u_\Lambda\cap\B_i  =&\{(c_i,s_i,u_i)\,|\, s_i=0\}.
\end{split}
\end{equation}
We note that the  systems of coordinates associated to two open sets $\B_i$ and $\B_j$ that have non-empty intersection do not have to agree with one another.

There exists a smooth partition of unity $\{\Psi_i\}$ subordinate to $\{\B_i\}$, with $\Psi_i:M\to\mathbb{R}$,  such that:
\begin{itemize}
\item
For each $x\in \OO_\rho$ and each $i$ we have $0\le \Psi_i(x)\le 1$;
\item
For each $x\in \OO_\rho$ there is an open neighborhood of  $x$ such that all but finitely many $\Psi_i$'s are $0$ on that open neighborhood of $x$;
\item
For each $x\in \OO_\rho$  we have $\Sigma_i\Psi_i(x)=1$ (by the previous condition this sum is finite on an open neighborhood of $x$);
\item
For each $i$, $\textrm{supp}(\Psi_i)\subseteq \B_i$.
\end{itemize}
This is a direct consequence of \cite[Theorem 3-11]{Spivak}.

Now define $G^0_i: \B_i\to\mathbb{R}$ by
\[
G^{0}_i(c_i,s_i,u_i)= G^+(c_i,s_i)+G^-(c_i,u_i),
\]
where the underlying coordinate system  corresponds to $\B_i$.
Since $G^+_{\mid\Lambda}=G^-_{\mid\Lambda}=0$, we then have $G^0_i=0$ on $\Lambda\cap\B_i$.
The function $G^0_i$ is a local extension of both $G^+$ and $G^-$, and depends on the underlying coordinate system.

Since $\textrm{supp}(\Psi_i)\subseteq \B_i$, we define a global extension $G^{ext}_i:M\to \mathbb{R}$ of $G^0_i$ given by:
\begin{equation*}
G^{ext}_i(x)= \left\{
\begin{array}{ll}
           G^0_i(x)\Psi_i(x)  & \hbox{ on $\textrm{supp}(\Psi_i)\subseteq \B_i$},\\
           0& \hbox{ on $ M\setminus \textrm{supp}(\Psi_i)$}.
\end{array}
\right.
\end{equation*}

Then,  we combine the  functions $G^{ext}_i$ into a single global extension $G^{ext}:M\to \mathbb{R}$ by
\begin{equation}\label{eqn:Gext_def}
\begin{split}
G^{ext}(x)=&\sum_i G^{ext}_i(x).
\end{split}
\end{equation}
Although  a point $x$ may be covered by finitely many open sets $\B_{i_1},\ldots, B_{i_{L+1}}$, in which case the point $x$ has different
coordinate representations
$x=x(c_{i_j},s_{i_j},u_{i_j})$, $j=1,\ldots,L+1$,
we have that
\[G^{ext}(x)= \sum_{j=1}^{L+1} (G^+(c_{i_j},s_{i_j})+G^-(c_{i_j},u_{i_j}))\Psi_{i_j}(c_{i_j},s_{i_j},u_{i_j})\]
is independent of the local system of coordinates.

By the uniformity assumption {\bf (U1)}, for some $C>0$ we also have
\[
\|G^{ext}\|_{\C^r(M)}\le C \left[\|G^+\|_{\C^r(W^{s,\mathrm{loc}}_{\mid \Lambda})} + \|G^-\|_{\C^r(W^{u,\mathrm{loc}}_{\mid \Lambda})}\right].
\]

Note that
\begin{equation}\label{eqn:Gext1}
G^{ext} =
\left\{
\begin{array}{ll}
           G^+  & \hbox{ on $W^{s,\mathrm{loc}}_{ \Lambda} \cap \OO_\rho$},\\
           G^- & \hbox{ on $ W^{u,\mathrm{loc}}_{ \Lambda} \cap \OO_\rho$}.
\end{array}
\right.
\end{equation}
and
\begin{equation}\label{eqn:Gext2}
\begin{split}
G^{ext} =   & 0  \textrm{ on } M\setminus \bigcup_{i} \B_i \implies G^{ext} =    0  \textrm{ on }M\setminus \OO_{\rho'}.
\end{split}
\end{equation}

We now show that for the modified action form $\tilde\alpha=\alpha+dG^{ext}$ the series in \eqref{sum_path_plus} is convergent (in fact, it becomes a finite sum for every point).

Let $x$ be a point in $W^s_\Lambda$. Since $d(f^n(x),\Lambda)\to 0$ as $n\to\infty$, there exists $N$ depending on $x$ such that $f^N(x)\in \OO_\rho$, and so $\gamma^N_+\subseteq  \OO_\rho$.
Then
\begin{equation}\label{eqn:Gext_cancellation}
  \begin{split}
\int_{\gamma^N_+}\alpha+dG^{ext}=&\int_{\gamma^N_+}\alpha+dG^+\\
=&P^+_\alpha(f^N(x))-G^+(f^N(x))=0,
\end{split}
\end{equation}
and therefore, using \eqref{sum_path_plus} yields
\begin{equation}\label{eqn:P_plus_3}
P^+_{\tilde\alpha}(x)= \sum_{j=0}^{N-1}\eta^{-j-1}[P^f_{\tilde\alpha}\circ f^j\circ \Omega_+(x) -P^f_{\tilde\alpha}\circ f^j(x)].
\end{equation}
Thus
\begin{equation}\label{eqn:P_plus_4}
   P^+_{\tilde\alpha} = \sum_{j=0}^{\infty}\eta^{-j-1}[P^f_{\tilde\alpha}\circ f^j\circ \Omega_+ -P^f_{\tilde\alpha}\circ f^j],
\end{equation}
where for each $x$ the above sum is finite.

Similarly
\begin{equation}\label{eqn:P_plus_5}
   P^-_{\tilde\alpha} = \sum_{j=1}^{\infty}\eta^{j-1}[P^f_{\tilde\alpha}\circ f^{-j}\circ \Omega_- -P^f_{\tilde\alpha}\circ f^{-j}].
\end{equation}

By substituting \eqref{eqn:P_plus_4} and  \eqref{eqn:P_plus_5} in \eqref{eqn:primitive_SM_good_action}, we obtain the desired conclusion.
\end{proof}

\subsection{Construction of gauges yielding  divergence of the series for the primitives  of the wave maps  when there are escaping orbits
}\label{divergence}

The purpose of this section is to show that, if  $f_{\mid\Lambda}$ has orbits that escape to infinity\footnote{the orbit $f^n(y)$ escapes to infinity if
every compact $K \subset \Lambda$ contains only finitely many points of the orbit}, there is always a function $G$ (in fact, plenty of them) such that the series \eqref{sum_path_plus}, which give the primitive $P_\alpha^+$ of the wave map $\Omega^+$, corresponding to the modified $1$-form $\alpha+dG$,
\begin{equation}\label{eqn:series0}
\begin{split}
&\sum_{j\ge 0}\eta^{-j-1}[P_{\alpha+dG}\circ f^j(\Omega_+(x))-P_{\alpha+dG}\circ f^j(x)]
\end{split}
\end{equation}
is divergent for some $x$.

More concretely, we have the following:
\begin{lem}
Assume that for a given $\alpha$ the series \eqref{eqn:series0} with $G=0$ is convergent.
Assume also  that there exists a point
$y\in \Lambda$ such that its orbit $y_n=f^n(y)$ escapes to infinity as $n\to \infty$.

Then, there exists $G:M\to \R$, such that the series \eqref{eqn:series0} for $\alpha +dG$
is divergent at the point $y$.
\end{lem}
\begin{proof}

If the  orbit $y_n=f^n(y)$ escapes to infinity, then it has a subsequence $f^{k_n}(y)$ such that for some $\delta>0$ we have
\[
d(y_{k_n},y_{k_{m}})\ge \delta
\]
for all $n,m$ with $m \ne n$.

Take $x\in W^s_y$,
and denote  $x_n =f^n(x)$ and $y_n =f^n(\Omega_+(x))=f^n(y)$.
The points $x_n,y_n$ are the endpoints of a curve $\gamma^n_+(\cdot)=f^n(\gamma_+(\cdot;x, \Omega_+(x)))$ in $W^s_{y_n}$.
Assume that the sequence
$\eta^{-n} \int_{\gamma^n_+}\alpha$ is convergent, otherwise there is nothing to prove (see Remark \ref{rem:converegence_equivalent}).

For any point $z_n=\gamma^n_+(t)$ on this curve,
 we know by  Theorem \ref{localstable} part (II)-(iv) that there exists $\tilde C >0$ such that:
\[
d(z_{k_n},y_{k_n})
\le \tilde C   \lambda_+ ^{k_n} \to 0, \ n\to \infty .
\]
Since $d(y_{k_n},y_{k_{n'}})\ge \delta$ for all $n \ge 0$,  it follows  that, changing $\delta$ if necessary:
\[
d_H(\gamma^{k_n}_+,\gamma^{k_{n'}}_+)\ge \delta , \textrm{ for } n \in \N,\, n \ne n'
\]
where $d_H$ is the Hausdorff distance.

Therefore, we can choose small, open neighborhoods
$V_n$  of the curves $\gamma^{k_n}_+$ in $M$, such that $V_{k_n}\cap V_{k_{n'}}=\emptyset$, and  construct    smooth functions $G_n$ with support in $V_n$ such that:
\[
G_n(x_{k_n})- G_n(y_{k_n})=(\eta\beta)^n, \ n\ge 0,
\]
for some $\beta>1$.
We  write $G = \sum_n G_n$ (recall that the supports of $G_n$ are disjoint).
Using that $x_{k_n}, y_{k_n}$ are in the support of $G_{k_n}$ and not
in the support of any other, we have
\[
G (x_{k_n})- G (y_{k_n})= (\eta\beta)^{k_n} , \ n\ge 0.
\]

Then the series
\begin{equation}
\begin{split}
&\eta^{-k_n} \int_{\gamma^{k_n}_+}\left(\alpha+dG\right)\\
&=\eta^{-k_n}\int_{\gamma^{k_n}_+}\alpha+\eta^{-k_n} (G(y_{k_n})-G(x_{k_n}))\\
&=\eta^{-k_n}  \int_{\gamma^{k_n}_+}\alpha+\beta^{k_n}.
\end{split}
\end{equation}
Since $\beta>1$, the sequence  is divergent.

Going back to formula \eqref{sum_path_plus} and using that   \eqref{eqn:series0}  is convergent, we obtain that the series for $P_{\alpha +dG}$ is divergent.
\end{proof}

\subsection{Variational interpretation of the iterative formulas for primitive functions of scattering map}
\label{sec:variational}

In this section, we discuss the variational interpretation of
\eqref{sum_path_plus_minus}.
The material in this section will not be used in this paper.
The sole purpose of this section is to point out a possible bridge between variational and geometric approaches to heteroclinic jumps.

Let $T^*Q$ be a cotangent bundle of a   manifold -- the standard
example of a symplectic manifold  -- with the canonical
$1$-form defined in coordinates $(p,q)$ as $\alpha_0 = pdq$
\footnote{A geometrically natural  definition  of $\alpha_0$ is
standard, see, for example:
\cite[Proposition 3.2,11 p. 180]{AbrahamM78}.}.
Let $f$ be a mapping on $T^*Q$ that is homotopic to the identity,   exact conformally symplectic
($f^*\alpha-\eta\alpha=dP^f_\alpha$ for some function $P^f_\alpha$ on $T^*Q$), and satisfies a \emph{twist condition},
\footnote{
The twist condition  is clearly
non-generic -- it is not
verified by the identity--  but  it is verified by the
geodesic flow of a compact manifold at time $t > 0 $ for small enough $t$
\cite{Gole94}.}
meaning that if $f(p,q)=(p',q')$ then $(q,q')$ gives a system of coordinates on $Q\times Q$.
The primitive function $P^f_\alpha$ has the interpretation
of an action and that the orbits of $f$ are critical points
of the formal action

\def\SS{\mathscr{S}}
\begin{equation}\label{Aformal}
  \SS(x) = \sum_n \eta^{-n} P^f_\alpha(x_n),
\end{equation}
where $x =(x_n)_n=(f^n(x))_n$.
See \cite{Haro00} for the  symplectic case $\eta = 1$.

\begin{rem}
The variational principle in \eqref{Aformal},
often called \emph{discounted} variational principle,
appears naturally in finance. Each $x_n$ is a transaction at time $n$
and $P^f_\alpha(x_n)$ is the cost of the transaction in currency.
If there is constant inflation, to evaluate the cost
of a strategy, it is natural to add the costs at different times
by reducing them to a common time \cite{Bensoussan88}.
The conversion of the cost $P^f_\alpha(x_n)$ to currency
at $n = 0$ is $\eta^{-n} P^f_\alpha(x_n)$.

The variational principles \eqref{Aformal} also appear in control theory under the
name \emph{finite horizon} approximation \cite{KaliseKR17}.
\end{rem}

We can write the function $P^f_\alpha$ on the manifold  in
a coordinate patch as $S(q,q')$, so
that the conformally symplectic property
can be written as
\[
p' dq' = \eta p dq + dS(q,q') .
\]
This is equivalent to
\begin{equation} \label{motion}
\begin{split}
  & p' =  \partial_2 S(q, q') , \\
  &  p =  - \eta^{-1}  \partial_1 S(q,q') .
\end{split}
\end{equation}

A sequence of points $(p_n,q_n)_{n \in \Z}$
is an orbit of  \eqref{motion} is equivalent
to the sequence $\{q_n\}_{n \in \Z}$
being a critical point of the
formal action\footnote{We recall that the critical points of
a formal action $\SS(q)$ are obtained by setting to
zero  the derivatives with respect to all arguments $q_n$ (ignoring
all the terms in the sum which do not involve $q_n$).
In our case,
the condition of equilibrium
becomes $\forall n,\quad \eta^{-(n-1)} \partial_2 S(q_{n-1}, q_n)+\eta^{-n} \partial_1 S(q_n, q_{n+1}) = 0$,
which is equivalent to \eqref{motion}.
}
\[
\SS(q) = \sum_{n \in \Z}
\eta^{-n} S(q_n, q_{n+1}).
\]

Given any real valued functions $\phi_n$
we can consider instead of the formal variational principle
\eqref{Aformal}, the variational principle
\[
\SS_\phi (x) = \sum_n \eta^{-n} P^f_\alpha(x_n) + \phi_n
\]
Clearly, the critical points of $\SS$ and $\SS_\phi$
are the same.
By making choices on the function $\phi_n$,
we can ensure that, for some sequences, the functional $\SS_\phi$
is well defined. For example, if $y=(y_n)_n$ is an orbit
we can imagine that taking $\phi_n  = -\eta^{-n} P^f_\alpha(y_n)$,
\begin{equation}\label{Arenormalized}
\SS_\phi (x) = \sum_n \eta^{-n}(  P^f_\alpha(x_n) -  P^f_\alpha(y_n) )
\end{equation}
the functional $\SS_\phi$
(sometimes called \emph{renormalized action})
may be well defined as a true functional for orbits that are fast
asymptotic in the future and in the past to $y_n$.

Hence, taking the variational principle consisting of the
primitive of the scattering map and the primitive of the map
gives a variational principle for orbits of the map that include
homoclinic excursions.

The variational approach to homoclinic orbits has also been used
as a numerical tool for symplectic  systems
\cite{Tabacman95,MacKayMS89}.
Using that
\eqref{explicit_forward_path}, \eqref{explicit_backward_path}
can be computed by other methods and
considered as boundary terms, it seems that the method could
be adapted to conformally symplectic systems.

\subsection{Unbounded  action forms}
\label{sec:unbounded_action}
It is well know that a $(2n)$-dimensional, connected, Riemannian manifold  $M$ that is closed (i.e., boundaryless and compact), cannot have a symplectic form $\omega$ that is exact.
  Indeed, if $\omega = d\alpha$ then
    $\omega^n = d(\alpha \wedge \omega^{n-1})$, and Stokes' Theorem  implies
    \begin{equation}\label{identity}
    \int_M \omega^n =  \int_M d(\alpha \wedge \omega^{n-1})
    = \int_{\partial M} \alpha \wedge \omega^{n-1} = 0 ,
    \end{equation}
    which contradicts the fact that $\omega^n$ is a volume form on $M$.
Therefore, if such a  manifold  has an exact symplectic form, the manifold cannot be closed.

\medskip

Below we show that if the  manifold has an exact symplectic form $\omega=d\alpha$ and satisfies some additional conditions, then $\|\alpha\|$ must be unbounded. Moreover, we can provide some quantitative estimates on the growth of $\|\alpha\|$ along geodesic balls $B_R$.

\begin{rem}\label{ex:noncompact_bounded}
  A symplectic manifold with an exact symplectic form can be with boundary or non-compact, but does not need to be unbounded. Of course, such phenomenon can only happen for manifolds
  which do not satisfy {\bf (U1)} or {\bf (U1')}

For instance, $M$  can be a bounded cylinder (with or without boundary)
\[M=\{(I,\theta)\,|\, I \in  B^n_R,\, \theta\in \T^n\},\]
where $B^n_R$ is a ball in $\R^n$ (open or closed).
The standard symplectic form $\omega=dI\wedge d\theta$ is exact with action form $\alpha=I d\theta$.
Moreover, $\alpha=I d\theta$ is bounded on $M$.

\end{rem}

We recall that if the Riemannian metric  is complete, then every geodesic can be extended to a all times.
Fixing a point $o\in M$, the geodesic ball $B_R$ is the set of points   $x\in M$ for which $d(o, x)\le R$.
The distance $d(0,x)$ is a smooth function in $x$ except for the cut locus\footnote{The cut locus consists of points that are conjugate to $o$ and points that have multiple minimal geodesics connecting them to $o$.} of $o$.
\medskip

Denote by $\textrm{Vol}_{d}$ the Riemannian volume on a $d$ dimensional manifold.

The form $\omega$ is  said to be \emph{uniformly nondegenerate}, if there exists a constant $C>0$ such that
 \begin{equation}\label{eqn:uniform_nondegenerate}
  |\omega^n| \ge C \cdot d{\textrm{Vol}_{2n}}
  \end{equation}
 where $d{\textrm{Vol}_{2n}}$ is the Riemannian volume form
 and we recall that volumes can be compared.

  \begin{lem}  \label{lem:lowerbound}
    Assume that $M$ is a $2n$-dimensional, connected,  complete Riemannian manifold, and $\omega=d\alpha$ is an exact symplectic form on $\omega$.

Assume  that $\omega$ is bounded on $M$ and uniformly nondegenerate.
Let $B_R$ be a geodesic ball of radius $R>0$ in $M$ such that $B_R$ and  $\partial B_R$
are piecewise differentiable manifolds.

Then, there exists a constant $\bar C>0$, depending on $\omega$ but not on $\alpha$, such that
  \begin{equation}
    \label{eqn:lowerbounds}
\begin{split}
  & \int_{\partial B_R} |\alpha| \ge
  \bar{C}\cdot \textrm{Vol}_{2n}(B_R), \\
  & \sup_{x \in  \partial B_R}\|\alpha(x)\| \ge
   \bar{C} \cdot \textrm{Vol}_{2n}(B_R)  / \textrm{Vol}_{2n -1}(\partial B_R)  .
   \end{split}
  \end{equation}

\end{lem}

\begin{proof}
Using  the uniform non-degeneracy of $\omega$, \eqref{eqn:uniform_nondegenerate},
 the assumption that $\omega$ is bounded,
and  Stokes' Theorem, we have
  \[
  \begin{split}
  C \cdot\textrm{Vol}_{2n}(B_R) \le & \left|\int_{B_R} \omega^n \right|=\left| \int_{\partial B_R} \alpha \wedge \omega^{n-1} \right|\le C'  \int_{\partial B_R}
 \| \alpha\|\\
  \le& C'    \cdot \sup_{x \in \partial B_R} \| \alpha(x)\|\cdot\textrm{Vol}_{2n-1}(\partial B_R)
  \end{split}\]
  for some $C' >0$, depending on the norm of $\omega$.
\end{proof}

Note that  right hand sides of \eqref{eqn:lowerbounds} has a factor depending on $R$ that
depends only on
the Riemannian metric. The symplectic properties enter only as a constant.
Hence, we obtain that any $\alpha$ has to be unbounded using only properties of
the Riemannian metric.

\begin{ex} An application of the Lemma~\ref{lem:lowerbound}
is when $M  = \R^n \times  \T^n$ with the standard symplectic form $dI\wedge d\theta$, where $(I,\theta)\in \R^n\times\T^n$.
In such a case, $\textrm{Vol}_{2n}(B_R) \approx  C_1 R^{2n}$,
$\textrm{Vol}_{2n-1}(\partial B_R) \approx  C_2 R^{2n-1}$.
Hence we obtain
\begin{equation}\sup_{x \in \partial B_R}  \|\alpha(x)\|  >  C\cdot R   \textrm { for all $R$ large.}
\label{linearbound}
\end{equation}

We conclude that any action form in the manifold has to grow
linearly. The standard symplectic form and action form
saturate the bound and show that the result
cannot be improved.
\end{ex}

\begin{rem} \label{tautologicalform}
When $M  = T^* Q$ -- the symplectic manifold is the
cotangent bundle of a compact Riemannian manifold $Q$ --
we see that there is $C'>0$ such that $\sup_{x \in \partial B_R}  \|\alpha (x) \|  \approx  C'\cdot R   \textrm { for all $R$ large}$.
This shows that the inequality \eqref{eqn:lowerbounds} is sharp in this case.

Similarly, we can consider other  action forms on cotangent bundles
\begin{equation}\label{eqn:minimalcoupling}
\alpha =\alpha_0 + \pi^* A
\end{equation}
where $A$ is a closed $1$-form on $Q$,
and $\pi$ is the projection in the bundle $T^*Q$,
and $\alpha_0$ is the standard action form.

The bound~\eqref{eqn:lowerbounds}
also applies to this case (which appears in the study of magnetic fields).
\end{rem}

\section{Proof  of Theorem \ref{mainpre}}
\label{sec:proofs_presymplectic}

Part \textbf{(A)}  is automatic in the pre-symplectic case.

The proof of Theorem~\ref{mainpre} is basically walking through the proofs of
Theorem~\ref{thm:main1} and Theorem~\ref{thm:main2}.
The proofs of the vanishing lemmas Lemma~\ref{lem:vanishing} and Lemma~\ref{lem:vanishingmanifolds} do not require any change since the proofs are just iterating the definition, and neither closedness
nor nondegeneracy play a role.

In the case the conformal factor is not constant  \eqref{iterates} has to be adapted to:
\begin{equation}\label{newiterates}
\begin{split}
|\omega(x) (u,v)| &\le C\eta_-^{-n} \| \omega (f^n(x)\| \|Df^n(x) u\| \|Df^n(x) v\|, \ \text{for} \ n\ge 0\\
|\omega(x) (u,v)| &\le C\eta_+^{-n} \| \omega (f^n(x)\| \|Df^n(x) u\| \|Df^n(x) v\|, \ \text{for} \ n\le 0\\
\end{split}
\end{equation}
Therefore the same proof of the vanishing lemma \ref{lem:vanishingmanifolds} works in the presymplectic case, under the conditions
$|\lambda_+\mu_+\eta_-^{-1} |< 1$ for the stable case and $|\lambda_-\mu_-\eta_+|<1$ in the unstable one.

Observe that, when the presymplectic factor $\eta$ is constant,
the rate  conditions  \eqref{eqn:rates_conditions_pre} entering in Theorem~\ref{mainpre} are
implied by the hypothesis \eqref{eqn:conformal_rates}  in Theorem~\ref{thm:main1}.

On the other hand, the pairing rules \eqref{eqn:pairing_rules} may
fail to hold for presymplectic  NHIMs,
as in Example~\ref{intermediatenew}.

To obtain the  proof of part \textbf{(B)} of Theorem~\ref{mainpre}
we can apply the proof  of part \textbf{(B)} of  Theorem \ref{thm:main2} given in section \ref{sec:main_2_B_proof_v}.
These proofs use the vanishing lemma \ref{lem:vanishingmanifolds} and the fact that the form is closed.
Hence, go through without change.
The proofs in Section~\ref{sec:proof_main_2_b_L}, \ref{sec:iteration}, which do not use any geometry,
(but use different hyperbolicity rates) do not require any adaptation.

Analogously, the proofs of part \textbf{(C)} of Theorem~\ref{mainpre} are basically the same as the ones to prove part \textbf{(C)} of  Theorem \ref{thm:main2} given in section \ref{sec:proofpartC}.

Of course, part \textbf{(D)}  on the leaf dynamics in Theorem~\ref{mainpre} does not have an analogue in Theorem~\ref{thm:main1} and in Theorem~\ref{thm:main2} but it is an easy consequence  of the conformal dynamics and the scattering map
preserving the kernel of the presymplectic form.

\appendix

\section{Summary of   the theory on  properties of  NHIMs}\label{sec:appendix_NHIM}

In this appendix, we collect, without detailed proofs but with
references,  several results on the theory of NHIMs
paying special care  to the case of unbounded manifolds
and the needed explicit uniformity assumptions.

The theory of NHIMs is very rich and there are many results
we do not use in this paper (e.g existence of locally invariant foliations,
linearization, persistence, etc.) and, hence, we do not mention
them in this appendix. We will concentrate in the
properties of (un)stable  and strong (un)stable manifolds of NHIMs
including dynamical characterizations and regularity properties.

Note that, in the unbounded case that we are considering in this paper,
it is important to make assumptions  that make it explicit that
the properties of the manifold are uniform.

For us, the most important result of  the theory of NHIM is
the characterization of invariant objects and their regularity.

We consider the setting from Section \ref{sec:standing} without assuming a symplectic structure on the manifold and that the map is conformally symplectic.

\begin{thm}
  \label{localstable}
  Consider a  manifold $M$ satisfying the assumption \eqref{U1}, and
  $f:M\to M$ a $C^r$ diffeomorphism on $M$.

  Let $\Lambda \subset M$  invariant under $f$,
  satisfy Definition~\ref{def:nhim} for rates $\lambda_\pm, \mu_\pm$, which, moreover, satisfy condition \eqref{eqn:mubounds}.

  Assume, furthermore, that   the manifold $\Lambda$ and
  the stable and unstable bundles satisfy the uniformity assumption \eqref{U2}.

  Then, there exist (rather explicit)
  $\ell, \tilde \ell, m, (\tilde m, m) $ -- called regularities
  of invariant objects --
  depending only on $r$, the regularity of the map,
  and the hyperbolicity rates $\lambda_\pm, \mu_\pm$ such that:
\begin{itemize}
\item [(I)]
$\Lambda$ is a $\C^\ell$-manifold;
\item[(II)]
There exist $0<\tilde\rho <\rho$, $\tilde{C} ,\tilde{D} >0$, and a $\C^{\tilde \ell}$-manifold  $W^{s,\textrm{loc}}_\Lambda$ in $\OO_{\tilde\rho}(\Lambda)$
described by the following equivalent conditions:
  \begin{itemize}
\item[(i)]
      $y \in  W^{s,\textrm{loc}}_\Lambda$;
\item[(ii)]
      $f^n(y) \in \OO_{\tilde\rho}(\Lambda)$,
  i.e., $d(f^n(y), \Lambda) \le \tilde\rho$  for all $n \ge 0$;
\item[(iii)]
      $f^n(y) \in \OO_{\tilde\rho}(\Lambda)$,
  i.e., $d(f^n(y), \Lambda) \le \tilde\rho$  for all $n \ge 0$;
  and
  \[
   \lim_{n \to +\infty} d( f^n(y), \Lambda)  = 0;
  \]
\item[(iv)] $d(f^n(y), \Lambda) \le \tilde C  (\lambda_+)^n$   for all $n \ge 0$;

\item[(v)]  $d(f^n(y), \Lambda) \le \tilde D \left (\mu_-\right )^{-n}$   for all $n \ge 0$;
\item[(vi)]
      There exists a unique $x \in \Lambda$ such that
      \[
      d(f^n(x), f^n(y) ) \le \tilde C  (\lambda_+)^n
      \quad \textrm{ for all }n \ge 0;
      \]

      We denote such $x = \Omega_+(y)$.

    \item[(vii)] For
       $x = \Omega_+(y)  \in \Lambda$ we have.
      \[
      d(f^n(x), f^n(y) ) \le \tilde D \left ( \mu_- \right)^{-n}
      \quad \textrm{ for all }n \ge 0; \]

      (that is, if  the orbit of $y$ converges to
      the orbit of $x\in \Lambda$ at a certain rate, it converges
      to another faster rate).

  \end{itemize}

  \item[(III)]  Given $x \in \Lambda$, we denote for $0< \rho $ sufficiently
  small:

  \[\begin{split}
  W^{s,\rm{\textrm{loc}}}_x
  =&\{y\in\OO_{\rho}\,|\, d(f^n(y), f^n(x)) \le \tilde C  (\lambda_+)^n \textrm{ for all $n \ge 0$} \}\\
  =&\{y\in\OO_{\rho}\,|\, d(f^n(y), f^n(x)) \le \tilde D  \left ( \mu_- \right)^{-n} \textrm{ for all $n \ge 0$} \}.
  \end{split}
  \]

  We refer to $W^{s, \textrm{loc}}_x$ as the \emph{local strong
  stable manifolds}.

\item[(IV)]
  \begin{itemize}
    \item[ (i)] The manifold $W^{s,\textrm{loc}}_\Lambda$ is diffeomorphic
  to a neighborhood of the zero section in $E^s_\Lambda$.

  \item[(ii)]
  Moreover,  $W^{s,\rm{\textrm{loc}} }_x $ is
a $\C^m $ manifold, diffeomorphic to a  ball in  $E^s_x$ centered at~$0$
and tangent to $E^s_x$ at $0$.

\item[(iii)]
  As a consequence of II.(iv) and II.(v),
  $\{W^{s, \textrm{loc}}_x\}_{x \in \Lambda}$ is a $C^{\tilde m, m}$ foliation  of
  a neighborhood of $\Lambda$ in $W^{s, \textrm{loc}}_\Lambda$, in the sense of
  Definition~\ref{Ctildemm}.
\end{itemize}
\end{itemize}

There are  similar characterizations of
  $W^{u,\textrm{loc}}_\Lambda$, $W^{u,\rm{\textrm{loc}}}_x$,  involving negative times, which we leave to
the reader. They can be obtained by noting that these
unstable objects are the stable objects for the  inverse map $f^{-1}$.

The regularities $\ell, \tilde \ell, m, (\tilde m, m)$ can
be made as large as desired by making assumptions on $\lambda_\pm, \mu_\pm,r$.
Hence,  the assumptions ${\bf (H1), (H2), (H3), (H4) }$ are assumptions
  on the rates.
\end{thm}

\begin{rem}
In some treatments, the statements of Theorem~\ref{localstable} are given with the rates in (II)(iv), (II)(v), (II)(vi), (II)(vii),
and (III) being $\lambda_++\eps$, $\mu_-+\eps$ rather than $\lambda_+$, $\mu_- $, respectively, where $\eps>0$ can be chosen arbitrarily, and
$\tilde\rho$ depends on $\eps$. If we used such a statement, Lemma~\ref{lem:globalrates} would imply the statement of Theorem \ref{localstable},
that is, we can get rid of the $\eps$ terms by choosing the constants $\tilde C$, $\tilde D$ a little larger.
\end{rem}

\begin{rem} \label{subtleties}
For the sake of readability, we have decided
to consider only $\C^r$ regularity for integer $r$.
In many settings, $\C^r$ is also defined for non-integer $r$
with the fractional part interpreted as H\"older regularity.
The definition of H\"older regularity on manifolds is delicate
since it involves comparing geometric objects at separate points.
Even if the notion of H\"older function is non-controversial,
the notion of H\"older distance or H\"older norm  (needed
to work out proofs) is cumbersome.
Using fractional regularities is needed to obtain sharp regularity results.

Note that the mapping $(f,g) \rightarrow f\circ g$
is not continuous from $C^0 \times C^0 \rightarrow C^0$
unless $f$ is uniformly continuous.
If one does not use fractional regularities for $f$, the only way to obtain
uniform continuity for $f$ is to assume $f \in C^1$. This may lead
to extra losses of regularity in the conclusions.  One possibility used in several references
is to include the uniform continuity of the highest derivative
in the definition of $C^r$ but note that, when the domain is unbounded,
the uniform continuity is not preserved under uniform limits, so
that the space thus defined is not a Banach space.
\end{rem}

The way that Theorem~\ref{localstable} is usually proved is by
representing the  objects of interest  using functions
and solving functional equations that express invariance.
Many of these functional equations involve the composition operator
whose properties involve subtleties related to uniformity.

Theorem~\ref{localstable}  gives several equivalent  characterizations of the
local stable manifolds of $\Lambda$. The weaker ones are boundedness, others
are just convergence and  convergence with fast rates.
The fact that these characterizations are equivalent
is rather remarkable. Describing homoclinic excursions
by intersections of manifolds, allows to take advantage of the
regularity of the manifolds and their geometric properties to
compute homoclinic excursions.

For a point $y$ to be in $W^{s,\textrm{loc}}_x$ it is important
that the convergence of its orbit to the orbit of $x$
happens with a rate $\lambda_+$ faster  than the rate of convergence
$\mu_+$ of points in $\Lambda$ to the orbits of $x$.
It is not a characterization by topological properties
such as convergence of orbits.

The fact that if we fix $y \in W^{s,\textrm{loc}}_\Lambda$, $x \in \Lambda$, is
determined uniquely, as claimed in (II.vi) of
Theorem~\ref{localstable}, is crucial for this paper, since
the projection from $y$ to $x$ -- defined by $\Omega_+$ --
is an important ingredient.
This property requires assumption {\bf (U2)}.
\cite[Example 2.9]{Eldering12}
(See Figure~\ref{fig:manifold_no_unif_tubular})
provides an example of
NHIM for which  $W^s_{x} =  W^s_{x'}$ for
some $x\ne x'$.  Therefore,  the decomposition of $W^s_\Lambda$
into strong stable leaves $W^{s,loc}_x$  in \eqref{decomposition_strong}
is not a disjoint union  \footnote{Since the relation given by
$y \approx  \tilde y \iff d(f^n(y), f^n(\tilde y))
\le C_{y,\tilde y}\lambda_+^n, n \ge 0$ is an equivalence relation, the equivalence classes under
this relation
are leaves of a local foliation. Nevertheless, in  \cite[Example 2.9]{Eldering12}
a leaf can contain two
points in $\Lambda$, so that the leaves cannot
be labelled by their point of intersection
with $\Lambda$. For our purposes, this prevents us from
defining   $\Omega_+$. }.

To understand the phenomenon of non-unique projection $\Omega_+$
mentioned above, the following remark may be useful.

\begin{rem}
In the study of NHIM, it is natural to consider two distances on $\Lambda$:
$d_\Lambda(x,x')$ is the shortest length of paths in $\Lambda$ joining $x,x'$,
and $d_M(x,x') $ is the shortest length of paths in $M$.
Clearly, $d_M(x,x') \le d_\Lambda(x,x')$.
The distance that enters in the definition of $W^s_x$ is $d_M$.
It is easy to see that for $\eps > 0 $, $x,x'\in\Lambda$,
$d_\Lambda( f^n(x), f^n(x') ) \le C (\lambda_+ -\eps)^n, n >0$
implies $x = x'$.
On the other, hand, if $\Lambda$ folds into itself (as in \cite[Example 2.9]{Eldering12}) it
is possible that  for two different  points $x,x' \in \Lambda$,
$d_M( f^n(x), f^n(x')) \le C \lambda_+^n $.
\end{rem}

\begin{rem}
Note that the characterization $(IV)$ of local
stable manifolds and locally strong stable manifolds
involves the choice of a sufficiently small $\rho$.
Clearly, using the characterization by rates,
if we choose $\rho_1 < \rho_2$,
the set of manifolds corresponding to $\rho_1$ will
be contained in those corresponding to $\rho_2$.
\end{rem}

\begin{rem}
Note that the regularity of the manifolds $\Lambda$, $W^{u,\mathrm{loc}}_\Lambda$, $W^{s,\mathrm{loc}}_\Lambda$, is limited not just   by the regularity of the map $f$.
Next Example~\ref{regularitylimited} shows that the regularity of these objects also depends on relations between the rates.
It  shows that, even for analytic (indeed polynomial)
maps, the NHIM could be only finitely differentiable.
The differentiability is an expression in terms of the
hyperbolic rates.
\end{rem}

\begin{ex}\label{regularitylimited}
  Consider the map $f: \torus^d \times \real^2 \rightarrow  \torus^d \times \real^2$
  given by
  \begin{equation} \label{limited_map}
  f(\theta, s, u ) =
  (A \theta, \lambda_+s + a_s(\theta), (1 /\lambda_-)  u  + a_u(\theta))
  \end{equation}
  where
  $a_s, a_u: \torus^d \rightarrow \real$ are continuous functions
  (the concrete example is a trigonometric polynomial),
  $A \in SL(d, \Z)$ has spectrum contained in $1/\mu_-, \mu_+$.

  To make the example easier to analyze, we will assume that
  the leading modulus eigenvalues are simple and irrational.
  Hence,  $\lambda_\pm, \mu_\pm$ are real  numbers
    as in Figure~\ref{fig:rates}
  \[
  \lambda_+ < 1/\mu_- <  1 < \mu_+ <  1/\lambda_- .
  \]

  It is standard  that $A$ can be interpreted either as a diffeomorphism
  of $\torus^d$ (taking the action mod $1$)
  or as a linear  map  on $\integer^d
  $ (acting on the Fourier
  coefficients) \cite{KatokH95}.

  We search for an invariant set  of \eqref{limited_map}
of   the form or a graph
  \[
  \Lambda = \{ (\theta, b_s (\theta), b_u (\theta)) \,  | \theta \in \torus^d\} .
  \]

  The image of a point on $\Lambda$  under $f$ is
  \[
  ( A(\theta), \lambda_+  b_s (\theta) + a_s(\theta),
  (1/\lambda_-)   b_u(\theta) + a_u(\theta) ) .
  \]

  This image is in $\Lambda$  if
  and only if
  \[
  \begin{split}
    &  b_s(A\theta) =   \lambda_+  b_s (\theta) + a_s(\theta) ,\\
    &  b_u(A\theta) =   (1/\lambda_-)  b_u (\theta) + a_u(\theta) . \\
  \end{split}
  \]
  The above equations for $b_s, b_u$ can be rearranged as
  \begin{equation} \label{fixed_point}
  \begin{split}
    &  b_s(\theta) =   \lambda_+  b_s (A^{-1} \theta) + a_s(A^{-1} \theta) , \\
    &  b_u (\theta)  =  \lambda_- b_u(A\theta)  - \lambda_- a_u(\theta) . \\
  \end{split}
  \end{equation}

  The equations \eqref{fixed_point} can be thought of
  as fixed point equations for an operator given
  by the right hand side.
  They admit an unique  bounded
  solution obtained by iteration of the operator given
  by the right hand side.

 Analyzing these solutions reveals that, in many cases, they possess only a finite number of continuous derivatives.
  We will see that the analysis is very similar
  to the analysis of the classical Weierstrass example
  \footnote{
  \cite{Hardy16} considers the harder problem of no
  derivative of Weierstrass function at any point. }.

  A version of the argument
  close to the one here appears in \cite[Section 6.2]{Llave92}.
Similar arguments appears often in
  hyperbolic systems.

  We give the explicit formulas  only for the stable case. The
  unstable one is very similar.

  We observe that if \eqref{fixed_point} is to
  hold, substituting the right-hand side of \eqref{fixed_point}
  repeatedly, we obtain that for any finite $N$ we have:

  \[
  \begin{split}
  b_s(\theta) &= a_s(A^{-1} \theta)  + \lambda_+ a_s(A^{-2} \theta)
    + \lambda_+^2 a_s(A^{-3} \theta) + \cdots
    + \lambda_+^N a_s( A^{-(N+1)} \theta) \\
    &+ \lambda_+^{N+1} b_s(A^{-(N+1)}\theta)
    \end{split}
    \]

    If $b_s$ were bounded (or in any $L^p(\torus^d, \R)$ )
    the last term in the above formula would tend to zero.
    Hence the only possible bounded $b_s$  solving  \eqref{fixed_point} is
   \begin{equation}\label{solution}
  b_s(\theta)  = \sum_{j = 0}^\infty  (\lambda_+)^j  a_s( A^{-j -1} \theta )
   \end{equation}

   The series in \eqref{solution} is
   uniformly convergent because the general term is
   bounded by a geometric series.
   \[
   \| (\lambda_+)^j  a_s \circ  A^{-j -1}  \|_{\C^0}
   \le   (\lambda_+)^j  \| a_s\|_{\C^0}
   \]
   Hence, \eqref{solution} defines a continuous function on the torus
   and we can compute its Fourier coefficients term by term.

   Now we analyze \eqref{solution} to show that it cannot be very differentiable.

The chain rule gives
$ D^m  (a_s \circ A^{-j-1})  =(D^m  a_s ) \circ  A^{-j -1}  (A^{-j -1} )^{\otimes m }$ and, if $Av = \mu v$, we have:
\[
( v \cdot \partial)^m (a_s \circ A^{-j-1} )
=(  ( v \cdot \partial)^m  a_s ) \circ  A^{-j -1} \mu^m|v|^m.
\]
So, for high enough $m$,
the derivation term by term of
\eqref{solution} becomes problematic in general.

To show that indeed
the sum \eqref{solution} has only a limited number of derivatives,
we recall
that  for a  $\C^\ell $ function $b$ we have
$ | \hat b_k| \le C |k|^{- \ell}$, where $\hat b_k $ are its Fourier coefficients.

If we take $a_s$ in \eqref{solution} to be
$a_s(\theta)= \cos( 2 \pi k_0\cdot \theta)$, then
$a_s(A^{j-1}\theta)= \cos( 2 \pi A^{j-1}k_0\cdot \theta)$

We have that if $\tilde k  = (  A^{-(j+1)})^T k_0  $,
there is only one term in \eqref{solution}
with $\tilde k $ index.

Therefore, $ |\widehat{ (b_s)}_{\tilde k }| = \frac{1}{2} \lambda_+^j  $.
Since $|\tilde k |  \le  C  \mu_-^{j+1}$,
we see that it is impossible to have an inequality
of the form $| \widehat{ (b_s)}_k | \le C |k|^{-\ell} $ if
$\ell  $ is large enough that $\lambda_+ \mu_-^\ell  > 1$
and therefore, the  $b_s$ corresponding to
$a_s$ as before
is not $\C^{\ell}$.

A similar restriction happens for the unstable part.
\end{ex}

\begin{rem}
  The Example~\ref{regularitylimited}
  gives an idea of what are the optimal regularities.

  The final optimal regularities depend however
  upon subtleties such as those in Remark~\ref{subtleties}.

In the compact case, this  example gives the limit of the regularity
that can be obtained using the rates as input.
There are proofs that reach this limit.

The example  can be modified to yield restrictions on regularity even if the map $f$ is
furthermore assumed to be symplectic (take $d$ even,   $A$  symplectic  on $\torus^d$
and $\lambda_+/\lambda_-  = 1$).

The argument above can also be used for fractional regularities
and also to conclude  that this lack of differentiability
is generic (indeed, when $\lambda_+ \mu_-^\ell > 1$,
 $b_s$ is not  $\C^\ell$ for all trigonometric polynomials
except in a linear space of infinite codimension).

One interesting problem is to study deeper geometric properties
(fractal dimension, directional derivatives of these examples).

\end{rem}

\medskip

Once  we have the local stable and unstable manifolds,
we define the global stable and unstable manifolds.
\[
\begin{split}
  &W^s_\Lambda = \bigcup_{n \ge 0}  f^{-n}(W^{s,\textrm{loc}}_\Lambda);  \qquad
  W^u_\Lambda = \bigcup_{n \ge 0}  f^{n}(W^{u,\textrm{loc}}_\Lambda);  \qquad \\
  &W^s_x  = \bigcup_{n \ge 0}  f^{-n}(W^{s,\rm{\textrm{loc}}}_{f^n(x)});  \qquad
  W^u_x  = \bigcup_{n \ge 0}  f^{n}(W^{u,\rm{\textrm{loc}}}_{f^{-n}(x)}).
\end{split}
\]

This follows from the observation
that for any point $y \in W^s_x$, there exist $N$ such
that
$f^N(y) \in W^{s, \textrm{loc}}_{f^N (x)}$.
As a consequence, we obtain the point II.(v) in
Theorem~\ref{localstable}, giving a characterization
of the global stable manifold by rates of convergence.

Since $W^{s, \textrm{loc}}_{f^N(x)}$
is simply connected, so is $f^{-N}(W^{s, \textrm{loc}}_{f^N(x)} )$
and all of these sets overlap in an open set.

\[
\begin{split}
  &W^s_\Lambda = \bigcup_{x\in\Lambda}   W^{s}_x,  \qquad
  W^u_\Lambda = \bigcup_{x\in\Lambda}   W^{u}_x, \qquad \\
  &W^s_x\cap W^s_{x'}=\emptyset ,  \qquad
  W^u_x\cap W^u_{x'}=\emptyset , \textrm{ for } x\neq x'.
\end{split}
\]

The above can be described as saying that the
decomposition above is a foliation of $W^s_\Lambda$
with leaves $W^s_x$. Note that the leaves are in general
more regular than $W^s_\Lambda$.
The regularity of $W^s_\Lambda$ is determined both
by the regularity of the leaves and the way that they fit together.

There are many proofs of Theorem~\ref{localstable}, some of them
have more assumptions. For the purposes of this paper,
an efficient proof of the local stable/unstable manifolds
is in \cite[p. 33--38]{Pesin}. This proof is based on
constructing first the embeddings giving the local stable/unstable manifolds
by studying the functional equations satisfied via a fixed
point argument.
Note that this proof requires uniformity assumptions
for the map or the derivative in a neighborhood of $\Lambda$.

\begin{rem}
Large parts of Theorem~\ref{localstable} require only \eqref{U1}.

Nevertheless, \cite[Example~3.8]{Eldering12}
(Illustrated here in Figure~\ref{fig:manifold_no_unif_tubular} )
  shows  that to get
  $(II).(vi)$,$(II).(vii)$, $(III)$ and all
  the other properties of $W^{s,\textrm{loc}}_x$ in Theorem~\ref{localstable}
  one needs
    the uniformity assumption \eqref{U2}. For the purposes of this paper,
    the fact that $W^s_\Lambda$ is foliated by $W^s_x$ is crucial
    since it is what we use  to define $\Omega_+$ and the scattering map.

\end{rem}

\begin{rem}\label{alternative_regularity}
  The theory of regularity based on hyperbolicity rates
  as in Theorem~\ref{localstable} is
  not the only possible way to establish regularity of
  invariant objects.

  The paper \cite{Fenichel74}  establishes regularity
  based on the study of the numbers $\alpha$ for which
  \begin{equation} \label{fenichelconditions}
    \sup_x \|Df(x)_{\mid E^s}\| \|Df^{-1}(x)_{\mid T_x \Lambda}\|^\alpha  < 1.
    \end{equation}
  which is sharper than taking the supremum on all factors.
  as is done in \cite{HPS}.

  Conditions similar to \eqref{fenichelconditions}
  are used in the invariant cone approach to regularity
  of invariant manifolds \cite{CapinskiZ15}.

  For  the purposes of this paper, the use of rates is more natural
  because rates enter in the formulation of vanishing lemmas.
  We have, however, formulated the standing assumptions
  {\bf (H1)-(H4)} and their variants as regularity assumptions
  that can be verified in concrete examples either using rates
  or \eqref{fenichelconditions}  or any other method.
  \end{rem}

Boundedness of symplectic form indeed plays a role in the main results of
this paper, starting with the vanishing lemmas (Section~\ref{sec:vanishing_lemmas}).
When the symplectic form is unbounded some of the main results do not hold (see Example~\ref{ex:unbounded_symplectic_form}).
On the other hand,
in Section~\ref{sec:unbounded_forms} we give some results
that show that the boundedness assumption on the symplectic form
can be weakened at the price of other assumptions on rates. This seems
like a possible line of research that we hope to
pursue in the future.

\section{Some basic results on hyperbolicity rates}\label{sec:NHIM_results}
In this appendix, we collect and prove  some basic results on
hyperbolicity rates.

\subsection{Relations between forward and backwards uniform rates}

\begin{lem}\label{lem:rates_inverses}
Let $E$ be an invariant vector bundle over some space $X$. The statement

\[
\forall\, x\in X,\, v\in E_x,\,\|Df^n(x)v\|\leq C\lambda^n \|v\|\textrm{ for all $n\geq 0$}
\]
is equivalent to the statement
\[
\forall\, y\in X,\,  w\in E_y,\, \|Df^{-n}(y)w\|\ge \tilde{C}\lambda^{-n} \|w\|\textrm{ for all $n\geq 0$}.
\]
\end{lem}
\begin{proof}
Given $y\in X$ and $w\in E_y$, by setting $y=f^n(x)$, $w=Df^n(x)v$, we have,
\[
Df^n(x)v=w \Leftrightarrow Df^{-n}(y)w =v
\]
and therefore
\[
\|w\|=\|Df^n(x)v\|\leq C\lambda^n \|v\|= C\lambda^n \|Df^{-n}(y)w\|
\]
which gives the result with $\tilde C=\frac{1}{C}$.
\end{proof}

Lemma \ref{lem:rates_inverses} illustrates one of the advantages of the convention of denoting the rates as in \eqref{eq:rates0} and \eqref{eqn:NHIM}, namely,   if we change $f$ to $f^{-1}$, it suffices to change $\mu_+$ to $\mu_-$ and $\lambda_+$ to $\lambda_-$.

\begin{lem}\label{rem:ratesplusminus}
For the rates \eqref{eq:rates0}  satisfying \eqref{eqn:NHIM} we always have:
\begin{equation}\label{eq:rem_ratesplusminus}
  \lambda_+ <\frac{1}{\lambda_-} \textrm { and }
\mu_+ \ge  \frac{1}{\mu_-}.
\end{equation}
\end{lem}
\begin{proof}
The reason for the second inequality is that if
for   $n\geq 0$ we have  for all $x\in \Lambda, v \in T_x\Lambda$:
\[
\| Df^n(x)(v)\|  \leq D_+\mu_+^n\|v\|
\]
then, by Lemma \ref{lem:rates_inverses}, we also have for  $y\in \Lambda, w \in T_y\Lambda$
\[
\| Df^{-n}(y)(w)\| \geq D_+^{-1}\mu_+^{-n}\|w\|.
\]
Thus, using \eqref{eqn:NHIM}, $\mu_+^{-1}\le \mu_- $.

The inequality $\lambda_+ < \lambda_- ^{-1}$ follows from the previous result about $\mu_+$ and $\mu_-$ and  the normal hyperbolicity
assumptions \eqref{eq:rates0}.
\end{proof}

With the above considerations,  we can also write as
a characterization of the tangent space to the NHIM.
\[
  v \in T_x \Lambda \iff \tilde D_-  \mu_-^{-n} \|v\|
  \le \| Df^n(x)v\| \le \tilde D_+  \mu_+ ^n \|v\|, \quad \forall n \ge 0.
  \]

  \subsection{Hyperbolicity rates in the stable manifold}
  \label{sec:globalrates}

  In this section we study the hyperbolicity rates for
  tangent vectors to the stable manifold. The intuition is that
  for orbits that converge to the manifold $\Lambda$,
  most of the factors are derivatives near the manifold.
  Of course, the convergence may be slow, therefore, one needs
  some precisions on the statements.

\begin{lem}\label{lem:globalrates}
Choosing   constants $C_+$, $D_+>0$
larger than  those in the manifold $\Lambda$,  for all $y \in W^{s,\mathrm{loc}}_{x}\cap\OO_\rho$ (resp., $y\in \cap W^{s,\mathrm{loc}}_{\Lambda}\cap\OO_\rho$) we have:
\begin{equation}\label{eqn:ratesinmanifolds}
\begin{split}
 v\in T_yW^{s,\mathrm{loc}}_{x} \Leftrightarrow&\| Df^n(y)(v)\|\leq  C_+\lambda_+^n\|v\| \textrm{ for all } n\geq N,\\
v\in T_yW^{s,\mathrm{loc}}_{\Lambda}\Leftrightarrow&\| Df^n(y)(v)\|\leq  D_+\mu_+^n\|v\| \textrm{ for all } n\geq N.
\end{split}
\end{equation}
An analogous property holds for points in $W^{u,\mathrm{loc}}_\Lambda\cap \OO_\rho$.
\end{lem}

\begin{rem}\label{rem:globalrates}
Recalling that if
$y\in W^{s,\mathrm{loc}}_\Lambda$, there exists $N=N(y) > 0 $ such
that $f^N(y) \in \OO_{\bar \rho}$, we have
that \eqref{eqn:ratesinmanifolds} holds for  any
$y\in W^{s,\mathrm{loc}}_\Lambda$ but with constants $C_\pm, D_\pm$
which depend on $y$.
\end{rem}

\begin{rem} Versions of Lemma \ref{lem:globalrates} appear in the classical references
with the rates of convergence are
slightly worse that those in the NHIM.
These results imply that the stable/unstable manifolds are NHIMs themselves.

Lemma \ref{lem:globalrates} is an improvement from previous results in the literature,
because  the rates claimed in the stable manifold are  exactly
the same as the rates for the linearization in the NHIM.
For most of the results in this paper, the
classical results with slightly worse rates are enough,
so the proof can be just skimmed.
\end{rem}

Lemma~\ref{lem:globalrates} will be a consequence  of the following preliminary result (Proposition~\ref{provisional}),  the chain rule,
\begin{equation}\label{chain}
Df^n(y) = Df(f^{n-1}(y)) Df(f^{n-2}(y))\cdots Df(y),
\end{equation}
and Proposition~\ref{ratesmatrices}.  We postpone the details of the proof of
Lemma \ref{lem:globalrates} after these two  propositions.

The first preliminary result is fairly standard and is indeed enough for many applications.
For us, it will be the first step and it will be later bootstrapped
to get Lemma \ref{lem:globalrates}.

\begin{prop} \label{provisional}
 In the conditions of Lemma~\ref{lem:globalrates}, given $\eps  >0$ we can find a radius
 $0<\bar \rho =\bar \rho (\eps) <\rho$ and a constant $\tilde{C}_{+}=\tilde{C}_{+}(\eps)$
 such that if we take  any
 $y \in W^{s,\mathrm{loc}}_\Lambda$,
 $d(y, \Lambda) \le \bar \rho$, and take $x=\Omega^+(y)$, so that $y \in W^{s,\mathrm{loc}}_x$,
 we have for all $n >0 $, and for all
$v\in T_{y} W_{x}^{s,\mathrm{loc}}$
\begin{itemize}
\item[(i)]
\[
\|Df^n(y) v\|  \le \tilde{C}_{+} (\lambda_+ +\eps)^n  \|v\|
\]
\item[(ii)]
As a consequence,
\[
    d(f^n(x), f^n(y)) \le \rho \tilde{C}_{+} (\lambda + \eps)^n .
\]
\end{itemize}
\end{prop}

\begin{proof}
Because of the uniformity of the bounds assumed in Definition~\ref{def:nhim} we have that, for all $x\in \Lambda$ we have, for any $n\ge 0$:
\[
\|Df^n (x)|_{E^s_{f^n(x)}} \| \le C_+ \lambda_+^n
\]
Of course, $E^s_{f^n(x)} = T_{f^n(x)} W^s_{f^n(x)}$ and this will be useful later.

Once we choose $\eps > 0 $, we can find $N= N(\eps) > 0 $ so
that,
\[
 C_+ \lambda_+^N \le \frac{1}{10} (\lambda_+ + \eps)^N
\]
Therefore, for any $x\in \Lambda$ we have:

\[
\| Df^N (x)|_{E^s_x} \| \le \frac{1}{10} (\lambda_+ + \eps)^N .
\]
Because of the chain rule \eqref{chain} and the uniformity
of the bounds on the derivatives and the uniform continuity of
the derivatives, we can find $\bar \rho =\bar \rho (\eps) > 0$ so
that for all $y \in W^s_x$, $d(y, x) \le \rho$ we have
\[
\| Df^N (y)|_{T_y W^s_x } \| \le  (\lambda_+ + \eps)^N .
\]
Using \eqref{chain} and that all the derivatives of $f$ are uniformly bounded and uniformly continuous, we have
that there exists a (big enough) constant $ \tilde C$ such that
\[
\| Df^j (y)|_{T_y W^s_x } \| \le \tilde C_{+} (\lambda_+ + \eps)^j \quad \textrm { for }0< j < N
\]
Using that any positive number $n$ can be written $n = k N +j $ with $0 < j \le  N$ we have, using again \eqref{chain},
\[
\| Df^n (y)|_{T_y W^s_x } \| \le (\lambda_+ + \eps)^{kN} \tilde C_{+} (\lambda_+ + \eps)^j  =\tilde C_+ (\lambda_+ +\eps)^n
\]
    From this, using that
    $d(f^n(x), f^n(y)) \le \int_0^1  \|Df^n(\gamma(s)) \| \gamma'(s)| \, d s$
    we obtain  the last consequence.
   \end{proof}

\subsubsection{A sharp result on perturbations of products
  of a sequence of operators (cocycles)}
\label{sec:sharpperturbation}

The following result assumes rates of growth for a sequence
of successive products  of a sequence of
operators (sometimes called \emph{cocycles}).
and the factors are changed by a summable sequence, then, the
new sequence grows at the same rate.

The interesting thing for us is that it shows that the
rates of growth of vectors do not need to be modified at all.

\begin{prop}\label{ratesmatrices}
  Given a sequence of Banach spaces $\{ X_j\}_{j \in \mathbb{N}}$

  Let $\{\alpha_j\}_{j \in \mathbb{N}}$ and
   $\{\beta_j\}_{j \in \mathbb{N}}$
  be two sequences of operators
  $\alpha_j, \beta_j: X_j \rightarrow X_{j+1}$
  Denote by $A_{n,m}, B_{n,m}$ the associated \emph{cocycles}
  \begin{equation}\label{cocycle}
    \begin{split}
      &A_{n,m} = \alpha_{n-1} \alpha_{n-2} \cdots \alpha_m , \\
      &B_{n,m} = \beta_{n-1} \beta_{n-2} \cdots \beta_m , \\
      \end{split}
  \end{equation}
  Therefore, we have for all $n -1\ge m \ge l$, $A_{n,m} A_{m,l} = A_{n,l}$,
  $B_{n,m} B_{m,l} = B_{n,l}$.

  Assume that
  \begin{itemize}
  \item[(i)] $ \exists C, \lambda  > 0 $ s.t. $\forall n> m \in \mathbb{N}$
    \[
    |A_{n,m} |  \le C \lambda^{(n-m)},
    \]
  \item[(ii)]
    \[
    \sum_{j=0}^{\infty} | \beta_j - \alpha_j | < \infty .
    \]
  \end{itemize}

  Then, there exists $ \tilde C$ such that  for all $n -1\ge m$
    \[
    |B_{n,m} |  \le \tilde C \lambda^{(n-m)}
    \]
  \end{prop}

\begin{proof}
  If we multiply $\alpha$, $\beta$ by a constant $\sigma$, then both
  $A_{n,m},  B_{n,m}$ multiply by $\sigma^{(n-m)}$,  so without
  loss of generality, we can assume that $\lambda <1 $ in the
  hypothesis of Lemma~\ref{ratesmatrices}.

  We fix $n> m$ and
  note that we have adding and subtracting and  grouping the terms
  with different numbers of factors $\beta -\alpha$
  \[
  \begin{split}
    B_{n,m}= &A_{n,m} +
     \sum_{n-1\ge j_1\ge m}  A_{n, j_1+1} (\beta_{j_1} -  \alpha_{j_1}) A_{j_1, m}
     \\
     &+ \sum_{n-1\ge j_1> j_2\ge m} A_{n, j_1+1} (\beta_{j_1} -  \alpha_{j_1})
     A_{j_1, j_2+1}
  (\beta_{j_2} - \alpha_{j_2} ) A_{j_2, m} \\
  &+ \cdots \\
 & + (\beta_{n-1} - \alpha_{n-1}) \cdots (\beta_{m} - \alpha_{m}) .
  \end{split}
  \]

  If we write
  \[
  B_{n,m} = (\alpha_n + (\beta_n - \alpha_n) ) \cdots (\alpha_m + (\beta_m - \alpha_m) )
  \]
  and expand the product and group by the number of factors
  $\beta -\alpha$.

  The general term in the sum above is obtained by replacing $k$
  factors $\alpha$ in
  the expression for $A_{n,m}$
  with $\beta_j - \alpha_j$ and leaving all the others as $\alpha_j$.
  The products of consecutive factors $\alpha_j$ are
  transformed into the $A$ cocycles.

  Using the assumptions, the first term
  in the sum above is bounded by
  \[
  \| A_{n,m} \| \le C\lambda^{(n-m)} .
  \]
  The second term is bounded
  \[
  \begin{split}
\|  \sum_{n-1\ge j_1\ge m}  A_{n, {j_1}+1} (\beta_{j_1}-  \alpha_j) A_{{j_1}, m} \|
&\le C^2 \sum_{n-1\ge j_1\ge m}\lambda^{n-{j_1}-1} \| \beta_{j_1} -  \alpha_{j_1} \|  \lambda^{{j_1}-m} \\
&\le C^2 \lambda^{-1}  \lambda^{(n-m)} \sum_{j=0}^{\infty} \| \beta_{j} - \alpha_{j}\|
  \end{split}
  \]
  Similarly, we bound the next term
  as:
  \[
  \begin{split}
    &\| \sum_{n-1\ge j_1> j_2\ge m} A_{n, j_1+1} (\beta_{j_1} -  \alpha_{j_1}) A_{j_1, j_2+1}
(\beta_{j_2} - \alpha_{j_2} ) A_{j_2, m} \|\\&\quad \le
C^3  \sum_{n-1\ge j_1> j_2\ge m} \lambda^{n-j_1-1} \| \beta_{j_1} -  \alpha_{j_1}\|  \lambda^{j_1-j_2-1}
 \| \beta_{j_2} -  \alpha_{j_2}\|\lambda^{j_2- m }
 \\
&\quad \le C^3 \lambda^{-2}  \lambda^{n-m} \frac{1}{2}
 \left( \sum_{j=0}^{\infty} \| \beta_j - \alpha_j\| \right)^2 .
  \end{split}
  \]

  Note that both bounds include factors $\lambda^{n -m}$
  (remember we are assuming $\lambda <  1$).
  Similarly the general term  consists in product of cocycles of
  $(k+1)$ intervals that cover  $[m,n-1]$ with $k$ factors that have
  been changed into $\beta_j - \alpha_j$.

  The last term is  bounded (very wastefully)  as
  \[
  C^{n-m+1} \lambda^{-(n-m)} \lambda^{(n-m)} \frac{1}{(n-m)!}
  \left( \sum_{j=0}^{\infty} \| \beta_j - \alpha_j\| \right)^{n-m}.
  \]

  Therefore, adding all the bounds we get:
  \[
  \lambda^{n-m} C \sum_{k = 0}^{n-m} C^k \frac{1}{k!} \lambda^{-k}
  \left( \sum_{j=0}^{\infty} \| \beta_j - \alpha_j\| \right)^{k}
  \le \lambda^{n-m} C \exp\left( C \lambda^{-1} \sum_{j=0}^{\infty} \| \beta_j - \alpha_j\| \right).
  \]
 We conclude:
 \[B_{n,m}\le  \lambda^{n-m} C\left(1+\exp\left( C \lambda^{-1} \sum_{j=0}^{\infty} \| \beta_j - \alpha_j\|\right )\right):=\lambda^{n-m}\tilde{C}.\]

\end{proof}

\begin{proof}[Proof of Lemma~\ref{lem:globalrates}]

  If the manifold $M$ is Euclidian and the bundle $E^s$ is trivial,
  we can identify all the tangent spaces with an Euclidean
  space.

  In such a geometrically trivial case, to prove the first inequality in
  Lemma~\ref{lem:globalrates} it suffices to take:
  \[
  \begin{split}
    &\alpha_j = Df( f^j(x)): E^s_{f^j(x)} \rightarrow E^s_{f^{j+1}(x)} ,\\
      &\beta_j = Df( f^j(y)): T_{f^j(y)}(f^j(W^s_x) )
      \rightarrow T_{f^j(y)}(f^j(W^s_x)) ,
      \end{split}
        \]
but, in the geometrically trivial case, we identify $E^s_{f^j(x)}$
with  $T_{f^j(y)}(f^j(W^s_x)$.

Similarly, to prove the second inequality
Lemma~\ref{lem:globalrates} it suffices to take:
   \[
  \begin{split}
    &\alpha_j = Df( f^j(x)): E^s_{f^j(x)} \oplus T_{f^j(x)}\Lambda
    \rightarrow E^s_{f^{j+1}(x)} \oplus T_{f^{j+1}(x)}\Lambda , \\
      &\beta_j = Df( f^j(y)): T_{f^j(y)}(W^s_\Lambda)
      \rightarrow T_{f^{j+1}(y)}(f^{j+1}(W^s_\Lambda)) .
      \end{split}
        \]

  In the case that $M$ is a manifold or $E^s$ is a non-trivial bundle,
use a system of coordinates
on $\OO_\rho$
assumed to exist in {\bf (U2)}.
In such a geometric adaptation, we  need to include explicitly
  the connectors (see \eqref{eqn:connectors}) identifying the neighboring spaces and check
  that all remains uniform in the number of iterates.
  Even if it is mostly routine, we include the details.

  To prove the first inequality in Lemma~\ref{lem:globalrates}
  we observe that, geometrically:
  \[
  \begin{split}
   &Df (f^n(y)): T_{f^n(y)}f^n(W^s_x)
\rightarrow T_{f^{n+1}(y)}f^{n+1}(W^s_x), \\
    &Df (f^n(x)): T_{f^n(x)}f^n(W^s_x) \equiv E^s_{f^n(x)}
    \rightarrow T_{f^{n+1}(x)}f^{n+1}(W^s_x) \equiv E^s_{f^n(x)}.
  \end{split}
  \]
  The operators $Df (f^n(y))$, $ Df (f^n(x))$ act in different spaces
  and we have to identify them.

  Using the system of coordinates in $\OO_\rho$ we can use connectors (see \eqref{eqn:connectors})
  to identify $E^s_x$ with  $T_y W^s_x$.
  A geometrically natural way to identify these spaces is
  to take $S_x^y = (D \exp_x)( \exp^{-1}_x(y))$.
  We note that the norm of these operators is bounded
  and  in $\OO_\rho$ and the norm becomes close to $1$
  if $\rho$ is small.

  We apply the Proposition~\ref{ratesmatrices}
taking $X_j = E^s_{f^n(x)}$,
\[\begin{split}
\alpha_j =& Df( f^n(x) ),  \\
\beta_j =&
\left( S_{f^{j+1}(x)}^{f^{j+1}(y)} ) \right)^{-1}
 Df( f^j(x)) S_{f^j(x)}^{f^j(y)}.
\end{split}
\]

  Note that
  \[
  B_{n,m} = \left( S_{f^{n+1}(x)}^{f^{n+1}(y)} ) \right)^{-1}
    Df^{n-m}(f^m(x)) S_{f^j(x)}^{f^j(y)},
    \]
    so that the bounds on the cocycle are
    equivalent to the desired bounds on the derivatives.
\end{proof}

\section{Rates of convergence of homoclinic channels to a NHIM}
\label{sec:fiber_contraction}
The goal of this section is to study quantitatively the convergence of
the iterates of
channel $\Gamma$ to the invariant manifold $\Lambda$.
The explicit values of the rates of convergence enter into several proofs.
For example Sections~\ref{sec:proof_main_2_b_L}, \ref{sec:iteration}
as well in other future work.

More precisely, we prove:

\begin{lem}\label{convergence_rates}
  With the notations in the previous sections,
  assume, for simplicity  of statements and without loss of generality,
  that
  \begin{equation}\label{simplicity}
    \mu_+, \mu_- > 1.
  \end{equation}
\begin{itemize}
\item[(i)]
Assume that $ r\ge 2$ and  the foliation of $W^{s,u, \textrm{loc}}_\Lambda$ by $W^{s, u,\textrm{loc}}_x$ is of class  $\C^{1, 1}$.
  \begin{itemize}
  \item[(i.a)]
  If  $\lambda_+\mu_-<1$ then:

 \begin{equation}\label{Cr-ratess}
   d_{\C^1} (f^n(\Gamma), \Lambda) \le
      C  (\lambda_+  \mu_-)^n, \quad n > 0 .
  \end{equation}
\item [(i.b)]
If  $\lambda_-\mu_+<1$ then:
  \begin{equation}\label{Cr-ratesu}
   d_{\C^1} (f^n(\Gamma), \Lambda) \le    C  (\lambda_- \mu_+)^{|n|}, \quad n < 0 .
  \end{equation}
 \end{itemize}
 \item[(ii)]
 Assume moreover that $ r\ge j+1$ and the foliation of
 $W^{s,u, \textrm{loc}}_\Lambda$ by $W^{s, u,\textrm{loc}}_x$ is of class  $\C^{j, 1}$.
  \begin{itemize}
  \item[(ii.a)]
  If $\lambda_+ \mu_-^j<1$ then:
 \begin{equation}\label{Cr-ratessj}
   d_{\C^j} (f^n(\Gamma), \Lambda) \le     C  (\lambda_+ \mu_-^j)^n, \quad n > 0 .
  \end{equation}
\item[(ii.b)]
 If $\lambda_- \mu_+^j<1$ then:
  \begin{equation}\label{Cr-ratesuj}
   d_{\C^j} (f^n(\Gamma), \Lambda) \le      C  (\lambda_- \mu_+^j)^{|n|}, \quad n < 0 .
  \end{equation}
 \end{itemize}
\end{itemize}

\end{lem}

\begin{rem}
The Lemma~\ref{convergence_rates}
is closely related to results in the literature such as the
\emph{graph transform}, the \emph{inclination lemma} -- a.k.a. \emph{$\lambda$-lemma},
and even closer, to \emph{fiber contraction lemma}, or the  more general
\emph{cone conditions}, which have many variants. In our case, the result is easier
since we already know the fixed point of the contraction.
Unfortunately, many of the versions in the literature are only qualitative or involve
extra assumptions (e.g. \cite[p. 35]{HPS} assumes compactness).
\end{rem}

The tangent functor trick \cite{AbrahamR67, HirschP70} -- which we explain later --
shows that the case $j = 1$ implies the results for other $j$ in
Lemma~\ref{convergence_rates}. The case $j=1$ we present,  basically goes back to
\cite{Hadamard98} (translated to English  in \cite{Hasselblatt17}).

\subsection{Proof of Lemma~\ref{convergence_rates} }

  We will consider here  only the case of $n \rightarrow \infty$.
  The case $n \rightarrow -\infty$  is identical, up to a change in typography.

By Theorem \ref{localstable}, item (II) (iv), by choosing a suitable constant  $C$ we have that

  \[
  d_{\C^0}(f^n(\Gamma), \Lambda) \le C\lambda_+^n \textrm{ for } n>0 .
  \]

Hence, it suffices to prove the estimates under the extra assumption
that $\Gamma$ and its iterates remain in small  neighborhood of $\Lambda$.
Concretely, we will assume that
\[
\Gamma \subset \OO_\rho,
\]
where $\OO_\rho$ is the neighborhood of $\Lambda$ given in assumption {\textbf U2}.
Furthermore, to estimate derivatives, we can work in   arbitrary small patches.

In $\OO_\rho$, and in small enough patch, we can take the system of coordinates of section  \ref{rem:new coordinates} given in  \eqref{eqn:varphi_x_y} such that $W^{s,\textrm{loc}}_x$
can be identified with
\[
W^{s, \textrm{loc}}_x\simeq  \{(x, y) \,| y \in B_\rho(0) \}
 \]
 with $B_\rho(0)\subset E^s_x$.

All such coordinate systems can be made to have uniform differentiability
properties.
Remember that the foliation  $W^{s,\textrm{loc}}_x$ is invariant (the leaves are mapped into leaves by the
dynamics).
Therefore, in this system of coordinates, the map can be represented as
\begin{equation}\label{map_coordinate}
  f(x,y) = ( a(x), b_x(y) )
\end{equation}
The map $b_x$ represents in coordinates the motion on the leaf $W^{s,\textrm{loc}}_x$,
to the leaf $W^{s, \textrm{loc}}_{a(x)}$, where $x$ and $a(x)$ represent the foot-points of those leaves.

Note that the coordinate patches are different for the domain and the range.
Even if the domain and the range patches overlap,  we do not identify the coordinates.
We do not attempt to identify the points belonging to two coordinate
patches. Note, however, that the coordinate patches will only enter in
the proof to perform some algebraic operations with derivatives
and the patches covered by the coordinate systems may be arbitrarily
small.

The derivative of \eqref{map_coordinate} is, in these coordinates,
given by an upper diagonal matrix.
\[
Df(x,y) = \begin{pmatrix} Da(x) & 0  \\
  \partial_x b_x(y)  &  D b_x(y)
  \end{pmatrix} .
\]
We  can write $\Gamma$ as the graph of a section on $E^s$ in the foliation by the $W^{s,\textrm{loc}}_x$:
\[
\Gamma = \{(x,\sigma(x)), \ x\in \Lambda \} .
\]

Given a point $p = (x, \sigma(x) )$ in the section, we consider
its orbit $f^n(p):=p_n =(x_n,y_n)$.

\smallskip
For future reference, we compute some explicit
and elementary expressions for the  derivatives of iterations of $f$
(similarly to \cite[Proposition 15]{DLS08}).

To simplify notation, we write:
$p_n = f^n( x,y)$ and,
\[ Df(p_n) =  \begin{pmatrix} A_n  & 0 \\
  C_n  &  B_n
\end{pmatrix} .
\]
In the notation for $A_n, B_n,C_n$ we omit
the dependence on the point $p_n$.

We have:
\[
\begin{split}
  Df^n(p_0)  = Df(p_{n-1}) \cdots Df(p_0)  =
  \begin{pmatrix} A_{n-1}  &  0 \\
    C_{n-1} &  B_{n-1}
  \end{pmatrix} \cdots
  \begin{pmatrix} A_{0}  & 0\\
    C_{0} &  B_{0}
    \end{pmatrix} .
  \end{split}
 \]

Following \cite{Hadamard01,Hasselblatt17},
we consider $(\Delta,(D\sigma(x)) \Delta)$ the representation of
a tangent vector of $\Gamma$ at $(x, \sigma(x))$.
When $\Delta$ ranges over all the possible values,
$(\Delta,(D\sigma(x)) \Delta)$ ranges over the graph of $D\sigma(x)$,
the tangent  space of $\Gamma$.

We have that
\[
Df(p) \begin{pmatrix} \Delta \\ d\sigma(p) \Delta
\end{pmatrix}
= \begin{pmatrix}  A_0 \Delta \\
C_0 \Delta   + B_{0} D \sigma(p) \Delta
\end{pmatrix},
\]
\[
Df^2(p) \begin{pmatrix} \Delta \\ d\sigma(p) \Delta
\end{pmatrix}
= \begin{pmatrix} A_{1} A_0 \Delta \\
 (C_1 A_{0}+  B_1 C_0)\Delta+B_1B_0 D \sigma(p) \Delta
\end{pmatrix},
\]
and, for $n \ge 3$:
\[
\begin{split}
Df^n(p) \begin{pmatrix} \Delta \\ D\sigma(p) \Delta
\end{pmatrix}
&= \begin{pmatrix} A_{n-1} \cdots A_0 \Delta \\
 E_n\Delta
     \end{pmatrix}
\end{split}
\]
where
\[\begin{split}
E_n=& \sum_{k=1}^{n-2} B_{n-1} \cdots B_{k+1} C_k A_{k-1} \cdots A_0\\
&+  C_{n-1}  A_{n-2} \cdots A_0
+ B_{n-1} \cdots B_1 C_0
+ B_{n-1} \cdots B_0 D \sigma(p) .
\end{split}\]

As $\Delta$ ranges over the tangent space of $\Gamma$
at $(x,\sigma(x))$, the above  iterated tangent space
can be described as the graph of
the function obtained by writing the
second component as a function of the first component.

That is, the graph of the iterated
tangent space is the graph of the function:
\begin{equation}\label{new_graph}
\begin{split}
& \Big[ \sum_{k=1}^{n-2} B_{n-1} \cdots B_{k+1} C_k (A_{n-1} \cdots A_k)^{-1} \\
&    +C_{n-1}A_{n-1}^{-1}+
  B_{n-1} \cdots B_1 C_0 (A_{n-1} \cdots A_0)^{-1}\Big]\\
&  + B_{n-1} \cdots B_0 D \sigma(p) (A_{n-1} \cdots A_0)^{-1} .
\end{split}
\end{equation}

Now, we proceed to estimate the terms in
\eqref{new_graph}.

By Theorem \ref{localstable},  items II (iv), (v), the last term in \eqref{new_graph}  is straightforwardly  estimated by
  \begin{equation}\label{eqn:last term}
   \|  B_{n-1} \cdots B_0 \, D \sigma(p) \, (A_{n-1} \cdots A_0)^{-1} \|
  \le C  (\lambda_+)^n \| D \sigma\|_{\C^0}   (\mu_-)^n ,
  \end{equation}
where we combine all constants into a new constant which we still denote by $C$ (as we will do in subsequent estimates).

The first term of \eqref{new_graph}
  require  a bit more care.
  We observe that, since $b_x(0) = 0 $ for all $x$
  we have $\partial_x b_x(0) = 0 $.
  Taking into account that $C_k = \partial_x b( y_k) $,
  and that $\|y_k\| \le C\lambda_+^k$,
  since the foliation $\{W^{s,\textrm{loc}}_x\}_x$ is $\C^{1,1}$ and using Schwarz' theorem on mixed partials, we have
  $\| C_k \| \le C\lambda_+^k $.

  Hence
  the  first term  of \eqref{new_graph} can
  be estimated by:
  \begin{equation}\label{eqn:remaining terms}
  \begin{split}
&  \sum_{k=1}^{n-2} C (\lambda_+ )^{n-k-1}  (\lambda_+)^k (\mu_-)^{n-k}
  +C(\lambda_+)^{n-1}\mu_- +C(\lambda_+\mu_-)^{n-1}\\
&  \le
  C (\lambda_+ \mu_-)^n \left(
  \sum_{k=1}^{n-2}  (\lambda_+)^{-1}(\mu_-)^{-k}+(\lambda_+)^{-1}(\mu_-)^{1-n}+(\lambda_+\mu_-)^{-1}\right)\\
&  \le C (\lambda_+ \mu_-)^n ,
  \end{split}
  \end{equation}
where in the last inequality we have used  \eqref{simplicity}.
  Combining \eqref{eqn:last term} and \eqref{eqn:remaining terms} gives the desired result and finishes the proof
  of the case $j=1$ of Lemma~\ref{convergence_rates}.

 Once we have established   the case $j=1$ of Lemma~\ref{convergence_rates},
  the other cases are a corollary.

  For this, we  use the `tangent functor trick' \cite{AbrahamR67,HirschP70}.
   Note that given  $f:M \rightarrow N$,  $g:N\rightarrow P$,
   differentiable maps among manifolds, defining
   $Tf: TM \rightarrow  TN$,
   $Tg: TN \rightarrow  TP $ by $Tf(x,v) = (f(x),Df(x)v)$,
   we have:
   \[
   T(g\circ f) = Tg \circ Tf .
   \]
   Also, if $\Lambda$ is an invariant
   manifold for $f$, $T\Lambda$ is an invariant manifold
   for $Tf$.

  Applying the $j = 1$ result for the $\C^1$-convergence of $T\Gamma$  to
   $T\Lambda$ under $Tf$ iteration, with $\lambda_+$ replaced by $\lambda_+\mu_-$,
   we obtain
 \[d_{\C^1} (Tf^n(T\Gamma), T\Lambda) \le
      C  (\lambda_+  \mu_-^2)^n, \quad n > 0 ,\]
      which implies the $\C^2$ convergence of $\Gamma$ under $f$ iteration
\[d_{\C^2} ( f^n( \Gamma),  \Lambda) \le
      C  (\lambda_+  \mu_-^2)^n, \quad n > 0 .\]
Repeating the tangent functor trick yields the desired conclusion.\qed

\begin{rem} Assume that $\mu_+,\mu_-$ satisfy \eqref{eqn:mubounds} rather than \eqref{simplicity}. Then, we can obtain the same results as in \eqref{Cr-ratess} and \eqref{Cr-ratesu}, for $|n|$ sufficiently large, provided that either $\lambda_\pm$ or $\mu_\pm$ are chosen not to be the optimal rates, that is, \begin{equation}\label{eqn:non_optimal} \lambda^*_\pm<\lambda_\pm \textrm { or }\mu_\pm^*<\mu_\pm.\end{equation}
This is because we can choose $\lambda^*_\pm<\tilde\lambda_\pm<\lambda_\pm$ or $\mu_\pm^*<\tilde\mu_\pm<\mu_\pm$ and run the argument in the proof of Lemma~\ref{convergence_rates} for the rates $\tilde\lambda_\pm$, $\tilde\mu_\pm$, except for the last inequality in  \eqref{eqn:remaining terms}, when we use $n(\tilde\lambda_\pm\tilde\mu_\mp)^n \le C(\lambda_\pm \mu_\mp)^n$ for $|n|$ sufficiently large.
Without the extra condition \eqref{eqn:non_optimal},
instead of \eqref{Cr-ratess} and \eqref{Cr-ratesu} we obtain
we obtain $d_{\C^1} (f^n(\Gamma), \Lambda) \le
      C  n(\lambda_+  \mu_-)^n, \quad n > 0$, and $d_{\C^1} (f^n(\Gamma), \Lambda) \le
      C  n(\lambda_-  \mu_+)^n, \quad n < 0$, respectively.

Similar statements hold for \eqref{Cr-ratessj} and \eqref{Cr-ratesuj}.
\end{rem}

\section{Extensions to other models}
\label{othermodels}

The machinery developed in this paper is rather robust and
produces similar results for other models.

In this Appendix, we discuss two models
(see sections~\ref{partially},~\ref{thermostat})
of physical interest that have appeared in the literature.
We show in Section~\ref{newvanishing} that the
vanishing lemmas apply to these models, and in Section~\ref{newscattering}
that the scattering map also preserves the  corresponding forms.

The methods of this paper apply to these cases,
and we hope they could lead to results that complement the ones presented here.

\subsection{Partially conformally symplectic systems}
\label{partially}

We consider products of symplectic manifolds
\[
(M,\omega) =
(M_1 \times M_2 \times \cdots M_L, \omega_1 \oplus \omega_2 \oplus
\cdots \oplus \omega_L)
\]
where $\omega_i$ is a symplectic form  on $M_i$, $i=1,\ldots, L$.

We consider maps $f$ on $M$ such that
\begin{equation}\label{partialcs}
f^* \omega = \eta_1 \omega_1 \oplus \eta_2 \omega_2
\oplus  \cdots \oplus \eta_L \omega_L,
\end{equation}
for $\eta_i>0$, $i=1,\ldots, L$. Such a map is not conformally symplectic (see Definition \ref{def:conformally_symplectic}).

Systems of the form \eqref{partialcs} appear in mechanics when we consider models of
$L$ particles interacting by a Hamiltonian, where each particle is subject to
  a friction  proportional to its velocity.
See \cite[Remark 3]{CallejaCL13}.

Such models have been explored in the literature for various applications. Below, we provide only a few examples.

\begin{itemize}
\item
  The spin-spin model\footnote{The spin-spin model
is a time dependent flow. The time advance maps are of the form \eqref{partialcs}.} describing two
rotating bodies, each with its own tidal friction
\cite[eq. (42)]{Misquero21}.
\item
Networks of oscillators with Hamiltonian coupling among nearest neighbors, where one node experiences dissipative effects
\cite{EckmannW20,EckmanW18,CuneoEW17,DolgopyatFKY2025}.
These are models of the form \eqref{partialcs}
in which $\eta_i = 1 \textrm{ for  } i \ne 1$, $\eta_1 < 1$.
There are invariant manifolds that affect the transfer of energy
from the conservative modes to the dissipative modes.
\item
Models of planets with a two-layer  structure
subject to viscous and tidal friction, each of them
with different friction coefficients. See \cite{PinzariSV24}.
\end{itemize}

Formally,  several Hamiltonian PDEs
subject to dissipation are of the form \eqref{partialcs}.
For example, consider  the telegraph equation
\[
u_{tt} + u_t - u_{xx} = 0
\]
with periodic boundary conditions.
Writing  $u(t,x) = \sum_k \hat u_k(t) e^{2 \pi \imath k x} $,
we obtain formally
\[
\hat u_k''(t) + \hat u_k'(t) + k^2 \hat u_k (t) = 0,
\]
which is a system of uncoupled
conformally symplectic oscillators with uniform dissipation.
Other models with strong dissipation $-\frac{d}{dt}u_{xx}$ instead of
$u_t $ have also been considered.
Similar considerations apply to other Hamiltonian Partial Differential equations
and   infinite dimensional coupled systems.
For example, \cite{EckmannW20, EckmanW18, CuneoEW17} explore analogies with locally dissipative variants of the nonlinear Schr\"odinger  equation.

Making rigorous sense of geometric properties of NHIMs in PDEs
seems an interesting problem, but it is tractable for finite NHIMs \cite{BatesLZ08}.

One should also note that in infinite dimensional models, there may be dissipation of energy
even in Hamiltonian systems \cite{CorsiG22,ComechKK23}.

\subsection{The Gaussian thermostat}
\label{thermostat}
Several models of non-equilibrium thermodynamics
are based on introducing some forcing as well as some dissipation to
keep energy constant. See \cite{wojtkowski1998conformally}.

These systems have several extra structures,
but they lead to (see \cite[equation (2.4)]{wojtkowski1998conformally})
maps that satisfy:
\begin{equation}\label{isokinetic}
  (f^* \omega)(x) = \eta(x) \omega( x)
\end{equation}
where $\omega$ is a non-closed form (which  allows for $\eta$ to be
not constant), and $\eta(x)$ is a bounded function,
whose inverse is also bounded.

The form $\omega$ and  the factor $\eta(x)$
used in \cite{wojtkowski1998conformally}
satisfy several other properties that lead to other
consequences, but they will not be used below.

Condition \eqref{isokinetic} appears in this paper
in the consideration of presymplectic maps  (see \eqref{confpresymplectic}) with rank smaller than $4$.

\subsection{Vanishing lemmas for the generalized models}
\label{newvanishing}

Both the models in \eqref{partialcs}  and models
of the form \eqref{isokinetic}
are particular cases of  maps that satisfy
\begin{equation}
  \begin{split}
  & |   \omega (x) (u,v)| \le
 C \eta_-^{-n} \|\omega(f^n(x)) \| \| Df^n(x) u\| \| Df^n(x)v\|, \quad n \ge 0,\\
& |  \omega (x)(u,v) | \le C \eta_+^{-n}\|\omega (f^n (x)) \|\| Df^n(x) u\|, \| Df^n(x)v\| \quad n \le 0 .
    \end{split}
\end{equation}
Instead of \eqref{iterates}.
Observe that these inequalities are similar to the ones obtained in the presymplectic case \eqref{newiterates}.

In the case of the models \eqref{partialcs}
we take $\eta_ - = \min(\eta_1,\ldots, \eta_L)$,
$\eta_+ = \max (\eta_1,\ldots, \eta_L)$.

In the case of the model~\eqref{isokinetic}, which corresponds to the presymplectic case studied in \ref{confpresymplectic},
we can take\footnote{
Since $\eta_\pm$ are just bounds, one can use sharper bounds  to improve the conditions of theorems.
Some simple improved bounds used in
\cite{wojtkowski1998conformally} are,  for any $K$:
$\eta_- = \inf_x \left( |\eta(x) \eta(f(x))\cdots \eta(f^K(x))|\right)^{1/(K+1)}$,
$\eta_+ = \sup_x \left( |\eta(x) \eta(f^{-1}(x))\cdots \eta(f^{-L}(x))|\right)^{1/(K+1)}$.
}
$\eta_+ = \sup_x |\eta(x)| $, $\eta_- = \inf_x |\eta(x)|$.

For these models, the proofs of the vanishing lemmas
remain valid under hypotheses that the rates
of vectors  and the numbers $\eta_\pm$ are appropriately related.

For example, if  for $x\in\Lambda$, $u\in T_x M$ and $v\in E^s_x$, and for all $n\ge 0$ we have
\[
\begin{split}
  &\|Df^n(x) u \| \le C  \mu_+^n \|u\|, \\
  & \|Df^n(x) v \| \le C  \lambda_+^n \|v\|, \textrm { with }\\
  &\lambda_+ \mu_+ \eta_-^{-1} < 1 ,
\end{split}
\]
then we conclude , by the vanishing Lemma \ref{lem:vanishingmanifolds}, that  $\omega(u,v) = 0 $.

Observe that this condition on rates also appears  in the statements of Theorem \ref{mainpre}. See \eqref{eqn:rates_conditions_pre}.

In conclusion, the proofs of the vanishing lemmas can be adapted
without  change. Unfortunately,
the fact that it could happen that $\eta_+ \ne \eta_-$ prevents a proof of the
pairing rules.

\subsection{Geometric properties of the scattering map for the generalized models}
\label{newscattering}
Many of the proofs of part (B)
generalize to these cases (as noted before, the statement  that $S$ is symplectic,
when $\omega$ is not a symplectic form should be understood to mean that $S$
preserves $\omega$).

For both  models in Section~\ref{partialcs} and ~\ref{isokinetic}
we can use the proofs of simplecticity of
scattering map in Sections~\ref{sec:proof_main_2_b_L}, \ref{sec:iteration},
using the appropriate assumptions on rates. As remarked in the text,
these proofs do not require that $\omega$ is closed or non-degenerate.
So, they can be used modulo changing the assumptions\footnote{
As it turns out, the proof in Section~\ref{proofnonclosed} can also
be adapted. Taking $d$ of \eqref{isokinetic}, we obtain that
$f^* d \omega = \eta d\omega + d\eta \wedge \omega$.
The geometric setup of \cite{wojtkowski1998conformally}
implies that $ d\eta \wedge \omega = 0$.} on the rates
to include $\eta_\pm$ instead of $\eta$.

For the  models in Section ~\ref{partialcs} $\omega$ is a symplectic form.
Hence we can use without change (except in the assumptions on
 the rates) the proofs
in Sections~\ref{sec:main_2_B_proof_v}, \ref{sec:proof_main_2_b_2},
which  use that $\omega$ is closed.

\bibliographystyle{alpha}
\bibliography{conformal}
\end{document}